\documentclass[11pt,a4paper,twoside,reqno]{amsart}
\usepackage[X2,T2A]{fontenc}
\usepackage[cp1251]{inputenc} 
\usepackage[russian,english]{babel}
\usepackage{mathrsfs}
\usepackage{amsmath,amsthm,amssymb,graphicx,color}
\usepackage[pdfauthor={Yu. Brezhnev}, colorlinks, citecolor=blue,
            pdfwindowui=false,
            bookmarksopen=false,bookmarks=false]{hyperref}

\newcommand{\E}{{\fontencoding{X2}\selectfont\cyryat}}

\usepackage{YuraCutPage}
\textcenter{150mm}{211mm} \obrezka{5mm}{1mm}

\usepackage{YURA,myScripts}
\usepackage{mySymbols,myTOF}
\usepackage[blue]{myFoots}

\theoremstyle{plain}
\newtheorem{theorem}{Theorem}[section]
\newtheorem{lemma}[theorem]{Lemma}
\newtheorem{corollary}[theorem]{Corollary}
\newtheorem{proposition}[theorem]{Proposition}

\theoremstyle{remark}
\newtheorem{remark}[theorem]{Remark}
\def\T{\mbox{\protect\small$\mathrm{T}$}}
\def\ded{{\boldsymbol\eta}}
\def\btau{{\boldsymbol{\tau}}}
\def\G{{\boldsymbol{\mathfrak{G}}}}
\def\B{{\boldsymbol{\mathfrak{B}}}}
\def\bA{\boldsymbol{A}}
\def\bB{\boldsymbol{B}}
\def\bC{\boldsymbol{C}}
\def\bD{\boldsymbol{D}}
\def\calN{\mathcal{N}}
\newcommand{\SUM}[3]{{\sideset{}{_{#1}}\sum_{#2}^{#3}}}

\def\elambda{e_{\!\lambda}}
\def\emu{e_{\!\mu}}
\def\sigmalambda{{\sigma_{\!\!\s\lambda}}}
\def\Aleph{{\boldsymbol{\aleph}}}
\newcommand{\spin}[1]{\langle#1\rangle}
\def\tsint{\raise-0.15ex\hbox{\Larger[2]{\smallint}}\!}

\newcommand{\thinunderbrace}[3][4]
{\Smaller[#1]{\underbrace{\Larger[#1]{\ds#2}}_{\Larger[#1]{\ds#3}}}}

\author[Yu.~Brezhnev]{Yurii V.~Brezhnev}

\title[On Jacobi's $\theta$-functions and $\vartheta$-constants]
{Non-canonical extension of $\theta$-functions\\and
modular integrability of $\vartheta$-constants}

\email{brezhnev@mail.ru}
\date{3rd August 2013}


\thanks{Research supported by
the Federal Targeted Program under contract 14.B37.21.0911}

\begin{document}

\begin{flushright}
\smaller \em Proceedings of the Royal Society of Edinburgh
{\bf A~143}$(4)$ $(2013)$, \textup{689--738}\\[0.5ex]
{\tt http://arXiv.org/abs/1011.1643}
\end{flushright}
\bigskip\bigskip\bigskip

\begin{abstract}
We present new results in the theory of the classical
$\theta$-functions of Jacobi: series expansions and defining ordinary
differential equations (\odes). The proposed dynamical systems turn out
to be Hamiltonian and define fundamental differential properties of
theta-functions; they also yield an exponential quadratic extension of
the canonical $\theta$-series. An integrability condition of these
\odes\ explains appearance of the modular $\vartheta$-constants and
differential properties thereof. General solutions to all the \odes\
are given. For completeness, we also solve the Weierstrassian elliptic
modular inversion problem and consider its consequences. As a
nontrivial application, we apply proposed techni\-que to the Hitchin
case of the sixth Painlev\'e equation.

\end{abstract}

\maketitle\thispagestyle{empty}

\newpage
\tableofcontents

\section{Introduction}

\noindent The theta-functions of Jacobi and the Weierstrassian basis of
functions $\{\sigma$, $\zeta$, $\wp$, $\wp'\}$ arise in numerous
theories and applications. Since their discovery in the late 1820's,
this field became the subject of intensive study. The majority of
results and current form of the theory were obtained in the very works
of Jacobi and Weierstrass, and their contemporaries: Hermite
\cite{hermite}, Enneper \cite{enneper}, Kiepert,  Neumann
\cite{neumann}, Halphen \cite{halphen}, Hurwitz \cite{hurwitz},
Frobenius \cite{frobenius}, Fricke \cite{fricke,klein} et al.
Presently,  the current monographic literature on this topic may run
into tens of items. Thorough treatises (not only in Deutsche/Latin
\cite{somov}) had appeared as far back as Jacobi's lifetime; see, for
example, \cite{verhulst}, \cite{gudermann}\footnote{What is now called
an elliptic function was named a `Modular-Function' by Gudermann in
\cite{gudermann}.} and, continuing the list further, we should mention
treatises \cite{tikh,cayley,durege,forsyth,hancock,koenig,
krause,riemann,thomae,weber,briot},  references in these works, and
especially encyclopedic paper by Fricke \cite{fricke2}. Fricke had also
planned a 3rd volume to the series \cite{fricke} but it was composed
only recently \cite{fricke3}. See also the excellent history survey by
Koenigsberger \cite{koenig2}. Literature given in the reference list
does not claim on completeness, but comprehensively covers the known
properties of elliptic, modular, and $\theta$-functions  known
nowadays. In addition to monographs listed above it should be mentioned
the later presentation of the theory
\cite{rauch,a,lawden,farkas,rankin,hurwitz2,mumford,apostol,123,
armitage} and handbook literature as well
\cite{we2,abramowitz,bateman}. A great number of specific examples can
be found  in the classical \emph{`A course of modern analysis'} by
Whittaker \& Watson \cite{WW} and  very detailed exposition of the
theory is in Weber's \emph{`Lehrbuch der Algebra'} \cite{weber} and in
the two-volume Halphen treatise with the posthumous issue of the 3rd
volume \cite{halphen}. As  a source for the formulae, in most of cases
Schwarz's collection of Weierstrass' and Jacobi's classical results
\cite{we2} is by no means lacking and the four volume set by Tannery \&
Molk \cite{tannery} hitherto contains most exhaustive information along
these lines. The very Weierstrass' \cite{we} and Jacobi's \cite{jacobi}
works, being very detailed in presentation, should be referred hereinto
since they have still remained a source of important observations.

\subsection{The paper content and comments on results}

In the present work we describe some new properties of Jacobi's
$\theta$-series, that, to
the best of our knowledge, 
have not appeared in the extensive 
literature on this topic. Among these are series expansions,
differential equations, and their consequences. A characteristic
property of the Jacobi--Weierstrass theory is explicit analytic
$\theta,\wp$-formulae for solutions of applied problems. In this
connection we shall  exhibit  examples: modular inversion problem,
differential computations of Weierstrassian functions, and applications
to the famous sixth Painlev\'e transcendent (Sect.~\ref{S10}). Other
applications recently announced in \cite{br1,br3,br4}.

\subsubsection{The power series}

The series expansions of elliptic, modular, and $\theta$-functions are
an ever-operating instrument in many problems up to now.  This is
because coefficients of the series have nice analytic and combinatoric
properties. Suffice it to mention applications of the function series
for various $\theta$-quotients \cite{weber,tannery}, number-theoretic
$q$-series, series of Lambert \cite{apostol}, the famous
McKay--Thompson series and their corollaries like `Moonshine
Conjecture' and its modern extensions. Trigonometric series for
Jacobi's $\theta$-functions are put for their definition (see
Sect.~\ref{S2.1}) and  power series for Weierstrass' $\sigma$-function
is very well known \cite{we2, abramowitz}. It is frequently reproduced
in the literature and has multidimensional generalizations
\cite{eilbeck}. Considering this, it is somewhat surprise facet is the
fact that the power series expansions for $\theta$-functions, \ie,
$\theta$-analogs of Weierstrassian $\sigma$-series, are absent
heretofore. It is of interest to remark that even Jacobi
attempted\footnote{In a 1828 letter to Crelle \cite[{\bf
I}:~pp.~259--260]{jacobi}, \ie, before appearance of the
\emph{Fun\-da\-men\-ta Nova}.} to obtain that series and observed that
their coefficients resulted in interesting dynamical systems. In
Sect.~\ref{S3} we construct the canonical power expansions for
$\theta$-functions in a (fast computable) form that has the simplest
structure and is maximally effective from the analytic point of view.

\subsubsection{Differential equations}

In sects.~\ref{S4}--\ref{S7} and \ref{S9} we expound the main material
of the work. Namely, dynamical systems satisfied by $\theta$-series,
their integrability condition, non-canonical extension of the canonical
$\theta$-series, and modular integrability of $\vartheta$-constants. We
shall see that not only do elliptic functions satisfy ordinary
differential equations (\odes), but $\theta$-functions themselves also
satisfy certain \odes. These \odes\ are of interest in their own rights
if only because, it is more logical to consider \emph{\odes\ proper for
the $\theta$-functions as basic equations} (Sect.~\ref{S4}) rather than
\odes\ for elliptic functions. Moreover, such a viewpoint is natural in
a more general pattern since elliptic functions are a subclass of
Abelian elliptic integrals and the latter are expressible in terms of
theta-functions. In Sect.~\ref{S9} we shall also see that solutions of
introduced \odes\ have a remarkable consequence, namely, an exponential
quadratic extension with additional parameters. Under certain
parameters this non-canonical extension coincides with the case of
canonical $\theta$-series.  There is yet another point worth noting.
Differential relations between $\theta$-functions are sometimes present
in old literature \cite{baruch,weber,tannery} but they are regarded
there just as differential identities. However, the principal point is
a \emph{differential closedness of the finitely many basic
$\theta$-objects} and, together with $\theta$-functions, the
theta-derivative $\Dtheta$ should play independent part in the theory.
Upon introducing this object, analytic and differential manipulations
by theta-functions are not in essence distinct from that by elliptic
ones or even elementary functions like sine-, cosine-functions.

\subsubsection{Algebraic integrability and Hamiltonicity}

One further remarkable property of the systems mentioned above is the
fact that the known basic polynomial theta-identities between canonical
$\theta$-series are nothing but the specific values of algebraic
integrals of the systems (Sect.~\ref{S7.1}). Clearly, we may take these
integrals as the Hamilton functions for these \odes\  and the
equations, being regarded as dynamical systems, thus possess the
remarkable property of being Hamiltonian. In Sect~\ref{S9.2} we shall
exemplify a particular case along these lines. Such treatments have a
large number of applications and intimate connection with the theory of
integrable systems \cite{belokolos}. A complete Hamiltonian
description, along with quantization, will be, however, the subject
matter of a separate work but the theta-constant case is discussed at
greater length in \cite{br3}. In this work we underlined (with some
counterexamples) the need to distinguish differential identities and
defining \odes. Algebraic integrability of the '$\theta$-\odes' is also
used in \cite{br4} for a new treatment of the finite-gap spectral
problems.

\subsubsection{Modularity and integrability conditions}

Yet another consequence of the `differential viewpoint' is an automatic
appearance of the `modular objects' (Sect.~\ref{S6}). Differential
properties of the `modular part' of Jacobi's functions are more
transcendental and closely related to the classical theory of linear
\odes\ with infinite groups of Fuchsian monodromies. Moreover, in the
early 1990's M.~Ablowitz \& S.~Chakravarty with coauthors
\cite{ablowitz0,ablowitz,chakr,tah} observed the deep linkage between
this theory and complete integrable equations. In the last two decades,
in works of Harnad \cite{harnad}, McKay \cite{harnad2}, Ablowitz et all
\cite[pp.~573--589]{conte}, \cite{ablowitz2}, Ohyama \cite{ohyama},
Hitchin \cite{hitchin}, this field was substantially advanced and found
nice applications known as monopoles \cite{hitchin2},
Chazy--Picard--Fuch's equations \cite{tah}, cosmological metrics of Tod
\cite{tod}, Hitchin \cite{hitchin} etc. In Sect.~\ref{S7.1} we shall
give further explanations and see that `modular integrability' of the
classical $\vartheta$-constants has the following characterization. It
constitutes a compatibility condition of the linear heat equation
$4\pi\,\ri\,\theta_\tau=\theta_{\mathit{zz}}$ and, on the other hand,
the quadrature integrable nonlinear `$z$-equations' for the functions
$\theta(z|\tau)$ and $\theta'(z|\tau)$.

It may be also mentioned here that proposed `differential technique'
may be applied equally well to the cases of multidimensional
$\Theta$-functions when the latter split into a decomposition of the
1-dimensional, \ie, Jacobi's $\theta$-functions. It is known that such
cases correspond to jacobians of algebraic curves covering the elliptic
tori \cite{belokolos}. The problems, wherein such a kind curves arise,
are very nontrivial in the one part, and are completely solvable as the
pure elliptic case in the other part.

\subsubsection{Modular inversion problem}

Complete form of the `differential theory'  requires a closed form
solution to the Weierstrassian elliptic modular inversion problem.
Strange though it may seem, but its formula realization does not appear
in the literature and, in Sect.~\ref{S8}, we shall give that solution
and exhibit some consequences and applications thereof.

Content of the next section is a matter of common knowledge and we
present it here to fix notation and terminology. In the subsequent
sections we do not touch upon the known results and simply use them
(accompanied with references) where they are required.

\section{Definitions and notation}

\subsection{The Jacobi functions\label{S2.1}}

The four theta-functions $\theta_{1,2,3,4}$  and their
$\theta_{\alpha\beta}$-equivalents are defined by the  canonical series
\begin{alignat*}{3}
-\theta_{11}(z|\tau)\equiv \theta_1(z|\tau)&={}& -\ri\,
\re_{\strut}^{\s{\frac14}\pi\ri\tau}
&\sideset{}{_k}\sum\limits_{-\infty}^{+\infty}\! (-1)^k\,
\re^{(k^2+k)\pi\ri\,\tau}_{\mathstrut}\, \re^{(2k+1)\pi \ri\,
z}_{\mathstrut}\\
&=&2\,\re_{\strut}^{\s{\frac14}\pi\ri\tau}&
\sideset{}{_k}\sum_{0}^\infty
(-1)^k\re_{\strut}^{(k^2+k)\pi\ri\tau}\sin(2k+1)\pi\,z\,,\\
\theta_{10}(z|\tau)\equiv \theta_2(z|\tau)&=&
\re_{\strut}^{\s{\frac14}\pi\ri\tau}
&\sideset{}{_k}\sum\limits_{-\infty}^{+\infty}\!
\re^{(k^2+k)\pi\ri\,\tau}_{\mathstrut}\, \re^{(2k+1)\pi
\ri\, z}_{\mathstrut}\\
&=&2\,\re_{\strut}^{\s{\frac14}\pi\ri\tau}&
\sideset{}{_k}\sum_{0}^\infty
\re_{\strut}^{(k^2+k)\pi\ri\tau}\cos(2k+1)\pi\,z\,,\\
\theta_{00}(z|\tau)\equiv \theta_3(z|\tau)&=&
&\sideset{}{_k}\sum\limits_{-\infty}^{+\infty}
\!\re^{k^2\pi\ri\,\tau}\,\re^{2k\pi\ri\,z}\\
&=&1+ 2& \sideset{}{_k}\sum_{1}^\infty
\re_{\strut}^{k^2\pi\ri\tau}\cos2k\pi\,z\,,\\
\theta_{01}(z|\tau)\equiv \theta_4(z|\tau)&=&
&\sideset{}{_k}\sum\limits_{-\infty}^{+\infty}\! (-1)^k\,
\re^{k^2\pi\ri\,\tau}\, \re^{2k\pi \ri\, z}\\
&=&1+ 2& \sideset{}{_k}\sum_{1}^\infty
(-1)^k\re_{\strut}^{k^2\pi\ri\tau}\cos2k\pi\,z\,,
\end{alignat*}
where, to avoid further confusion with Weierstrass' branch points
$e$'s, we use notation $\re^z$ for exponent of $z$. We shall use also
the shorthand form $\theta_k\DEF \theta_k(z|\tau)$. Values of
\mbox{$\theta$-functions} under $z=0$ (Thetanullwerthe \cite{jacobi})
are called the theta-constants. Because each of the variables $z$ and
$\tau$ has an `independent theory' we introduce a separate notation for
the nullwerthe: $\vartheta_k\DEF\vartheta_k(\tau)=\theta_k(0|\tau)$.
Their series representations follow from the $\theta$-series above:
\begin{equation}\label{consts}
\vartheta_2(\tau)\DEF
\sideset{}{_k}\sum_{-\infty}^\infty\!
\re^{\left(k+\s{\frac12}\right)^2\pi\ri\tau}_{\mathstrut},\qquad
\vartheta_3(\tau)\DEF
\sideset{}{_k}\sum_{-\infty}^\infty\!
\re^{k^2\pi\ri\tau}_{\mathstrut},\qquad
\vartheta_4(\tau)\DEF
\sideset{}{_k}\sum_{-\infty}^\infty\!
(-1)^k\re^{k^2\pi\ri\tau}_{\mathstrut}\,.
\end{equation}
Convergency of the series implies that $\tau$ belongs to the upper
half-plane: $\tau\in\Hp$, that is $\boldsymbol\Im(\tau)>0$. For
typographical convenience we adopt two pieces of notation for
$\theta$-functions with characteristics $\AB{\alpha}{\beta}$ (Hermite
(1858)):
\begin{equation}\label{hermite}
\theta\AB{\alpha}{\beta}(z|\tau)\equiv \theta_{\alpha\beta}(z|\tau)=
\sideset{}{_k}\sum\limits_{-\infty}^{+\infty}\! \re^{\pi\ri
\left(\!k+\frac\alpha2\!\right)^{\!2}\tau+
2\pi\ri\left(\!k+\frac\alpha2\!\right)\!
\left(\!z+\frac{\smash{\beta}}2\!\right)\! }_{\mathstrut}\,.
\end{equation}
Let $n$, $m$ be arbitrary integers: $n,m=0,\pm1,\pm2,\ldots\,$. We
consider only integral characteristics and hence, by virtue of formula
\begin{equation}\label{*}
\theta\AB{\alpha+2m}{\beta+2n}=(-1)^{\alpha\,
n}\cdot\theta\AB{\alpha}{\beta}\,,
\end{equation}
functions $\theta_{\alpha\beta}$ always reduce to
$\pm\theta_{1,2,3,4}$. When adding a half-period,
$\theta$-characteristics undergo a shift:
\begin{equation}\label{shift}
\theta\AB{\alpha}{\beta} \!\!\left({z+\Mfrac{n}{2}+\Mfrac{m}{2}\,
\tau}\Big|\tau\right)=(-\ri)^{(\beta+n)m}\:
\theta\AB{\alpha+m}{\beta+n}(z|\tau) \!
\cdot\re^{\sm\s{\frac14}\pi\ri
\,m(4z+m\tau)}_{\mathstrut}\,.
\end{equation}
Two-fold shifts by half-periods yield the law of transformation of
$\theta$-function into itself:
$$
\theta_{\alpha\beta} (z+n+m\,\tau|\tau)=
(-1)^{n\alpha-m\beta}\,\theta_{\alpha\beta} (z|\tau)\cdot
\re^{\sm\pi\ri\,m(2z+m\tau)}\,.
$$
Value of any $\theta$-function at any half-period is a certain
$\vartheta$-constant multiplied by the exponential factor:
$$
\theta\AB{\alpha}{\beta} \!\!\left({\Mfrac{n}{2}+\Mfrac
m2\,\tau}\Big|\tau\right)=(-\ri)^{(\beta+n)m}\:
\vartheta\AB{\alpha+m}{{\beta+n}}(\tau)
\cdot\re^{\sm\s{\frac14}\pi\,\ri\,m^2\tau}_{\mathstrut}\,.
$$
In the present work we use the `$\tau$-representation' for
$\vartheta,\theta$-functions. Transition to frequently used
`$q$-representation' $\big(q=\re^{\pi\ri\tau}_{}\big)$ is performed by
the formula $\partial_q=\pi\,\ri\,q\,\partial_\tau$.

We supplement the set of functions $\{\theta_k\}$ with the derivative
$\partial_z\theta_1(z|\tau)$ and consider it as a fifth independent
object:
\begin{equation}\label{der}
\begin{alignedat}{3}
\partial_z\theta_1(z|\tau)\FED\Dtheta(z|\tau)&={}&
\pi\, \re_{\strut}^{\s{\frac14}\pi\ri\tau}
&\sideset{}{_k}\sum\limits_{-\infty}^{+\infty}\! (-1)^k\,(2\,k+1)\,
\re^{(k^2+k)\pi\ri\,\tau}_{\mathstrut}\, \re^{(2k+1)\pi
\ri\, z}_{\mathstrut}\\
&=&2\,\pi\,\re_{\strut}^{\s{\frac14}\pi\ri\tau}&
\sideset{}{_k}\sum_{0}^\infty
\,(-1)^k(2\,k+1)\,\re_{\strut}^{(k^2+k)\pi\ri\tau}\cos(2k+1)\pi\,z\,.
\end{alignedat}
\end{equation}
It is interesting to note that the very object and its $z$-dependence
have an intriguing  correlation with experimental data in purely
physical considerations \cite{scott}.

\subsection{The Weierstrass functions\label{S2.2}}

We use the conventional Weierstrassian notation \cite{we2}
$$
\begin{aligned}
\sigma(z|\om,\om')&=\sigma(z;\g2,\g3)\,,\qquad&
\zeta(z|\om,\om')&=\zeta(z;\g2,\g3)\,,\\
\wp(z|\om,\om')&=\wp(z;\g2,\g3)\,,&
\wp'(z|\om,\om')&=\wp'(z;\g2,\g3)\,.
\end{aligned}
$$
Invariants $(\g2,\,\g3)$ are functions of periods $(2\om,2\om')$ (and
vice versa) and modulus \mbox{$\tau=\om'\!/\om$}. They are defined by
the well-known Weierstrass--Eisenstein series \cite{we,eisenstein,weil}
which are, however, entirely unsuited for numeric computations.
Hurwitz, in his dissertation \cite[p.~547]{hurwitz}, found a nice
transition to the Lambert series \cite{apostol}
\begin{equation}\label{star}
\begin{split}
\g2(\tau) &=20\,\pi^4\,\mbig[9]\{\frac{1}{240} +
\sideset{}{_k}\sum_{1}^\infty
\frac{k^3\,\re^{2k\pi\ri\,\tau}}{1-\re^{2k\pi\ri\,\tau}}
\mbig[9]\},\\
\g3(\tau) &= \;\frac73\,\pi^6\, \mbig[9]\{\frac{1}{504} -
\sideset{}{_k}\sum_{1}^\infty
\frac{k^5\,\re^{2k\pi\ri\,\tau}}{1-\re^{2k\pi\ri\,\tau}} \mbig[9]\}\,,
\end{split}
\end{equation}
which are used in theories, have applications, and are most effective
in computations.

Determination of periods $(2\om,2\om')$ by coefficients  $(a,b)$ of
elliptic curve in Weierstrassian form $w^2=4\,z^3-a\,z-b$ is know as
the elliptic modular inversion problem. Its solution involves the
transcendental equation $J(\tau)=A$, where $J(\tau)$ is the classical
modular function of Klein \cite{klein,hurwitz2,a}. Modular inversion is
then realized by the scheme
\begin{equation}\label{Klein}
(a,b) \quad\dashrightarrow\quad
J(\tau)=\frac{a^3}{a^3-27\,b^2} \quad\dashrightarrow\quad\om=\pm
\sqrt{\frac ab\frac{\g3(\tau)}{\g2(\tau)}}\quad
\dashrightarrow\quad\om'=\tau\,\om\,.
\end{equation}
The degenerated cases---lemniscatic $(b=0)$ and equi-anharmonic $(a=0)$
ones---require  separate formulae. In both of these cases there exist
exact solutions. The lemniscatic solution $\om_{\s{\textsc{l}}}$ was
found by Gauss. In our notation it is as follows
$$
\om_{\s{\textsc{l}}}=
\!\!\sqrt[\sm4]{8\,a\,}\,\pi\cdot\vartheta_4^2(2\,\ri)\,,\qquad
\om'=\ri\,\om_{\s{\textsc{l}}}\,.
$$
See works by Todd \cite{todd} and Levin \cite{levin} for exhaustive
information and voluminous bibliography on the lemniscate. Exact
solution $\om_{\s{\textsc{e}}}$ to the equi-anharmonic case we display
here seems to be new:
$$
\om_{\s{\textsc{e}}}=\!\!
\sqrt[\leftroot{-2}\sm12]{\!-27\,b^2\,}\,\pi\cdot\ded^2(\varrho)\,,
\qquad \om'=\varrho\,\om_{\s{\textsc{e}}}\,,
\quad\varrho\DEF-\frac{1}{2}\big(1-\ri\,\sqrt{3}\big)\,,
$$
where $\ded$ is the Dedekind function (see Sect.~\ref{S2.3} for
definition). The arbitrary branches of the $\surd$-roots are allowed in
the previous formulae.

By virtue of homogeneity relations, say $\alpha^2\wp(\alpha\,
z|\alpha\,\om,\alpha\,\om')=\wp(z|\om,\om')$, the couple of
half-periods $(\om,\om')$ or invariants $(\g2,\g3)$ can be replaced by
one quantity, \ie, modulus $\tau=\om'\!/\om$. We denote corresponding
functions as follows:
$$
\begin{aligned}
\sigma(z|\tau)&\DEF\sigma(z|1,\tau),\qquad&
\zeta(z|\tau)&\DEF\zeta(z|1,\tau),\\
\wp(z|\tau)&\DEF\wp(z|1,\tau),&\wp'(z|\tau)&\DEF\wp'(z|1,\tau)\,.
\end{aligned}
$$
Weierstrassian  $\eta$-function is defined by the formula
$\eta(\tau)\DEF\zeta(1|1,\tau)$ and its series representation reads
\begin{equation}\label{eta}
\eta(\tau)=2\,\pi^2\,\mbig[9]\{\frac{1}{24}- \sideset{}{_k}
\sum_{1}^\infty\frac{\re^{2k\pi\ri\,\tau}}
{(1-\re^{2k\pi\ri\,\tau})^2}\mbig[9]\}\,.
\end{equation}

Modular transformations in the elliptic/modular theory are of not only
theoretical interest since  value of the modulus $\tau$ strongly
affects convergence of the series. Moving $\tau$ into  fundamental
domain of the modular group
$$
\tau\mapsto\frac{a\,\tau+b}{c\,\tau+d}\,,\qquad
\Big(\,\begin{matrix}a&\!\!b\\c&\!\!d
\end{matrix}\,
\Big)\in\mathrm{PSL}_2(\mathbb{Z})\FED\boldsymbol{\Gamma}(1)
$$
(this process  is easily algorithmizable), one obtains values of $\tau$
having the minimal imaginary part
$\boldsymbol\Im(\tau)=\frac12\sqrt{3}$. In such the `worst' point all
the series converge very fast. For example modular property of the
$\eta$-series is as follows
$$
\eta\Big(\Mfrac{a\,\tau+b}{c\,\tau+d}\Big)
=(c\tau+d)^2\,\eta(\tau)-\frac{\ri}{2}\,\pi\,c\,(c\,\tau+d)\,,
$$
where $(a,b,c,d)$ are integers and, as usual, $a\,d-b\,c=1$.

The three Weierstrassian $\sigma$-functions are defined  through
Jacobian functions by the  expressions \cite{tannery,we2}
$$
\sigmalambda\!(z|\om,\om')= \frac{\theta_{\s{\lambda{+}1}}\!
\big(\frac{z}{2\om}\big|\frac{\om{\s'}}{\om}\big)}
{\vartheta_{\s{\lambda{+}1}}\! \big(\frac{\om{\s'}}{\om}\big)}\,
\re^{\eta(\om,\om')\,\frac{z^2}{2\om}}_{\mathstrut}\,,\qquad
\lambda=1,2,3\,.
$$
The Weierstrass $\sigma$-function, as function of $(z,\,\g2,\,\g3)$,
satisfies linear differential equations obtained by Weierstrass
himself:
\begin{equation*}
\begin{aligned}
z\,\frac{\partial\sigma}{\partial z}-
4\,\g2\,\frac{\partial\sigma}{\partial \g2}-
6\,\g3\,\frac{\partial\sigma}{\partial \g3}-\sigma=0\,, \\
\frac{\partial^2\sigma}{\partial z^2}-
12\,\g3\,\frac{\partial\sigma}{\partial \g2}-
\frac23\,\g2{}^{\hspace{-0.4em}2}\,\frac{\partial\sigma}{\partial \g3}
+\frac{1}{12}\,\g2\,z^2\,\sigma=0\,.
\end{aligned}
\end{equation*}
It immediately follows that there exists a recursive relation for
coefficients $C_k$ of the power series
\begin{equation}\label{sigma}
\sigma(z;\g2,\g3)= C_0\,z+C_1\,\frac{z^3}{3!}+\cdots=
z-\frac{\g2}{240}\,z^5-\frac{\g3}{840}\,z^7 +\cdots\,,
\end{equation}
where the standard normalization $\sigma(0)=0$, $\sigma'(0)=1$,
$\sigma''(0)=0$ has been adopted. The two classical recurrences are
known. The first one is due to Halphen \cite[{\bf I}:~p.~300]{halphen}:
\begin{equation}\label{D}
\mathfrak{}
C_k= \Larger{\widehat{\Smaller{\boldsymbol{\mathfrak D}}}}\,
C_{k\sm1}-\frac16\,(k-1)(2\,k-1)\,\g2\,C_{k\sm2}\,,
\end{equation}
where
$$
\Larger{\widehat{\Smaller{\boldsymbol{\mathfrak D}}}}
\DEF 12\,\g3\,\frac{\partial}{\partial \g2}+
\frac23\,\g2{}^{\hspace{-0.4em}2}\,\frac{\partial}{\partial \g3}\,,
$$
but in different notation it was written down by Weierstrass \cite[{\bf
V}:~p.~49]{we} and even by Jacobi (see Sect.~\ref{S5}). The second
recurrence was obtained by Weierstrass:
\begin{equation}\label{sigmaW}
\sigma(z;\g2,\g3)=\mbox{\large$\ds\sum_{m,\,n=0}^\infty$}
\,A_{m,\,n} \left(\!{\frac{\g2}{2}}\!\right)^{\!m} \!\big(
2\g3\big)^n {\frac{z^{4m+6n+1}} {(4m{+}6n{+}1)!}}\,,
\end{equation}
\begin{equation}\label{Wmn}
\begin{aligned}
A_{m,\,n}&= \frac{16}{3}\,(n+1)\,A_{m-2,\,n+1}
+3\,(m+1)\,A_{m+1,\,n-1}\\
&\==-\frac13\,(2m+3n-1)(4m+6\,n-1)\, A_{m-1,\,n}\,,
\end{aligned}
\end{equation}
where $A_{0,0}=1$ and $A_{m,n}=0$ under $n,m<0$. Other recurrences are
also known \cite{eilbeck}. Among all the recurrences the Weierstrassian
one is least expendable because it contains only multiplication of
integers and  coefficients $C_k$ have already been collected in
parameters. It is interesting to remark that Weierstrass proves
separately the fact that $A_{m,n}$ have integral values \cite[{\bf
V}:~p.~50]{we}. We shall be guided by the same motivation when deriving
the power series for Jacobi's $\theta$-functions in Sect.~\ref{S3}.

\subsection{Dedekind's function\label{S2.3}}

Since the standard notations for Weierstrassian  $\eta$-function and
Dedekind's one coincide, we  use for the latter the symbol
$\ded(\tau)$:
$$
\ded(\tau)= \re_{\strut}^{\frac{\pi\ri}{12}\,\tau}\,
{\sideset{}{_k}\prod\limits_{1}^\infty}
\big(1-\re^{2k\pi\ri\,\tau}\big)=
\re_{\strut}^{\frac{\pi\ri}{12}\,\tau}
\sideset{}{_k}\sum\limits_{-\infty}^{+\infty}
\!(-1)^k\,\re^{(3k^2+k)\pi\ri\,\tau}\qquad \big(\mbox{Euler
(1748)}\big)\,.
$$
Dedekind's function is connected with the Jacobi--Weierstrass ones
through the differential and algebraic relations \cite{rankin,weber}:
\begin{equation}\label{ded}
\frac{1}{\ded}\,\frac{d\ded}{d\tau}=\frac{\ri}{\pi}\,\eta\,,\qquad
2\,\ded^3=\vartheta_2\,\vartheta_3\,\vartheta_4\,.
\end{equation}

\section{Canonical power $\theta$-series\label{S3}}

\noindent
Before proceeding to the  $\theta$-series we need some
preparatory material on series for Weierstrassian $\sigma$'s and
graphical illustration to the recurrence $A_{m,n}$; see
fig.~\ref{fig1}.
\begin{figure}[htbp]\label{Amn}
\centering
\includegraphics[width=6 cm]{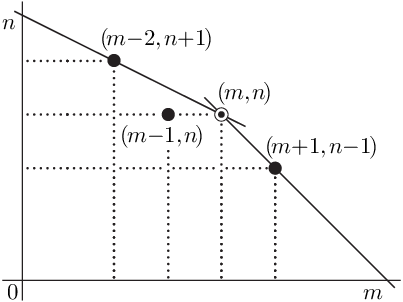}
\caption{\label{fig1}The Weierstrass recurrence \eqref{Wmn}.}
\end{figure}
This figure and formulas \eqref{sigmaW}--\eqref{Wmn}  mean that
computation of the point $(m,n)$ involves computations of all the
interior points of delineated quadrangle.

Throughout the paper we use the symbol $[n]$ for integer part of the
number $n$. Then one can easily show that Weierstrassian series
\eqref{sigmaW} simplifies into the following expression explicitly
collected in variable $z$:
\begin{equation}\label{wei}
\sigma(z;\g2,\g3)= \mbox{\large$\ds\sideset{}{_k}
\sum\limits_{0}^\infty$}\,
\Bigg\{ \sideset{}{_\nu}\sum_{\s{[k/3]}}^{\s{k/2}}
\!2^{2k\sm5\nu} A_{3\nu\sm k,\,k\sm2\nu}\,
\cdot \g2{}^{\hspace{-0.4em}3\nu\sm k}\,
\g3{}^{\hspace{-0.4em}k\sm2\nu}\Bigg\} \frac{z^{2k+1}}{(2k{+}1)!}\ .
\end{equation}
Analogous series exist for all the $\sigma$-functions but their form
depends on which parameters $\g{2,3}$ or $\elambda$ are chosen as the
basic ones (see Remark~\ref{R1} further below).

\subsection{Halphen's operator and  series for $\sigmalambda$}

Let us denote $\elambda\DEF \wp(\om_{\!\s\lambda}^{}|\om,\om')$.

\begin{lemma}\label{L1}
Power series for functions $\sigmalambda$ are given by the expression
$$
\sigmalambda\!(z;\elambda,\g2)=
\mbox{\large$\ds\sideset{}{_k}\sum_{0}^\infty$}\, \Bigg\{
\sideset{}{_\nu}\sum_{\s
0}^{\s k/2}\,2^{\sm\nu}\,\B_{k\sm2\nu,\,\nu}\cdot
\pow{e}{\s\lambda}{k\sm2\nu}\g2{}^{\hspace{-0.4em}\nu} \Bigg\}\,
\frac{z^{2k}}{(2k)!}
$$
with the  integral recurrence\/$:$
\begin{equation*}
\begin{aligned}
\B_{m,n}&=24\,(n+1)\,\B_{m-3,\,n+1}+
(4\,m-12\,n-5)\,\B_{m-1,\,n}\\
&\quad-\frac43\,(m+1)\,\B_{m+1,\,n-1}-
\frac13\,(m+2\,n-1)(2\,m+4\,n-3)\,\B_{m,\,n-1}\,, \\
\B_{\s0,\s0}&=1\quad\text{and}\quad\B_{m,\,n}=0 \text{\ \ if\ \ }
m,n<0\,.
\end{aligned}
\end{equation*}
\end{lemma}

\begin{proof}
Calculations are based on Halphen's equations satisfied by
$\sigmalambda$-functions \cite{halphen}:
\begin{equation}\label{H}
\begin{aligned}
z\,\frac{\partial\sigmalambda}{\partial z}-
2\,\elambda\,\frac{\partial\sigmalambda}{\partial\elambda}-
4\,\g2\,\frac{\partial\sigmalambda}{\partial
\g2}&=0\,,\\
\frac{\partial^2\sigmalambda}{\partial z^2}-
\mbig[7](4\,\pow{e}{\s\lambda}{2}
-\Mfrac23\,\g2\mbig[7])\!
\frac{\partial\sigmalambda}{\partial \elambda}-
12\,\big(4\,\pow{e}{\s\lambda}{3}-\g2\,\elambda\big)
\frac{\partial\sigmalambda}{\partial \g2}+
\mbig[7](\elambda+\Mfrac{1}{12}\,\g2\,z^2\mbig[7])\, \sigmalambda
&=0
\end{aligned}
\end{equation}
and proof of integrality of the recurrence $\B_{m,\,n}$ is analogous to
Weierstrassian argument on p.~50 in \cite[{\bf V}]{we}.
\end{proof}

We can also include into this recurrence the $\sigma$-function. In this
case  the analog of Halphen's equation \eqref{H} takes the form
\begin{equation}\label{epsilon}
\begin{aligned}
z\,\frac{\partial\Xi}{\partial z}-
2\,\elambda\,\frac{\partial\Xi}{\partial\elambda}-
4\,\g2\,\frac{\partial\Xi}{\partial \g2}
-(1-\varepsilon)\,\Xi&=0\,,\\
\frac{\partial^2\Xi}{\partial z^2}-
\mbig[7](4\,\pow{e}{\s\lambda}{2}-\Mfrac23\,\g2\mbig[7])\!
\frac{\partial\Xi}{\partial \elambda}-
12\,\big(4\,\pow{e}{\s\lambda}{3}-\g2\,\elambda\big)
\frac{\partial\Xi}{\partial \g2}+
\mbig[7](\varepsilon\,\elambda+\Mfrac{1}{12}\,\g2\,z^2\mbig[7])
\,\Xi&=0\,,
\end{aligned}
\end{equation}
where case $\Xi=\sigmalambda$ corresponds to $\varepsilon=1$ and
$\Xi=\sigma$ corresponds to $\varepsilon=0$  with arbitrary $\elambda$.
This is because the $\sigma$-function does not depends on permutation
of branch-points $\elambda$.

\begin{corollary}\label{C1}
Power series for all the Weierstrass $\sigma$-functions are determined
by the expression
\begin{equation}\label{universal}
\Xi(z;\elambda,\g2)=\mbox{\large$\ds
\sideset{}{_k}\sum_{0}^{\infty}$}\,
\Bigg\{
\sideset{}{_\nu}\sum_{\s0}^{\s k/2}\,2^{\sm\nu}\,
\B_{k\sm2\nu,\,\nu}^{{\s(}\varepsilon{\s)}}\cdot
\pow{e}{\s\lambda}{k\sm2\nu}\g2{}^{\hspace{-0.4em}\nu}
\Bigg\}\,\frac{z^{2k+1-\varepsilon}}
{(2k{+}1{-}\varepsilon)!}
\end{equation}
under the universal integral recurrence
\begin{equation*}
\begin{aligned}
\B_{m,n}^{{\s(}\varepsilon{\s)}}&=
24\,(n+1)\,\B_{m-3,\,n+1}^{{\s(}\varepsilon{\s)}}+
(4\,m-12\,n-4-\varepsilon)\,\B_{m-1,\,n}^
{{\s(}\varepsilon{\s)}}\\
&\quad -\frac43\,(m+1)\,\B_{m+1,\,n-1}^{{\s(}\varepsilon{\s)}}-
\frac13\, (m+2\,n-1)(2\,m+4\,n-1-2\,\varepsilon)\,\B_{m,\,n-1}
^{{\s(}\varepsilon{\s)}}\, .
\end{aligned}
\end{equation*}
\end{corollary}

\begin{remark}\label{R1}
Weierstrass himself wrote out recurrences not for his
$\sigma$-functions but for  functions
$S_\lambda=\exp\big\{\frac12\elambda z^2\}\,\sigmalambda(z)$ with
parameters $\big(\elambda,\,\varepsilon_{\s\lambda}=
3\pow{e}{\s\lambda}{2}-\frac14\g2\big)$ \cite[{\bf
II}:~pp.~253--254]{we}. A possible explanation of this choice is that
functions $S_\lambda$ yield a four-term recurrence like $A_{m,\,n}$,
whereas our recurrences are the five-term ones. However, one can show
that recurrences for functions $\sigmalambda$ through parameters
$(\g2,\g3)$ do not exist. Nevertheless transition between pairs
$(\elambda,\g2)$ and $(\elambda,\emu)$ is one-to-one and therefore this
universal recurrence may be written out in any of these
`representations'. It will be completely symmetric in the
$(\elambda,\emu)$-representation.
\end{remark}

Another kind of formulae for coefficients of the series is the
$\vartheta$-constant one. It is a natural choice for the power
$\theta$-series. Indeed, the $\vartheta$-constant expressions for
branch points $\elambda$ are well known. In turn, $\vartheta$-constants
are related by the Jacobi identity
$\vartheta_3^4=\vartheta_2^4+\vartheta_4^4$. This allows us to pass
between the representations, choosing arbitrary pair. If formulae are
written  through the constants
$(\vartheta_{\!\alpha0},\vartheta_{\!0\beta})$, where
$(\alpha,\beta)\ne(0,0)$, we shall use the term
\emph{$(\alpha,\beta)$-representation}.

The foregoing Halphen equations and their other representations can be
written in form of one equation if we avail of  operator \eqref{D}. We
can derive, \eg, that
\begin{align*}
\Larger{\widehat{\Smaller{\boldsymbol{\mathfrak D}}}} &=
\mbig[7](4\,\pow{e}{\s\lambda}{2}-\Mfrac23 \g2\mbig[7] )\,
\frac{\partial}{\partial\elambda}+
12\,\big(4\,\pow{e}{\s\lambda}{3}-\g2\,
\elambda\big)\frac{\partial}{\partial\g2}\,.
\end{align*}
The $\vartheta$-constant form for
$\Larger{\widehat{\Smaller{\boldsymbol{\mathfrak D}}}}$-operator is
also easily calculated since all the parameters of the theory, \ie,
$\g2$, $\g3$, $\elambda$, and $\vartheta_{\alpha\beta}$ are related to
each other by the polynomial relations \cite{tannery}. For example
$(\vartheta_2,\vartheta_4)$-representation for
\Larger{\widehat{\Smaller{\boldsymbol{\mathfrak D}}}} reads
$$
\Larger{\widehat{\Smaller{\boldsymbol{\mathfrak D}}}}= \frac{\pi^2}{3}
\mbig(\vartheta_4^8+2\,\vartheta_2^4\,\vartheta_4^4\mbig)
\frac{\partial}{\partial(\vartheta_4^4)} -\frac{\pi^2}{3}
\mbig(\vartheta_2^8+2\,\vartheta_2^4\,\vartheta_4^4\mbig)
\frac{\partial}{\partial(\vartheta_2^4)}\;.
$$
Let us denote $ \spin{\alpha}\DEF(-1)^\alpha$. Then the general
$(\alpha,\beta)$-representation to the Halphen operator is given by the
formula
$$
\Larger{\widehat{\Smaller{\boldsymbol{\mathfrak D}}}}=
\frac{\pi^2}{3}\mbig(\langle\beta\rangle\,
\pow{\vartheta}{\!\alpha0}{8}+ 2\,\spin{\alpha}\,
\pow{\vartheta}{\!\alpha0}{4}
\,\pow{\vartheta}{0\beta}{4}\mbig)
\frac{\partial}{\partial(\pow{\vartheta}{\!\alpha0}{4})}
-\frac{\pi^2}{3}\mbig(\spin{\alpha}\,
\pow{\vartheta}{0\beta}{8}+ 2\,\spin{\beta}\,
\pow{\vartheta}{\!\alpha0}{4}\,
\pow{\vartheta}{0\beta}{4}\mbig)
\frac{\partial}{\partial(\pow{\vartheta}{0\beta}{4} )} \;.
$$
Define also the symbol $\varepsilon$ depending on parity of the
function $\theta_{\alpha\beta}$:
$$
\varepsilon\DEF\frac{\spin{\alpha\beta}+1}{2}\quad
\hence\quad\bigg\{
\begin{aligned}
&\varepsilon=0\quad \mbox{if \ }
\theta_{\!\alpha\beta}=\pm\,\theta_1\\
&\varepsilon=1\quad \mbox{if \ } \theta_{\!\alpha\beta}=\pm\,
\theta_{2,3,4}
\end{aligned}\;.
$$

\begin{lemma}\label{L2}
Halphen's equation \eqref{epsilon} for the functions
$\Xi=\{\sigma,\sigmalambda\}$ has the form
$$
\frac{\partial^2\,\Xi}{\partial z^2}-
\Larger{\widehat{\Smaller{\boldsymbol{\mathfrak D}}}}\, \Xi
+\Big\{\varepsilon\,\elambda(\vartheta)+\Mfrac{\pi^4}{12^2}
\big(\vartheta_2^8+\vartheta_2^4\,\vartheta_4^4+ \vartheta_4^8 \big)
z^2\Big\}\,\Xi=0
$$
and its general $(\alpha,\beta)$-representation is as follows
\begin{equation}\label{eq}
\frac{\partial^2\,\Xi}{\partial
z^2}-\Larger{\widehat{\Smaller{\boldsymbol{\mathfrak D}}}}\, \Xi
+\Big\{e_{\s{\gamma\delta}}(\vartheta) + \Mfrac{\pi^4}{12^2}\,
\mbig[4][
\pow{\vartheta}{\!\alpha0}{8}+\spin{\alpha{+}\beta}\,
\pow{\vartheta}{\!\alpha0}{4}\,
\pow{\vartheta}{0\beta}{4}+
\pow{\vartheta}{0\beta}{8} \mbig[4]]\, z^2 \Big\}\,\Xi=0\,.
\end{equation}
Here,  quantities
\begin{equation}\label{ee}
e_{\gamma\delta}^{}(\vartheta)= \frac{\pi^2}{12} \mbig[3](
\spin{\gamma}\,\pow{\vartheta}{0\delta}{4}-
\spin{\delta}\,\pow{\vartheta}{\gamma0}{4}\mbig[3])
\end{equation}
do not dependent on representation $(\alpha,\beta)$ and correspond to
functions $\sigma$, $\sigmalambda$ by the rules
$$
\sigma\leftrightarrow e_{\s{00}}^{}= 0\,,\qquad
\sigma_{\s1}^{}\leftrightarrow e_{\s{01}}^{}= e_{\s 1}\,,\qquad
\sigma_{\s2}^{}\leftrightarrow e_{\s{11}}=e_{\s 2}\,,\qquad
\sigma_{\s3}^{}\leftrightarrow e_{\s{10}}= e_{\s 3}\,.
$$
\end{lemma}

We say that a representation is a \emph{proper} one (or symmetric) if
$(\gamma,\delta)=(\alpha,\beta)$.

\begin{remark}[historical]\label{R2}
Nice recurrences on the plane (analogs of $A_{m,n}$) were obtained by
Weierstrass already in the 1840's. At the time he was using his old
notation \emph{Al} instead of Jacobi's  $\Theta$ and $H$ \cite[{\bf
I}]{we}. At about the same time Jacobi considered power series for his
theta-functions and introduced the important multiplier $\re^{Az^2}$
much as Weierstrass did this for his $\sigma$-function. See further
Remark~\ref{R4} in Sect.~\ref{S5} or lectures by Koenigsberger
\cite[{\bf II}:~pp.~79--81]{koenig}.
\end{remark}

\subsection{Power $\theta$-series}

The functions $\theta$ are fundamental objects in numerous theories.
For this reason we shall give representation for their power series in
a maximally simplified (canonical) form. Because of this, we do collect
all the parameters similarly to Weierstrassian recurrence $A_{m,\,n}$,
so only multiplications of integers remain. It is not difficult to see
that the series under question must be series with coefficients being
polynomials in variables $\eta(\tau)$, $\vartheta(\tau)$. This follows
from the obvious formulae
\begin{equation}\label{st}
\theta_1(z|\tau)=\pi\,\ded^3(\tau)
\cdot\re^{\sm2\eta(\tau)z^2}_{\mathstrut}\sigma(2\,z|\tau)\;,
\quad\;\;\; \theta_{\s\lambda}\!(z|\tau)=
\vartheta_{\s{\!\lambda}}\!(\tau)
\cdot\re^{\sm2\eta(\tau)z^2}_{\mathstrut}
\sigma_{\s{\!\lambda}\sm1}\!(2\,z|\tau)\,.
\end{equation}
This entails that all the formulae that follows are derivable with use
of various $\vartheta$-re\-pre\-sent\-ations to operator
$\Larger{\widehat{\Smaller{\boldsymbol{\mathfrak D}}}}$ in Halphen's
equations \eqref{H}, \eqref{epsilon}, and \eqref{eq} followed by
multiplying the result into the series for an exponent. Computations
are somewhat lengthy but routine and we therefore omit them entirely.

\begin{theorem}\label{T1}
Power series for the function
\begin{equation}\label{theta1}
\begin{aligned}
\theta_1(z|\tau)&=\sideset{}{_k}\sum_{0}^{\infty}\,C_k(\tau)\cdot
z^{2k+1}\\
&=2\,\pi\,\ded^3 \,\bigg\{ z-2\,\eta{\cdot} z^3+\Big(2\,\eta^2-
\frac{\pi^4}{180}
\big(\vartheta_2^8+\vartheta_2^4\,\vartheta_4^4+\vartheta_4^8\big)
\Big) \cdot z^5+\cdots\bigg\}
\end{aligned}
\end{equation}
is determined by the analytic expression
\begin{equation}\label{t1}
\begin{aligned}
\theta_1(z|\tau)&=\phantom{\ded^3}2\,\pi\,\,
\mbox{\large$\ds\sideset{}{_k}\sum_{0}^\infty$}
\,\frac{(4\pi\ri)^k}{(2k+1)!}\,\frac{d^k\ded^3}{d\tau^k}
\cdot z^{2k+1}\\
&=2\,\pi\,\ded^3\,\, \mbox{\large$\ds\sideset{}{_k}\sum_{0}^\infty$}
\,(-2)^k\,\Bigg\{ \sideset{}{_\nu}\sum_{0}^{k}
{\Big(}\!\!{-}\frac{\pi^2}{6}\Big)^\nu
\frac{\eta^{k-\nu}\,\boldsymbol{\mathcal{N}}_{\!\nu}(\vartheta)}
{(k-\nu)!\,(2\nu+1)!} \Bigg\} \,z^{2k+1}\;,
\end{aligned}
\end{equation}
where the $\vartheta$-polynomial
$$
\boldsymbol{\calN}_{\!\nu}(\vartheta)=
\sideset{}{_s}\sum_{0}^{\nu}
\left\{\!\!\!
\begin{array}{r}
\G_{\nu\sm s,\,s}\cdot\vartheta_4^{4s\mathstrut}\,
\vartheta_2^{4{\s(}\nu\sm s{\s)}}\\
(-1)^s\,\G_{\nu\sm s,\,s}\cdot
\vartheta_3^{4s\ds\mathstrut}\,
\vartheta_4^{4{\s(}\nu\sm s{\s)}\ds\mathstrut}\\
(-1)^s\,\G_{s,\,\nu\sm s}
\cdot\vartheta_3^{4s\ds\mathstrut}\,
\vartheta_2^{4{\s(}\nu\sm s{\s)}\ds\mathstrut}
\end{array}\!\!\!
\right\}
$$
is chosen $($in braces$)$ according to which $\vartheta$-constant
$(4,2)$-, $(3,4)$-, or $(3,2)$-repre\-sentation is taken. Here, the
integral recurrence $\G_{m,\,n}$ is defined  as follows
\begin{equation*}
\begin{aligned}
\G_{m,\,n}&=4\,(n-2\,m-1)\,\G_{m,\,n\sm1}-
4\,(m-2\,n-1)\,\G_{m\sm1,\,n}
\\
&\quad -2\,(m+n-1)(2\,m+2\,n-1)\big(\G_{m\sm2,\,n}+
\G_{m\sm1,\,n\sm1}+\G_{m,\,n\sm2}\big)\,,\\
\G_{\s0,\s0}&=1\quad\mbox{and}\quad\G_{m,\,n}=0\mbox{\ \ if\ \ }
m,n<0
\end{aligned}
\end{equation*}
and  has the symmetry property $\G_{m,n}=(-1)^{m+n}\,\G_{n,m}$.
\end{theorem}

\begin{remark}\label{R3}
We might of course derive representation of the type $\theta_1=\sum
C_{\mathit{mnp}}\,\g2{}^{\hspace{-0.4em}m}\,\g3{}^{\hspace{-0.4em}n}\,
\eta^p\cdot z^k$ like Weierstrassian recurrence, but $\G_{m,n}$ is more
effective than $A_{m,n}$ since all the polynomials have already been
collected in $\vartheta$-constants.
\end{remark}

It is interesting to observe that  odd derivatives
$\theta_1^{\s(2k+1)}(0|\tau)$, \ie, coefficients in front of $z^{2k+1}$
in \eqref{theta1}--\eqref{t1}, generate polynomial expressions in
variables $(\eta,\vartheta)$ which are exactly integrable $k$ times in
$\tau$.

\begin{theorem}\label{T2}
Power series for the functions
$\theta\AB{\alpha}{\beta}=\pm\theta_{2,3,4}$$:$
\begin{equation}\label{theta234}
\begin{split}
\theta\AB{\alpha}{\beta}(z|\tau) &=
\sideset{}{_k}\sum_{0}^{\infty}\,
\pow{C}{k}{\s(\alpha,\beta\s)}(\tau)\cdot z^{2k}\\
&=\vartheta\AB{\alpha}{\beta}- \vartheta\AB{\alpha}{\beta} \mbig[7]\{
2\,\eta+\frac16\,\pi^2
\mbig(\spin{\beta}\,\vartheta\AB{\alpha\sm1}{0}^4-
\spin{\alpha}\,\vartheta\AB{0}{\beta\sm1}^4\mbig)\mbig[7]\}z^2+\cdots
\end{split}
\end{equation}
are determined by the analytic expressions
\begin{equation}\label{series}
\theta\AB{\alpha}{\beta}(z|\tau) = \sideset{}{_k}\sum_{0}^\infty
\,\frac{(4\pi\ri)^k}{(2k)!}\,
\frac{d^k\vartheta\AB{\alpha}{\beta}}{d\tau^k} \cdot z^{2k}\;.
\end{equation}
The proper representation to the series \eqref{series} has the form
\begin{equation}\label{t234}
\theta\AB{\alpha}{\beta}(z|\tau)= \vartheta\AB{\alpha}{\beta}\,\,
\mbox{\large$\ds\sideset{}{_k}\sum_{0}^\infty$}
\,(-2)^k\,\Bigg\{ \sideset{}{_\nu}\sum_{0}^{k}
{\Big(}\!\!{-}\frac{\pi^2}{6}\Big)^\nu\,
\frac{\eta^{k\sm\nu}\,
\boldsymbol{\calN}_\nu^{\s(\alpha,\beta\s)}
(\vartheta)}{(k-\nu)!\,(2\nu)!} \Bigg\}\, z^{2k}
\end{equation}
with the following universal integral recurrence\/$:$
$$
\boldsymbol{\calN}_\nu^{\s(\alpha,\beta\s)}(\vartheta)=
\mbox{\large$\ds\sideset{}{_s}\sum_{0}^{\nu}$}\;
\G_{s,\,\nu\sm s}^{\s(\alpha,\beta\s)} \cdot
\vartheta\AB{\alpha\sm1}{0}^{4s}\,
\vartheta\AB{0}{\beta\sm1}^{4{\s(}\nu\sm s{\s)}\mathstrut} \;,
$$
\begin{equation*}
\begin{aligned}
\G_{m,\,n}^{\s(\alpha,\beta\s)}&= \spin{\alpha}\,(4\,n-8\,m-3)\,
\G_{m,\,n\sm1}^{\s(\alpha,\beta\s)}- \spin{\beta}\,(4\,m-8\,n-3)\,
\G_{m\sm1,\,n}^{\s(\alpha,\beta\s)}
\\&\==
-2\,(m+n-1)(2\,m+2\,n-3) \big(\G_{m\sm2,\,n}^{\s(\alpha,\beta\s)}+
\spin{\alpha{+}\beta}\, \G_{m\sm1,\,n\sm1}^{\s(\alpha,\beta\s)}+
\G_{m,\,n\sm2}^{\s(\alpha,\beta\s)}\big)\,,
\end{aligned}
\end{equation*}
where $\G_{\s 0,\,\s0}^{\s(\alpha,\beta\s)}=1$ and
$\G_{m,\,n}^{\s(\alpha,\beta\s)}=0$ if $m,n<0$.
\end{theorem}

Some remarks are in order. These recurrences are quite effective but
there are additional symmetry properties which reduce computations in
half. It is evident from the recurrence
$\G_{m,\,n}^{\s(\alpha,\beta\s)}$ itself that it has a symmetry with
respect to permutations of indices
\begin{equation}\label{perm}
\G_{n,\,m}^{\s(\alpha,\beta\s)}= (-1)^{(m+n)(\alpha+\beta+1)}\,
\G_{m,\,n}^{\s(\alpha,\beta\s)}\;,\qquad
\G_{m,\,n}^{\s(\beta,\alpha\s)}= (-1)^{(m+n)(\alpha+\beta)}\,
\G_{m,\,n}^{\s(\alpha,\beta\s)}\;.
\end{equation}
This means  that we have in effect only two recurrences
$(\alpha,\beta)=(1,0)$ and $(\alpha,\beta)=(0,0)$, \ie, $\beta$ is
always equal to zero. Redenoting
$\G_{m,\,n}^{\s(\alpha,\beta\s)}=\G_{m,\,n}^{\s(\alpha,0\s)}\FED
\G_{m,\,n}^{\s(\alpha\s)}$, we have
\begin{equation*}
\begin{aligned}
\G_{m,\,n}^{\s(\alpha\s)}&= \spin{\alpha}\,(4\,n-8\,m-3)\,
\G_{m,\,n\sm1}^{\s(\alpha\s)}- (4\,m-8\,n-3)\,
\G_{m\sm1,\,n}^{\s(\alpha\s)}
\\&\==
-2\,(m+n-1)(2\,m+2\,n-3) \big(\G_{m\sm2,\,n}^{\s(\alpha\s)}+
\spin{\alpha}\, \G_{m\sm1,\,n\sm1}^{\s(\alpha\s)}+
\G_{m,\,n\sm2}^{\s(\alpha\s)}\big)
\end{aligned}
\end{equation*}
and permutations \eqref{perm} therefore reduce to the simple formulae
$$
\G_{n,\,m}^{\s(0\s)}= (-1)^{(m+n)}\,
\G_{m,\,n}^{\s(0\s)}\;,\qquad \G_{n,\,m}^{\s(1\s)}=
\G_{m,\,n}^{\s(1\s)}\;.
$$

We see that recurrences \eqref{t1} and \eqref{t234} differ only in
multipliers. Hence they can be unified into one recurrence  much as we
did it in \eqref{universal} by introducing the parity $\varepsilon$,
but the quantity $\spin{\alpha}$ still remains. Computer tests show
that $(m,n)$-entries of matrices $\G^{\s(\beta,\alpha\s)}$ differ each
other only in  signs but we failed to find this rule.

\begin{corollary}
All the coefficients $C_k(\tau)$ and
$\pow{C}{k}{\s(\alpha,\beta\s)}(\tau)$ are the $k$-fold exactly
$\tau$-integrable $(\eta,\vartheta)$-polynomials.
\end{corollary}

In Sect.~\ref{S7} we shall show that this integrability is a
consequence of one dynamical system. With use of the formulae above one
can construct series in neighborhoods of points
$z=\big\{{\pm}\frac12,\,\pm\frac\tau2\big\}$. By virtue of
\eqref{shift}, the resulting series are transformed into each other
with some obvious modifications.

\section{Dynamical systems satisfied by $\theta$-series\label{S4}}

\noindent In this and next section we describe new and important
property of Jacobi's $\vartheta$, $\theta$, and $\theta'$-series.
These, along with elliptic, elementary, or rational functions, are
differentially closed and define thereby the calculus in its own right.

\begin{theorem}\label{T3}
The five functions $\theta_1(z|\tau)$, $\theta_2(z|\tau)$,
$\theta_3(z|\tau)$, $\theta_4(z|\tau)$, and $\Dtheta(z|\tau)$ satisfy
the closed autonomous ordinary differential equations over the field of
coefficients $\vartheta_2$, $\vartheta_3$, $\vartheta_4$, and
$\eta$\/$:$
\begin{equation}\label{X'}
\left\{
\begin{aligned}
\frac{\partial\theta_1}{\partial z}&=\Dtheta\\
\frac{\partial\theta_2}{\partial z}&=
\frac{\Dtheta}{\theta_1}\,\theta_2- \pi\,\vartheta_2^2\cdot
\frac{\theta_3\theta_4}{\theta_1}\\
\frac{\partial\theta_4}{\partial z}&=
\frac{\Dtheta}{\theta_1}\,\theta_4^{}- \pi\,\vartheta_4^2\cdot
\frac{\theta_2\theta_3}{\theta_1}\\
\frac{\partial\theta_3}{\partial z}&=
\frac{\Dtheta}{\theta_1}\,\theta_3- \pi\,\vartheta_3^2\cdot
\frac{\theta_2\theta_4}{\theta_1}\\
\frac{\partial\Dtheta}{\partial z}&=
\frac{\Dtheta{}^2}{\theta_1}-\pi^2\vartheta_3^2\,\vartheta_4^2
\cdot \frac{\theta_2^2}{\theta_1}-
4\,\bigg\{\eta+\frac{\pi^2}{12}\big(\vartheta_3^4+\vartheta_4^4\big)
\bigg\} \cdot\theta_1\,.
\end{aligned}
\right.
\end{equation}
\end{theorem}

\begin{proof}
The proof is based on the theta-function differential identities which
occur infrequently in the literature (mostly in the old books;
\cite[p.~82]{weber}, \cite[{\bf II}:~p.~173]{tannery}\footnote{These
important relations are very implicitly  present in Jacobi's
\emph{Werke} \cite[\bf I]{jacobi} but have not got to the thorough
handbook for elliptic functions \cite{we2} compiled by Schwarz from
Weierstrassian lectures. Even lectures themselves \cite{we} contain no
these identities in  $\theta$-form. They present in \cite[p.~29]{we2},
\cite{we} in form of their $(\zeta,\sigmalambda)$-equivalents.
Differential relations for quotients of $\theta$-functions are of
course well known. These are differential equations for elliptic
functions \cite[Sect.~21$\boldsymbol{\cdot}$6]{WW},
\cite{baruch,krause,koenig,weber}.}). These relationships are nothing
else but those between the Weierstrassian functions
$(\sigma,\zeta)(z|\tau)$ taken at different half-periods \cite{we2}. We
can present them in the compact form
\begin{equation}\label{diff}
\theta_\mu\Dtheta[\nu]-
\theta_\nu\Dtheta[\mu]=\mathfrak{sign}(\nu-\mu)\,
\pi\,\vartheta_k^2\cdot\theta_1\theta_k\,,
\end{equation}
where $k=2,3,4$ and
\begin{equation}\label{nm}
\nu=\frac{8\,k-28}{3\,k-10}\,,\qquad\mu=\frac{10\,k-28}{3\,k-8}\,;
\end{equation}
the triple $(k,\nu,\mu)$ runs over the set $\{(2, 3, 4)$, $(3, 4, 2)$,
$(4, 2, 3)\}$. In order to turn identities \eqref{diff} into
differential equations we should find their differential closure.
Taking the property  $\vartheta_1\equiv 0$ into account, we can solve
\eqref{diff} with respect to the $\theta$-derivatives and rewrite the
result as first four equations in \eqref{X'}:
$$
\frac{\partial\theta_k}{\partial z}= \frac{\Dtheta}{\theta_1}
\,\theta_k- \pi\,\vartheta_k^2\cdot
\frac{\theta_\nu\,\theta_\mu}{\theta_1}\,,\qquad k=1,2,3,4\,.
$$
It only remains to compute  derivative of the object $\Dtheta$.
Consider  Weierstrassian identity
$$
(\sigma\,\zeta)'=\sigma\,\zeta^2-\sigma\,\wp
$$
and convert it into the $\theta$-functions. Then formulae
\begin{alignat}{1}
\zeta(2z|\tau)&=2\,\eta(\tau)\,z+\frac12\,
\frac{\Dtheta\!(z|\tau)}{\theta_1\!(z|\tau)}\,,\label{wp1}\\
\wp(2z|\tau)&=\frac{\pi^2}{12}\,
\bigg\{\vartheta_3^4(\tau)+\vartheta_4^4(\tau)+3\,\vartheta_3^2(\tau)\,
\vartheta_4^2(\tau)\,
\frac{\theta_2^2(z|\tau)}{\theta_1^2(z|\tau)}\bigg\} \label{wp2}
\end{alignat}
and formula \eqref{st} yield the fifth equation in \eqref{X'}.
\end{proof}

Weierstrass himself derived a $(\zeta,\sigmalambda)$-equivalent of
relations \eqref{diff} in a reverse order
\cite[\hbox{\S\S\,24--25}]{we2}, \ie, by differentiating the
$\sigma$-identities followed by use of differential equations for
ratios of $\sigma$-functions. It should be noted here that the closed
$\theta$-form of Weierstrassian identities contains not only branch
points $\elambda$, \ie, $\vartheta$-constants, but also the `constant'
$\eta$. In other words, the closed differential $\theta$-apparatus
inevitably contains a fifth function---any of $\Dtheta[k]$---and period
of a meromorphic elliptic integral, \ie,  $\eta(\tau)$; the total
number of equations is thus equal to five.

As mentioned in Introduction, Theorem~\ref{T3} appears explicable on
the basis of the theory of Abelian integrals. Namely, these integrals
are differentially closed and elliptic functions are the particular
case of the meromorphic integrals (integrals of exact differentials).
The logarithmic integral is a logarithmic $\theta$-ratio and canonical
meromorphic integral---Weierstrassian $\zeta$-function---is
proportional to the fifth function $\Dtheta$.

It is notable that the famous Jacobi identity
$\Dvartheta=\pi\,\vartheta_2\,\vartheta_3\,\vartheta_4$ turns out to be
an automatical consequence of Eqs.~\eqref{X'} taken at point $z=0$ and
this property pertains equally to generalizations of this identity
presented by formulae \eqref{jacobi}--\eqref{jac} next. By this means
we get one more (simple) proof of Jacobi's identity meanwhile in the
Whittaker--Watson book \cite{WW} all the known proofs of this identity
are summarized as `none are simple'
\cite[Sect.~21$\boldsymbol{\cdot}$41]{WW}.

The next step  suggests itself. All the $\theta,\theta'$-functions
satisfy the heat equation:
\begin{equation}\label{heat}
4\,\pi\,\ri\,\frac{\partial\theta}{\partial\tau}=
\frac{\partial^2\theta}{\partial z^2}\,,\qquad
4\,\pi\,\ri\,\frac{\partial\theta'}{\partial\tau}=
\frac{\partial^2\theta'}{\partial z^2}\,.
\end{equation}
Therefore, invoking Theorem~\ref{T3}, we establish that differential
closedness is shared also by theta-functions as functions of their
second argument.

\begin{theorem}\label{T4}
The five Jacobi's functions $\theta_k(z|\tau)$ and $\Dtheta(z|\tau)$
satisfy the  closed non-au\-to\-no\-mous ordinary differential
equations\/$:$
\begin{alignat*}{5}
\frac{\partial \theta_1}{\partial\tau}&=
\frac{-\ri}{4\pi}\,\frac{\Dtheta{}^2}{\theta_1}+{} &&{}+ \frac\ri4\pi
\,\vartheta_3^2\,\vartheta_4^2\cdot \frac{\theta_2^2}{\theta_1} &&{}+
\frac{\ri}{\pi}\bigg\{\eta+\frac{\pi^2}{12}
\big(\vartheta_3^4+\vartheta_4^4\big) \bigg\}\cdot\theta_1\,,
\\
\frac{\partial \theta_2}{\partial\tau}&= \frac{-\ri}{4\pi}
\bigg\{\frac{\Dtheta}{\theta_1}- \pi\,\vartheta_2^2\cdot
\frac{\theta_3\theta_4} {\theta_1\theta_2}\bigg\}^{\!\!2}\theta_2
&&{}+\frac\ri4\pi\,\vartheta_3^2\,\vartheta_4^2\cdot
\frac{\theta_1^2}{\theta_2}&&{}+
\frac{\ri}{\pi}\bigg\{\eta+\frac{\pi^2}{12}
\big(\vartheta_3^4+\vartheta_4^4\big) \bigg\}\cdot\theta_2\,,
\\
\frac{\partial \theta_3}{\partial\tau}&= \frac{-\ri}{4\pi}\,
\frac{\Dtheta{}^2}{\theta_1^2}\,\theta_3
+\frac{\ri}{2}\vartheta_3^2\cdot
\Dtheta\,\frac{\theta_2\theta_4}{\theta_1^2} &&{}-
\frac\ri4\pi\,\vartheta_2^2\,\vartheta_3^2\cdot
\frac{\theta_4^2}{\theta_1^2}\,\theta_3&&{}+
\frac{\ri}{\pi}\bigg\{\eta+\frac{\pi^2}{12}
\big(\vartheta_3^4+\vartheta_4^4\big) \bigg\}\cdot\theta_3\,,
\\
\frac{\partial \theta_4}{\partial\tau}&=
\frac{-\ri}{4\pi}\,\frac{\Dtheta{}^2}{\theta_1^2}\,\theta_4
+\frac{\ri}{2}\vartheta_4^2\cdot
\Dtheta\,\frac{\theta_2\theta_3}{\theta_1^2} &&{}-
\frac\ri4\pi\,\vartheta_2^2\,\vartheta_4^2\cdot
\frac{\theta_3^2}{\theta_1^2}\,\theta_4&&{}+
\frac{\ri}{\pi}\bigg\{\eta+\frac{\pi^2}{12}
\big(\vartheta_3^4+\vartheta_4^4\big) \bigg\}\cdot\theta_4\,,
\\
\frac{\partial \Dtheta}{\partial\tau}&=
\frac{-\ri}{4\pi}\,\frac{\Dtheta{}^3}{\theta_1^2}
+\frac{3\,\ri}{\pi}\bigg\{
\frac{\pi^2}{4}\,\vartheta_3^2\,\vartheta_4^2\cdot
\frac{\theta_2^2}{\theta_1^2} +\eta+
\frac{\pi^2}{12}\big(\vartheta_3^4+\vartheta_4^4\big) \bigg\}\,\Dtheta
-\frac\ri2\pi^2\,\vartheta_2^2\,\vartheta_3^2\, \vartheta_4^2\cdot
\frac{\theta_2\theta_3\theta_4}{\theta_1^2}\,. \hspace{-50em}
\end{alignat*}
\end{theorem}

When deriving second of these equations the standard quadratic
theta-identities
\begin{equation}\label{j}
\vartheta_2^2\,\theta_4^2-\vartheta_4^2\,\theta_2^2=
\vartheta_3^2\,\theta_1^2\,,\qquad
\vartheta_2^2\,\theta_3^2-\vartheta_3^2\,\theta_2^2=
\vartheta_4^2\,\theta_1^2
\end{equation}
were used. This system, combined with Eqs.~\eqref{X'}, constitutes a
complete set of rules for differential computations with theta-series
and, invoking notation \eqref{nm}, the rules can be written in the
compact form
\begin{align}
& \left\{
\begin{aligned}
\frac{\partial\theta_k}{\partial z}&= \frac{\Dtheta}{\theta_1}
\,\theta_k- \pi\,\vartheta_k^2\cdot
\frac{\theta_\nu\,\theta_\mu}{\theta_1}\\
\frac{\partial\Dtheta}{\partial z}&=
\frac{\Dtheta{}^2}{\theta_1}-\pi^2\vartheta_3^2\,\vartheta_4^2
\cdot \frac{\theta_2^2}{\theta_1}-
4\,\bigg\{\eta+\frac{\pi^2}{12}\big(\vartheta_3^4+\vartheta_4^4\big)
\bigg\} \cdot\theta_1\,,\label{X}
\end{aligned}
\right.\\[2ex]
&
\left\{
\begin{aligned}\label{Dtau}
\frac{\partial\theta_k}{\partial \tau}&= \frac{-\ri}{4\pi}\,
\frac{\Dtheta{}^2}{\theta_1^2} \,\theta_k+ \frac{\ri}{2}\,
\vartheta_k^2\cdot\Dtheta\,
\frac{\theta_\nu\,\theta_\mu}{\theta_1^2}\\
&\==+ \frac\ri4\pi \Big\{
\vartheta_3^2\,\vartheta_4^2\cdot\theta_2^2-
\vartheta_k^2\,\vartheta_\mu^2\cdot\theta_\nu^2-
\vartheta_k^2\,\vartheta_\nu^2\cdot\theta_\mu^2
\Big\}\,\frac{\theta_k}{\theta_1^2}
+\frac{\ri}{\pi}\bigg\{ \eta+\frac{\pi^2}{12}
\big(\vartheta_3^4+ \vartheta_4^4\big)\bigg\}\cdot
\theta_k\\
\frac{\partial \Dtheta}{\partial\tau}&=
\frac{-\ri}{4\pi}\,\frac{\Dtheta{}^3}{\theta_1^2}
+\frac{3\,\ri}{\pi}\bigg\{
\frac{\pi^2}{4}\,\vartheta_3^2\,\vartheta_4^2\cdot
\frac{\theta_2^2}{\theta_1^2} +\eta+
\frac{\pi^2}{12}\big(\vartheta_3^4+\vartheta_4^4\big)
\bigg\}\,\Dtheta \\
&\==-\frac\ri2\pi^2\,\vartheta_2^2\,\vartheta_3^2\,
\vartheta_4^2\cdot
\frac{\theta_2\theta_3\theta_4}{\theta_1^2}\,,
\end{aligned}
\right.
\end{align}
where $k=1,2,3,4$. These formulae, incidentally, are not completely
symmetric and no theta-identities were involved when deriving them;
this important point is discussed  in Sect.~\ref{S9.3}.

\section{The $\vartheta$-constant differential calculus\label{S5}}

Equations \eqref{Dtau} contain $\vartheta(\tau)$- and
$\eta(\tau)$-constants but their derivatives have not been defined yet.
On the other hand, Weierstrass' invariants \eqref{star} have their
$\vartheta$-constant equivalents
\begin{equation}\label{g23}
\begin{aligned}
\g2(\tau)&={\phantom{4}}\frac{\pi^4}{24}
\big\{\vartheta_2^8(\tau)+\vartheta_3^8(\tau)+\vartheta_4^8(\tau)
\big\}\,,\\
\g3(\tau)&=\,\frac{\pi^6}{432}
\big\{\vartheta_2^4(\tau)+\vartheta_3^4(\tau) \big\}
\big\{\vartheta_3^4(\tau)+\vartheta_4^4(\tau) \big\}
\big\{\vartheta_4^4(\tau)-\vartheta_2^4(\tau) \big\}
\end{aligned}
\end{equation}
and satisfy the famous Halphen dynamical system \cite[{\bf I}:~pp.~331,
449--450]{halphen}
\begin{equation}\label{g2g3}
\frac{d \g2}{d\tau} = \frac{\ri}{\pi}
\Big(8\,\g2\,\eta-12\,\g3\Big)\,,\qquad \frac{d
\g3}{d\tau}= \frac{\ri}{\pi}
\Big(12\,\g3\,\eta-\Mfrac23\,\g2{}^{\hspace{-0.4em}2}\Big)\,,\qquad
\frac{d\eta}{d\tau}=
\frac{\ri}{\pi}\Big(2\,\eta^2-\Mfrac16\,\g2\Big)\,,
\end{equation}
involving the $\eta$-function. In implicit form this system was also
written down by Weierstrass \cite[{\bf II}:~p.~249]{we} and Ramanujan
obtained its equivalent \cite{ramanujan} when studying his known
number-theoretic $P,Q,R$-series. It immediately follows that
differential closure of the $\vartheta$'s requires an extension of the
system \eqref{g2g3} to a 4-dimensional version. It is a direct
corollary of Eqs.~\eqref{g23}--\eqref{g2g3}.

\begin{theorem}\label{T5}
Jacobi's $\vartheta$-constants are differentially closed upon
adjoining the Weierstrass $\eta$-function\/$:$
\begin{equation}\label{var}
\begin{aligned}
\frac{d\vartheta_2}{d\tau}&=
\frac{\ri}{\pi}\,\bigg\{\eta+\frac{\pi^2}{12}\,
\big(\vartheta_3^4+\vartheta_4^4
\big)\bigg\}\,\vartheta_2\,,\\
\frac{d\vartheta_3}{d\tau}&=
\frac{\ri}{\pi}\,\bigg\{\eta+\frac{\pi^2}{12}\,
\big(\vartheta_2^4-\vartheta_4^4 \big)\bigg\}\,\vartheta_3\,,
\end{aligned}
\qquad
\begin{aligned}
\frac{d\vartheta_4}{d\tau}&=
\frac{\ri}{\pi}\,\bigg\{\eta-\frac{\pi^2}{12}\,
\big(\vartheta_2^4+\vartheta_3^4
\big)\bigg\}\,\vartheta_4\,,\\
\frac{d\eta}{d\tau}&=\frac{\ri}{\pi}\,\bigg\{2\,\eta^2-
\frac{\pi^4}{12^2}
\big(\vartheta_2^8+\vartheta_3^8+\vartheta_4^8 \big) \bigg\}\,.
\end{aligned}
\end{equation}
\end{theorem}

\begin{remark}[historical]\label{R4}
It is less known that Jacobi wrote out a different analog to this
4-dimensional dynamical system but Halphen does not mention this fact.
This result was published by Borchardt in 1857 on the basis of
manuscripts kept after Jacobi's death \cite[{\bf
II}:~pp.~383--398]{jacobi}. Namely, Jacobi introduced the four
variables $(A,B,a,b)$ in terms of Legendre's quantities
$(k,k',\ellK,\ellE)$ and showed that they satisfy the nice monomial
dynamical system (we keep completely to Jacobi's notation in
\cite[p.~386]{jacobi})
\begin{equation}\label{ABab}
{\left\{\vbox to 6.2ex{}\right.}
\begin{alignedat}{5}
\frac{\partial A}{\partial h}&=2\,A^2B\,,&\qquad
\frac{\partial a}{\partial h}&=-16\,b\,A^2,\\
\frac{\partial B}{\partial h}&=b\,A^3\,,&\qquad
\frac{\partial b}{\partial h}&=a\,b\,A^2\,,
\end{alignedat}
\end{equation}
where $h=\frac14\pi\,\ri\,\tau$. Interestingly enough, Jacobi
considered system \eqref{ABab} in a context of the power series for
$\theta$-functions and  noticed \cite[{\bf II}:~p.~390]{jacobi} that
the series would be simple if one extracts the exponential multiplier
\mbox{$\exp\!\big\{\!{-}\frac12 ABz^2\big\}$}. He described
corresponding recurrences for $\theta_k$ \cite[{\bf
II}:~pp.~394--398]{jacobi} and one can readily see that they are
equivalent to the Weierstrass--Halphen differential recurrence
\eqref{D} for $\sigma$-functions. Exhaustive comments to the Jacobi
system and its relation to Eqs.~\eqref{var} can be found in \cite{br3}.
\end{remark}

Before the system \eqref{ABab} was derived, Jacobi also obtained some
its analogs (see \cite[{\bf II}:~p.~176]{jacobi}) and, in particular,
his remarkable differential equation of 3rd order for the
$\vartheta$-series
\begin{equation}\label{Jacobi}
C^4\big(\!\ln
C^3C_{\tau\tau}\big)_\tau^2=16\,C^3C_{\tau\tau}-\pi^2\,,\qquad
C\DEF\vartheta^{\sm2}\,.
\end{equation}
($C$ is Jacobi's notation). In turn, a simple computation shows that
logarithmic derivatives of the $\vartheta$-series also satisfy a
compact differential equation which we shall meet in Sect.~\ref{S9.3}.
The equation and its general solution are as follows:
\begin{equation}\label{halphen}
\big(X_\tau-2\,X^2 \big)X_{\tau\tau\tau}
-\pow{X}{\tau\tau}{2} +16\,X^3 X_{\tau\tau}+4\,
\big(X_\tau-6\,X^2 \big)\pow{X}{\tau}{2}=0\,,
\end{equation}
$$
X=\frac{d}{d\tau}\ln\frac{\vartheta_{k\!}
\mbig(\mfrac{a\,\tau+b}{c\,\tau+d}\mbig)}{\sqrt{c\,\tau+d\,}}\,.
$$
We should also mention that  well-known differential relations on
logarithms of ratios $\vartheta_2\!:\!\vartheta_3\!:\!\vartheta_4$
\cite{weber,tannery}
\begin{equation*}\label{ln}
\frac{d}{d\tau}\ln\frac{\vartheta_2}{\vartheta_3}=
\frac{\ri}{4}\,\pi\,\vartheta_4^4\,,\qquad
\frac{d}{d\tau}\ln\frac{\vartheta_3}{\vartheta_4}=
\frac{\ri}{4}\,\pi\,\vartheta_2^4\,,\qquad
\frac{d}{d\tau}\ln\frac{\vartheta_2}{\vartheta_4}=
\frac{\ri}{4}\,\pi\,\vartheta_3^4
\end{equation*}
(see also  old dissertation \cite{baruch}) and ordinary differential
equation of Chazy \cite{chazy,clarkson}
\begin{equation}\label{chazy}
\pi\,\eta_{\vbox to1.4ex{}\tau\tau\tau}=
12\,\ri\,\big(2\,\eta\,\eta_{\vbox to1.4ex{}\tau\tau}
-3\,\eta_{\vbox to1.4ex{}\tau}{}^{\hspace{-0.4em}2}\big)
\end{equation}
are the direct consequences of the system \eqref{var}.

\begin{remark}[exercise]\label{R5}
If we view the last equation in \eqref{g2g3} as a Riccati equation then
we  get an interesting example of the solvable linear 2nd order \ode.
Coefficient of this equation is proportional to the everywhere
holomorphic in $\Hp$ form $\g2(\tau)$ which is an automorphic one with
respect to group $\boldsymbol{\Gamma}(1)$. Carry out the calculations
and bring the equation into the following form
$$
\Psi''+\frac{\g2(\tau)}{3\pi^2}\,\Psi=0\,.
$$
With use of differential relation \eqref{ded} generalize this equation
to equation
$$
\Psi''+\Mfrac{n+2}{\pi\,\ri}\,\eta(\tau)\,\Psi'-
\Mfrac{n}{6\pi^2}\,\g2(\tau)\,\Psi=0
$$
and prove that
$$
\Psi=\frac{1}{\ded^n(\tau)}\Big(A+B
\Smaller[2]{\ds\int\limits^{\Smaller[1]{\;\tau}}}
\!\ded^{2n}(\tau)\,d\tau\Big)
$$
is its general solution.
\end{remark}

To this section, we note that formulae for multiple differentiating the
$\vartheta,\eta$-constants are given explicitly as coefficients of
series  \eqref{series}--\eqref{t234}. The same coefficients provide the
general expressions for  quantities $\theta^{(n)}(0|\tau)$; relations
between derivatives $\theta'(0|\tau)$, $\theta''(0|\tau)$, \ldots,
$\theta^{(n)}(0|\tau)$ under small $n$ are very often used in the
literature as auxiliary identities \cite{mumford,WW,
bateman,tannery,weber,krause,koenig,hancock,riemann,lawden} (Baruch's
dissertation \cite{baruch} contains a lot of such identities). See also
work \cite{buchstaber} where modular functions like $\eta$,
$\vartheta$, $\vartheta'$, etc appear in the theory of hydrodynamical
chains and differential calculus described above significantly
simplifies computations in this work. Thanks to the fact that group
$\boldsymbol{\Gamma}(1)$ has a lot of interesting and nontrivial
subgroups the number of known differential systems related to the base
one \eqref{var} is far from being  exhausted. Even next  to
$\boldsymbol{\Gamma}(1)$ groups like  $\boldsymbol{\Gamma}_0(N)$
inspire a rich theory. See, for example, recent work \cite{maier}
containing many nice results along these lines and additional
bibliography.

\section{Unification: $\theta,\theta'$-functions with
characteristics\label{S6}}

\noindent In this section we summarize the previous results and other
basic properties of theta-functions and their derivatives in a unified
notation, \ie, in terms of theta-characteristics with use of
$(\alpha,\beta)$-representations. This will enable us primarily to
trivialize and  automate analytic manipulation with theta-functions by
including fundamental operations: shifts by half-periods, modular
transformations, and differential computations. Apart from unification
of formulae this can serve as the basis for further generalization to
the theta-functions of higher genera.

Any object, symmetrical in $\vartheta$-constants, can be written in
$(\alpha,\beta)$-representation. For example branch points \eqref{ee}
or $(\alpha,\beta)$-representation for invariants \eqref{g23} can be
written as follows:
\begin{alignat}{4}
\g2(\tau)&=& \frac{\pi^4}{12}\,&&&\Big\{ \pow{\vartheta}{\alpha0}{8}+
\spin{\alpha{+}\beta}\, \pow{\vartheta}{\!\alpha0}{4} \,
\pow{\vartheta}{0\beta}{4} + \pow{\vartheta}{0\beta}{8}\Big\}\;,
\qquad\quad
(\alpha,\beta)\ne(0,0)\notag\\
\g3(\tau)&={}&\frac{\pi^6}{432}\,&&&\Big\{ 2\,\spin{\beta}\,
\pow{\vartheta}{\!\alpha0}{12} - 3\, \pow{\vartheta}{\alpha0}{4} \,
\pow{\vartheta}{0\beta}{4}\, \big(\spin{\beta}\,
\pow{\vartheta}{0\beta}{4} -\spin{\alpha}\,
\pow{\vartheta}{\!\alpha0}{4} \big)
-2\,\spin{\alpha}\,\pow{\vartheta}{0\beta}{12} \Big\}\;.\notag
\end{alignat}
Other examples are the Jacobi identity
\begin{equation}\label{324}
\vartheta_3^4(\tau)=\vartheta_2^4(\tau)+\vartheta_4^4(\tau)
\end{equation}
and formula $\Dvartheta=2\pi\ded^3$; they have the following
$(\alpha,\beta)$-representation:
\begin{equation}\label{jacobi}
\begin{split}
\vartheta\AB{\alpha}{\beta}^{\,4} &= \mbig(
\spin{\beta}\,\vartheta\AB{\alpha\sm1}{0}^4+
\spin{\alpha}\,\vartheta\AB{0}{\beta\sm1}^4\mbig)
\,\Mfrac{\spin{\alpha\beta}+1}{2}\,,\\
\Dvartheta[\alpha\beta](\tau)
&= \ri^{\beta+1}\,\big(1-\spin{\alpha\beta}\big)\cdot
\pi\,\ded^3(\tau)\,.
\end{split}
\end{equation}
Here and hereafter $\Dvartheta[\alpha\beta](\tau)$ is understood to be
equal to $\Dtheta[\alpha\beta](0|\tau)$.

\subsection{Shifts by half-periods for $\theta$-derivatives.}

In connection with appearance of the object $\Dtheta$, we should
augment the rule \eqref{shift} by involving the fact that algebraic and
differential closedness of $\theta$'s entails a transformation law for
their derivatives. Hence, it is naturally to be expected that any
function $\Dtheta[\alpha\beta](z|\tau)$ with a $z$-argument shifted by
some half-period is expressible through the function $\Dtheta(z|\tau)$
and  other functions $\theta_{1,2,3,4}(z|\tau)$. The ultimate solution
is, however, not a simple differential consequence of formula
\eqref{shift} and is far from being obvious. It should be put to better
use as an independent property.

\begin{theorem}[Transformation law for  $\theta$-derivatives]\label{T6}
Let $\alpha,\beta,m,n$ be integers. Then
\begin{multline*}
\hspace{-0.7em} \Dtheta[\alpha\beta] \Big(z+\Mfrac n2+\Mfrac m2\,\tau
\Big|\tau\Big)= \ri^{\sm m(\beta+n)}_{\mathstrut}\cdot
\re^{\sm\s{\frac14}\pi\ri\, m(4z+m\tau)}_{\mathstrut} \,\Big\{ \big(
\Dtheta(z|\tau)-\pi\,\ri\,m\,\theta_1(z|\tau)\big)\,
\theta\AB{\alpha+m}{\beta+n}(z|\tau)
\\
-\big\langle(\alpha{+}m)\big[{\ts\frac{\beta+n}{2}}\big] \big\rangle
\,\pi\, \vartheta\AB{\alpha+m}{\beta+n}^2 \cdot
\theta\AB{\alpha+m-1}{0}(z|\tau)\, \theta\AB{0}{\beta+n-1}(z|\tau)
\Big\}\,\frac{1}{\theta_1(z|\tau)}\,,
\end{multline*}
where, for closedness of the formula, identity \eqref{*} should be
taken into account.
\end{theorem}

\begin{proof}
It is a combination of equations \eqref{X} and conversion of any
$\theta_{\alpha\beta}$-function into the function $\theta_1$ by formula
\begin{equation}\label{theta1ab}
\theta_1(z|\tau)=\ri^\alpha\, \theta\AB{\alpha\sm1}{\beta\sm1}\!
\Big(z-\Mfrac{\alpha}{2}\tau-\Mfrac{\beta}{2}\Big|\tau\Big)
\cdot
\re^{\sm\pi\ri\alpha\left(z-\s{\frac14}\alpha \tau\right)}_
{\mathstrut}\,,
\end{equation}
wherein we set $(\alpha,\beta)$ to be integers $(-m,-n)$.
\end{proof}

Taking the limit at $z=0$, which exploits the series expansions
described above, we get a generalization of Jacobi's derivative
formula, \ie, second formula in \eqref{jacobi}.

\begin{corollary}\label{C2}
The general $\Dtheta[\alpha\beta]$-constant, \ie, value of
$\theta'$-function at any half-period,  is expressed through a
$\vartheta$-constant and exponential multiplier\/$:$
\begin{equation}\label{jac}
\Dtheta[\alpha\beta]\Big(\Mfrac n2+\Mfrac m2\,\tau
\Big|\tau\Big) =\ri^{\s 1-(\beta+n)m}\cdot \pi\, \mbig\{
\ri^{\s\beta+n}\big(1-\big\langle(\alpha{+}m)(\beta{+}n)\big\rangle
\big) \cdot\ded^3-m\, \vartheta\AB{\alpha+m}{\beta+n} \mbig\}
\:\re^{\!\sm\s{\frac14}\pi\ri m^2\tau}_{\mathstrut}\,.
\end{equation}
\end{corollary}

Since $\big\langle(\alpha{+}m)(\beta{+}n)\big\rangle$ is equal to
$\pm1$, only one term remains in the right-hand side of \eqref{jac},
\ie, $\vartheta\AB{\alpha+m}{\beta+n}$ or
$\ded^3=\frac12\,\vartheta_2\,\vartheta_3\,\vartheta_4$. For tables of
some particular cases see  \cite[{\bf II}:~p.~256]{tannery}.

\subsection{Modular transformations\label{S6.2}}

Transformations of $\theta$-functions with respect to  modular group
$\big(\begin{smallmatrix}a&b\\c&d\end{smallmatrix}\big)
\in\boldsymbol{\Gamma}(1)$ belong among fundamental properties of
theta-functions and have numerous applications \cite{rankin,apostol}.
Suffice it to mention that corresponding transformation of the series
$\theta_1$ \cite{weil,tannery,weber}:
\begin{equation}\label{ab11}
\theta_1\!\Big(
\Mfrac{z}{c\,\tau+d}\Big|\Mfrac{a\,\tau+b}{c\,\tau+d}\Big)=
\Aleph^3\cdot\sqrt{c\,\tau+d\:}\, \,\re_{\strut}^{\frac{\pi\ri
c\, z^2}{c\,\tau+d}}\,\theta_1(z|\tau)\,,
\end{equation}
where $\Aleph^3$ denotes some eighth root of unity, may turn  a
hyper-convergent series into the never computable one. Hermite
represented the famous multiplier $\Aleph^3$ via the sum of quadratic
Gaussian exponents (it is known that these sums are not easily
computed) and Jacobi's symbol $\big(\frac ab\big)$ \cite[{\bf
I}:~pp.~482--486]{hermite} (see also \cite[pp.~183--193]{krazer},
\cite[{\bf II}:~pp.~57--58]{koenig}, \cite[pp.~124--132]{weber},
\cite[{\bf II}]{tannery}). For this reason, it is interesting that
there is a simpler formula for the modular transformation wherein
multiplier $\Aleph$ is merely an exponent of a rational. Without loss
of generality we may normalize $c$ to be positive: $c>0$.

\begin{theorem}
[The $\boldsymbol{\Gamma}(1)$-transformation law for the general
$\theta$-function]\label{T7} Let $\theta\AB{\alpha}{\beta}$ be the
theta-series with arbitrary integer characteristics \eqref{hermite} and
let $n\in\mathbb{Z}$. Then
\begin{align}
\label{trivial}\theta\AB{\alpha\sm1}{\beta} (z|\tau+n)&=
\ri_{\mathstrut}^{\frac n2 \,(1{-}\alpha^2)}
\cdot \theta\AB{\alpha\sm1}{\beta+n\alpha} (z|\tau)\,,\\
\label{ab}
\theta\AB{\alpha{\s'}\sm1}{\beta{\s'}\sm1}
\Big(\Mfrac{z}{c\,\tau+d}\Big| \Mfrac{a\,\tau+b}{c\,\tau+d}\Big)&=
\boldsymbol{\mathfrak{E}}_{\alpha\beta}\:\Aleph^3\cdot
\sqrt{c\,\tau+d\:}\,\,
\re_{\strut}^{\frac{\pi\ri cz^2}{c\,\tau+d}}\,
\theta\AB{\alpha\sm1}{\beta\sm1}(z|\tau)\,,
\end{align}
where multipliers $\boldsymbol{\mathfrak{E}}_{\alpha\beta}$ and\/
$\Aleph$ depend on $(a,b,c,d)$ as
\begin{align}
\boldsymbol{\mathfrak{E}}_{\alpha\beta}={}&
\exp\,\Mfrac{\pi}{4_{\strut}}\ri\,\mbig\{
2\,\alpha\,(\beta\,b\,c-d+1)-\beta\,c\,(\beta\,a-2)-
\alpha^2\,d\,b\mbig\},\nonumber\\
\label{aleph}
\Aleph\DEF{}& \exp\pi\ri\, \mbig[11]\{
\Mfrac{a-d}{12\,c} -\Mfrac{d}{6}(2c-3)+\Mfrac{c-1}{4}\mathfrak{sign}
(d)-\Mfrac14+\Mfrac1c\cdot
\SUM{k}{\raisebox{-0.03em}{\mbox{\small\textup\textbar}}\mspace{-2.5mu}
{[c\mspace{-0.9mu}/\mspace{-2mu}d]\mspace{-2.5mu}
\raisebox{-0.03em}{\mbox{\small\textup\textbar}}\mspace{-1mu}
+\mspace{-1mu}1}}{{\scriptstyle c-1}}
\Big[\Mfrac{d}{c}k\Big]\,k
\mbig[11]\}
\end{align}
and characteristics $(\alpha,\beta)$, $(\alpha',\beta')$ are related
through the linear transformation
\begin{equation}\label{abAB}
\begin{alignedat}{6}
\alpha'&= &&d\,&\alpha &{}-{}c\,&\beta\,,\qquad\qquad
\alpha&{}=a\,&\alpha'&{}+c\,&\beta'\,,\\
\beta' &=-&&b\,&\alpha &{}+{}a\,&\beta\,,\qquad\qquad
\beta&{}=b\,&\alpha'&{}+d\,&\beta'\,.
\end{alignedat}
\end{equation}
\end{theorem}

\begin{proof}
Formula \eqref{trivial} is an elementary consequence of the series
\eqref{hermite}. Proof of \eqref{ab} consists of two steps. The first
is  to use accurate manipulations/simplifications by Dedekind's sums
\cite{rademacher,apostol} determining the multiplier $\Aleph$ and
entering into the transformation formula for the $\ded$-function:
$$
\ded\Big(\Mfrac{a\,\tau+b}{c\,\tau+d}\Big)=\Aleph\,
\sqrt{c\,\tau+d\,}\,\,\ded(\tau)\,.
$$
Subsequent use of the fact that multiplier for $\theta_1$ in
\eqref{ab11} is a cube of multiplier $\Aleph$ for $\ded$
\cite{tannery,weil} yields formula \eqref{aleph}. The second step
exploits the fact that function $\theta_1$ transforms into itself  and
any of the functions $\theta_{\!\alpha\beta}(z|\tau)$ can be
transformed into the function $\theta_1(z|\tau)$ (and vice versa) by a
half-period shift of its $z$-argument like \eqref{theta1ab}. This gives
the linear transformation between characteristics \eqref{abAB}, and
with it the multiplier $\boldsymbol{\mathfrak{E}}_{\alpha\beta}$.
Characteristics, as appeared in \eqref{ab}, have been chosen in order
that the formula be most  symmetric.
\end{proof}

\begin{remark}\label{R6}
Hermite  gave also \emph{nonlinear} formulae for transformation of
characteristics $(\alpha,\beta)\mapsto (\alpha',\beta')$ \cite[{\bf
I}:~p.~483]{hermite} which are reproduced in subsequent works
\cite{weber,farkas,rauch} (some linear forms can be found in
\cite[p.~183]{krazer}). Somewhat surprising facet is the fact that no
such a self-contained formula seems to have hitherto been presented in
the literature. Ratio of any $\theta$-functions contains no the
multiplier $\Aleph$ and Hermite used this fact to build the functions
$\varphi(\tau)$, $\psi(\tau)$, $\chi(\tau)$ and tables of
transformation between them \cite{hermite,tannery} when constructing
his famous solution to the quintic equation $x^5-x=a$ in terms of
$\varphi$, $\psi$, $\chi$ \cite[p.~10]{hermite}. These functions are in
fact certain $\vartheta$-constants so that their transformations are
consequences of  the $\vartheta$-constant ones.
\end{remark}

\begin{corollary}\label{C3}
The $\boldsymbol{\Gamma}(1)$-transformations for the general
$\vartheta\AB{\alpha}{\beta}$-constants are
\begin{align}
\vartheta\AB{\alpha\sm1}{\beta} (\tau+n)&= \ri_{\mathstrut}^{\s{\frac
n2}(1{-}\alpha^2)} \cdot
\vartheta\AB{\alpha\sm1}{\beta+n\alpha}
(\tau)\,,\notag\\
\vartheta\AB{\alpha{\s'}\sm1}{\beta{\s'}\sm1} \Big(
\Mfrac{a\,\tau+b}{c\,\tau+d}\Big)&=
\ri_{\strut}^{\s{\frac12}\{
2\,\alpha\,(\beta\,b\,c-d+1)-\beta\,c\,(\beta\,a-2)-\alpha^2\,d\,b\}}
\,\Aleph^3 \cdot \sqrt{c\,\tau+d\:}\,
\vartheta\AB{\alpha\sm1}{\beta\sm1}(\tau)\,.\notag
\end{align}
\end{corollary}

The known property for each $\theta_k$-function to transforms into
itself under the group $\boldsymbol{\Gamma}(2)$ is also a consequence
of Theorem~\ref{T7} and formulae \eqref{abAB}.

\begin{corollary}\label{C4}
Let $(m,n,p,q)$ be integers. Then group
$\boldsymbol{\Gamma}(2)\!\ni\!\big(\begin{smallmatrix}2n+1&2m\\2p&2q+1
\end{smallmatrix}\big)=\big(\begin{smallmatrix}a&b\\c&d
\end{smallmatrix}\big)$ is necessary and sufficient for each
function $\theta_k$ to transform into itself. The transformations are
\begin{alignat*}{2}
\theta_1\Big(\Mfrac{z}{c\,\tau+d}\Big|
\Mfrac{a\,\tau+b}{c\,\tau+d}\Big)={}&&\Aleph^3 \cdot
\sqrt{c\,\tau+d\:}\,\, \re_{\strut}^{\!\frac{\pi\ri cz^2}{c\,\tau+d}}\,
\theta_1(z|\tau)\,,\\
\theta_2\Big(\Mfrac{z}{c\,\tau+d}\Big|
\Mfrac{a\,\tau+b}{c\,\tau+d}\Big)={}&&\ri^ {2q(p-1)+p}\,\Aleph^3\cdot
\sqrt{c\,\tau+d\:}\,\, \re_{\strut}^{\!\frac{\pi\ri cz^2}{c\,\tau+d}}\,
\theta_2(z|\tau)\,,
\\
\theta_3\Big(\Mfrac{z}{c\,\tau+d}\Big|
\Mfrac{a\,\tau+b}{c\,\tau+d}\Big)={}&& \ri^{2q(p+1)-m(2n+1)+p}\,
\Aleph^3\cdot \sqrt{c\,\tau+d\:}\,\, \re_{\strut}^{\!\frac{\pi\ri
cz^2}{c\,\tau+d}}\, \theta_3(z|\tau)\,,
\\
\theta_4\Big(\Mfrac{z}{c\,\tau+d}\Big|
\Mfrac{a\,\tau+b}{c\,\tau+d}\Big)={}&& \ri^{2n(m-1)-m}\,\Aleph^3\cdot
\sqrt{c\,\tau+d\:}\,\, \re_{\strut}^{\!\frac{\pi\ri cz^2}{c\,\tau+d}}\,
\theta_4(z|\tau)\,.
\end{alignat*}
\end{corollary}

\begin{proof}
With use of \eqref{abAB} and \eqref{ab} we get
\begin{alignat*}{3}
\theta_2\Big(\Mfrac{z}{c\,\tau+d}\Big|
\Mfrac{a\,\tau+b}{c\,\tau+d}\Big)={}&&\re_{\strut}^
{\s{\frac14}\pi\ri\,(2-d)c}\,\Aleph^3\cdot \sqrt{c\,\tau+d\:}\,\,
\re_{\strut}^{\!\frac{\pi\ri cz^2}{c\,\tau+d}}\,
&&\theta\AB{c-1}{d-1}&(z|\tau)\,,
\\
\theta_3\Big(\Mfrac{z}{c\,\tau+d}\Big|
\Mfrac{a\,\tau+b}{c\,\tau+d}\Big)={}&& \re_{\strut}^{\s{\frac14}\pi\ri
\,\{2(a+c-ad)-ab-cd\}}\, \Aleph^3\cdot \sqrt{c\,\tau+d\:}\,\,
\re_{\strut}^{\!\frac{\pi\ri cz^2}{c\,\tau+d}}\,
&&\theta\AB{a+c-1}{b+d-1}&(z|\tau)\,,
\\
\theta_4\Big(\Mfrac{z}{c\,\tau+d}\Big|
\Mfrac{a\,\tau+b}{c\,\tau+d}\Big)={}&&
\re_{\strut}^{\s{\frac14}\pi\ri\,(2a-ab-2)}\,\Aleph^3\cdot
\sqrt{c\,\tau+d\:}\,\, \re_{\strut}^{\!\frac{\pi\ri cz^2}{c\,\tau+d}}\,
&&\theta\AB{a-1}{b-1}&(z|\tau)\,.
\end{alignat*}

Requiring now $\theta\AB{c-1}{d-1}\simeq\theta\AB{1}{0}=\theta_2$,
$\theta\AB{a+c-1}{b+d-1} \simeq\theta\AB{0}{0}=\theta_3$, and
$\theta\AB{a-1}{b-1}\simeq\theta\AB{0}{1}=\theta_4$, we obtain
\emph{linear} equations for $a,b,c,d$. Their solution yields the matrix
$\big(\begin{smallmatrix}a&b\\c&d
\end{smallmatrix}\big)=\big(\begin{smallmatrix}2n+1&2m\\2p&2q+1
\end{smallmatrix}\big)$.
\end{proof}

Transformation for the fifth basic function $\Dtheta$ follows from a
derivative of \eqref{ab11}:
$$
\Dtheta\!\Big(
\Mfrac{z}{c\,\tau+d}\Big|\Mfrac{a\,\tau+b}{c\,\tau+d}\Big)=
\Aleph^3\cdot\sqrt{c\,\tau+d\:}\, \,\re_{\strut}^{\frac{\pi\ri
cz^2}{c\,\tau+d}}\,\mbig\{(c\,\tau+d)\,\Dtheta(z|\tau)+
2\,\pi\,\ri\,c\,\,z\,\theta_1\!(z|\tau) \mbig\}\,.
$$
(Exercise: with use of Theorem~\ref{T6} derive the
$\boldsymbol{\Gamma}(1)$-transformation law for the general
$\Dtheta[\alpha\beta]$-function.) It is not difficult to see that
general transformation can always be brought to the form
$\theta_k\mapsto\theta_k$ if we involve the inhomogeneous
transformations of argument $z\mapsto\frac{z+s\,\tau+r}{c\,\tau+d}$.

\subsection{Differential equations}

The  next theorem describes completely  differential calculus of the
classical Jacobi $\vartheta,\theta,\theta'$-series in both variables
$z$ and $\tau$.
\begin{theorem}\label{T8}
Jacobi's $\theta_{\alpha\beta}$, $\theta'(z|\tau)$-series
\eqref{hermite} and \eqref{der} with arbitrary integer characteristics
$(\alpha,\beta)$, as functions of variables $z$ and $\tau$, are
differentially closed over the  field of $\eta(\tau)$- and
$\vartheta^2(\tau)$-constants. Corresponding rules for differentiating
are defined by the $(\alpha,\beta)$-representation of $z$-equations
\eqref{X} as
\begin{equation}
\begin{split}\label{x}
\frac{\partial\theta\AB{\alpha}{\beta}}{\partial z}&=
\frac{\Dtheta}{\theta_1} \,\theta\AB{\alpha}{\beta}-
({-}1)^{
\raise-0.02em\hbox{\small$[$}\s{\frac{\beta}{2}}
\raise-0.02em\hbox{\small$]$}
\,\alpha}_{\mathstrut} \, \pi\,\vartheta\AB{\alpha}{\beta}^2\cdot
\frac{\theta\AB{\alpha\sm1}{0}\, \theta\AB{0}{\beta\sm1}}{\theta_1}\,,
\\
\frac{\partial\Dtheta}{\partial z}&=
\frac{\Dtheta{}^2}{\theta_1}-\pi^2\, \vartheta_3^2\,\vartheta_4^2
\cdot \frac{\theta_2^2}{\theta_1}-
4\,\bigg\{\eta+\frac{\pi^2}{12}\big(\vartheta_3^4+ \vartheta_4^4\big)
\bigg\} \cdot\theta_1
\end{split}\hspace{7.45em}
\end{equation}
and  by the $(\alpha,\beta)$-equivalent of $\tau$-equations
\eqref{Dtau} as
\begin{equation}
\begin{split}\label{tau}
\frac{\partial\theta\AB{\alpha}{\beta}}{\partial \tau}&=
\frac{-\ri}{4\pi}\, \frac{\Dtheta{}^2}{\theta_1^2}
\,\theta\AB{\alpha}{\beta}+\frac{\ri}{2}\,
(-1)^{\raise-0.02em\hbox{\small$[$}\s{\frac{\beta}{2}}
\raise-0.02em\hbox{\small$]$}\,\alpha}_{\mathstrut}\,
\vartheta\AB{\alpha}{\beta}^2 \cdot \Dtheta\,
\frac{\theta\AB{\alpha\sm1}{0}\, \theta\AB{0}{\beta\sm1}}{\theta_1^2}
\\
&\==+\frac\ri4\pi\! \left\{\vartheta_3^2\,\vartheta_4^2
\cdot \theta_2^2- \mbig(\vartheta\AB{0}{\beta\sm1}^2\cdot
\theta\AB{\alpha\sm1}{0}^2 + \vartheta\AB{\alpha\sm1}{0}^2\cdot
\theta\AB{0}{\beta\sm1}^2 \mbig)\,\vartheta\AB{\alpha}{\beta}^2
\right\}\frac{\theta\AB{\alpha}{\beta}}{\theta_1^2}\\
&\==+\frac{\ri}{\pi}
\bigg\{\eta+\frac{\pi^2}{12} \big(\vartheta_3^4+
\vartheta_4^4\big)\bigg\}\cdot \theta\AB{\alpha}{\beta}\,,
\\
\frac{\partial \Dtheta}{\partial\tau}&=
\frac{-\ri}{4\pi}\,\frac{\Dtheta{}^3}{\theta_1^2}
+\frac{3\,\ri}{\pi}\bigg\{
\frac{\pi^2}{4}\,\vartheta_3^2\,\vartheta_4^2\cdot
\frac{\theta_2^2}{\theta_1^2} +\eta+
\frac{\pi^2}{12}\big(\vartheta_3^4+\vartheta_4^4\big)
\bigg\}\,\Dtheta \\&
\== - \frac\ri2\pi^2\,\vartheta_2^2\,\vartheta_3^2\,\vartheta_4^2
\cdot\frac{\theta_2\,\theta_3\,\theta_4} {\theta_1^2}\,.
\end{split}
\end{equation}
The constants $\eta(\tau)$, $\vartheta^2(\tau)$ form a differential
ring $\mathbb{C}_\partial[\eta,\vartheta^2]$ that is defined by the
following system of polynomial \odes\/$:$
\begin{equation}\label{last}
\begin{split}
\frac{d\vartheta\AB{\alpha}{\beta}}{d\tau}&= \frac{\ri}{\pi}\,
\bigg\{
\eta+\frac{\pi^2}{12}\,
\mbig((-1)^\beta\,\vartheta\AB{\alpha\sm1}{0}^4-
(-1)^\alpha\,\vartheta\AB{0}{\beta\sm1}^4\mbig) \bigg\}\,
\vartheta\AB{\alpha}{\beta}\,,\\
\frac{d\eta}{d\tau}&= \frac{\ri}{\pi}\, \bigg\{2\,\eta^2-
\frac{\pi^4}{72} \mbig( \vartheta\AB{\alpha}{0}^8+
(-1)^{\alpha+\beta}\, \vartheta\AB{\alpha}{0}^4
\vartheta\AB{0}{\beta}^4 +\vartheta\AB{0}{\beta}^8\mbig) \bigg\}\,,
\end{split}
\end{equation}
where $(\alpha,\beta)\ne (0,0)$ for the 2nd of these equations.
\end{theorem}

We comment now on  some connections of these dynamical systems with the
classical properties of the theta-series. There are two fundamental
algebraic relations between $\theta$-series, \ie, \eqref{j} and they
are of course compatible with equations \eqref{x} and \eqref{tau}
(proof is a calculation). However these relations are satisfied not
only by $\theta$-series themselves but by solutions of the equations as
well. As before, a simple calculation shows that corresponding
solutions contain three  constants $A$, $B$, $C$ and have the form
\begin{equation}\label{sol}
\begin{split}
\theta\AB{\alpha}{\beta}&=C\,\re^{\pi\ri A (2z+A\tau)}\cdot
\theta\AB{\alpha}{\beta}(z+A\,\tau+B|\tau)\,,\\
\Dtheta&= C\,\re^{\pi\ri A
(2z+A\tau)}\cdot\mbig\{\Dtheta(z+A\,\tau+B|\tau) -
2\,\pi\,\ri\,A\,\theta\AB{1}{1}(z+A\,\tau+B|\tau)\mbig\}\,.
\end{split}
\end{equation}

Another point that should be noticed is the important fact that the
heat equation \eqref{heat} must be treated as a \emph{corollary} of the
above equations, rather than the reverse, because
Eqs.~\eqref{x}--\eqref{tau} are the \emph{ordinary} differential
equations, while \eqref{heat} is an equation in \emph{partial}
derivatives. The heat equation has a lot of solutions having nothing to
do with theta-functions.  In order to extract some special, say
$\theta$-, solutions to this equation we must impose  additional
(periodic, differential, modular, etc) conditions on them. Once this
has been done for $\theta$-functions, we arrive at the \odes\  above,
so that the initial consideration with the heat equation may be dropped
out or `forgotten'. In other words, equations \eqref{x}--\eqref{tau}
should be put to better consider as fundamental differential properties
of theta-functions at all.

The third point we would like to mention here is the fact that modular
transformation considered in Sect.~\ref{S6.2} may be derived as a
consequence of these equations rather than  as an `internal property'
of theta-series themselves. This can be briefly outlined as follows.
Equations \eqref{x}--\eqref{tau} admit automorphisms which can be found
through some linear fractional ansatz. That this ansatz is a linear
fractional one can be determined with use of Lie's symmetries of these
equations. This  does not even use  the fact that solutions to
Eqs.~\eqref{x}--\eqref{tau} are the $\theta$-series. Moreover, the
availability of a discrete automorphism follows from the property that
each function $\theta_k$ satisfies a \emph{common \ode} (see
Sect.~\ref{S9.1}). An analogous property holds for the
$\vartheta$-constant equation \eqref{Jacobi}. By this means one can
find two basic transformations $\tau\mapsto\tau+1$ and $\tau\mapsto
-1/\tau$ generating group $\boldsymbol{\Gamma}(1)$. For lack of space
we omit proofs of these statements but they can be partially
compensated from the procedure of  integration of the equations which
shall be detailed in Sect.~\ref{S9.4}. Good examples of  Lie's
symmetries application to Chazy's equation \eqref{chazy} and many other
Jacobi's `modular/elliptic' equations are presented in  nice works
\cite{clarkson} and \cite{rosati}. As with solutions \eqref{j} and
\eqref{sol} we can write solutions to system \eqref{last} that respect
Jacobi's identity \eqref{324}. These are separate solutions to
Eqs.~\eqref{Jacobi} and \eqref{chazy} and they are of course known:
\begin{equation}\label{JacChaz}
\begin{split}
\vartheta\AB{\alpha}{\beta}&=\frac{1}{\sqrt{c\,\tau+d\,\,}}\cdot
\vartheta\AB{\alpha}{\beta}\Big(\Mfrac{a\,\tau+b}{c\,\tau+d}\Big)\,,
\qquad\quad\,
\qquad\hbox{(Jacobi \cite[{\bf II}:~pp.~186--187]{jacobi})}\\
\eta&=\frac{1}{(c\,\tau+d)^2}\cdot\eta
\Big(\Mfrac{a\,\tau+b}{c\,\tau+d}\Big)+ \frac12\frac{\pi\,\ri
\,c}{c\,\tau+d}\,,\quad\hbox{(Chazy \cite{chazy})}
\end{split}
\end{equation}
where $(a,b,c,d)$ are the integration constants, $a\,d-b\,c=1$, and
right-hand sides in these formulae are the $\vartheta,\eta$-series.

\section{Compatible integrability of
Eqs.~\protect\eqref{X}--\protect\eqref{Dtau}\label{S7}}

\noindent
An important corollary of the preceding section is that \odes\
satisfied by Jacobi's functions have a wider class of solutions  that
are not bound to be the canonical $\theta$-series \eqref{hermite} or
formulae \eqref{sol}. The latter  contain only three free constants
while equations \eqref{X} are of fifth order.  In applications,
variations of system \eqref{X} may occur in their own rights and hence
the quantities $\vartheta$'s, being parameters in Eqs.~\eqref{X}, are
not bound to be the values of $\theta$-series at zero. For the same
reason the quantity $\eta$ must not necessarily be given by any known
expression related to the $\vartheta,\theta$-series, \eg\
\cite{we2,tannery},
$$
\eta(\tau)=-\frac{1}{12}\,\frac{\theta_1\!\!\!\!'''(0|\tau)}
{\Dtheta(0|\tau)}=
-\frac{1}{4}\,\frac{\theta_3\!\!\!\!''(0|\tau)}
{\vartheta_3(\tau)}-
\frac{\pi^2}{12}\big\{\vartheta_2^4(\tau)-\vartheta_4^4(\tau)\big\}
=\cdots\,.
$$
Thus equations \eqref{X}--\eqref{Dtau} may serve as an independent
origin of the $\vartheta,\theta$-functions  at all since equations are
no less fundamental objects than their solutions.

\subsection{The $\theta$-identities as algebraic integrals\label{S7.1}}

Let us assume that $\eta$ and $\vartheta$ are the undetermined
quantities in equations \eqref{x}--\eqref{tau} or in
\eqref{X}--\eqref{Dtau}.

\begin{theorem}\label{T9}
Nonlinear equations \textup{\eqref{X}--\eqref{Dtau}} are compatible if
and only if their coefficients $\eta$, $\vartheta_k$  satisfy the
dynamical system
\begin{equation}\label{intA}
\left\{
\begin{aligned}
\frac{d\vartheta_2}{d\tau}&=
\frac{\ri}{\pi}\bigg\{\eta+\frac{\pi^2}{12}\,
\big(\vartheta_3^4+\vartheta_4^4 \big)\bigg\}\,\vartheta_2\\
\frac{d\vartheta_3}{d\tau}&=
\frac{\ri}{\pi}\bigg\{\eta+\frac{\pi^2}{12}\,
\big(\vartheta_3^4+\vartheta_4^4
-3\bB^4\,\vartheta_4^4\big)\bigg\}\,
\vartheta_3\\
\frac{d\vartheta_4}{d\tau}&=
\frac{\ri}{\pi}\bigg\{\eta+\frac{\pi^2}{12}\,
\big(\vartheta_3^4+\vartheta_4^4
-3\bA^4\,\vartheta_3^4\big)\bigg\}\,\vartheta_4
\\
\frac{d\eta}{d\tau}&=\frac{\ri}{\pi}\,2\,\eta^2-
\frac{\pi^3}{72}\,\ri\,\mbig\{\vartheta_3^8+
\big(9\bA^4\bB^4-6\bA^4-6\bB^4+2\big)\,\vartheta_3^4\,\vartheta_4^4+
\vartheta_4^8 \mbig\}\,,
\end{aligned}\right.
\end{equation}
where constants $\bA^4$ and $\bB^4$ are the algebraic $($rational\/$)$
integrals of the systems  \textup{\eqref{X}--\eqref{Dtau}}\/$:$
\begin{equation}\label{A12}
\bA^4\cdot\vartheta_3^2\,\theta_1^2=
\vartheta_2^2\,\theta_4^2-\vartheta_4^2\,\theta_2^2\,,\qquad
\bB^4\cdot\vartheta_4^2\,\theta_1^2=
\vartheta_2^2\,\theta_3^2-\vartheta_3^2\,\theta_2^2\,.
\end{equation}
The three functions $(\vartheta_3, \,\vartheta_4,\, \eta)$ are
differentially closed.
\end{theorem}

\begin{proof}
Considering compatibility condition $\theta_{z\tau}=\theta_{\tau z}$ of
systems \eqref{X} and \eqref{Dtau}, we obtain not only certain
restrictions on coefficients $\eta$, $\vartheta$ but also algebraic
relations between $\theta$'s. The straightforward check of \eqref{A12}
shows that $\bA,\bB(\theta;\vartheta)$ are the arbitrary constants
indeed:
$$
\frac{\partial\bA}{\partial z}=
\frac{\partial\bB}{\partial z}\equiv 0\,,\qquad
\frac{\partial\bA}{\partial \tau}=
\frac{\partial\bB}{\partial \tau} \equiv 0
$$
so that  relations \eqref{A12} do determine two independent algebraic
integrals.
\end{proof}

Relations \eqref{A12} generalize the well-known quadratic identities
between  canonical $\theta$-series
\begin{equation}\label{sign}
\mathfrak{sign}(\nu-\mu)\cdot\vartheta_k^2\,\theta_1^2=
\vartheta_\mu^2\,\theta_\nu^2-\vartheta_\nu^2\,\theta_\mu^2 \qquad
(k=2,3,4)
\end{equation}
only two of which, say \eqref{j}, are  independent ones
\cite{weber,mumford} since  $\vartheta$-constants satisfy the Jacobi
identity \eqref{324}. This suggests that there exists one more
integral; this is indeed the case.

\begin{theorem}\label{T10}
The system  \eqref{intA} has  the algebraic $($rational\/$)$ integral
${\boldsymbol{\mathfrak{A}}}^4(\vartheta)$\/$:$
\begin{equation}\label{jacobiA}
{\boldsymbol{\mathfrak{A}}}^4\,\vartheta_2^4=
\bA^4\vartheta_3^4-\bB^4\vartheta_4^4
\qquad \hence\qquad\frac{d}{d\tau}{\boldsymbol{\mathfrak{A}}}
\equiv 0.
\end{equation}
\end{theorem}

Integrals \eqref{A12} and \eqref{jacobiA}, as generalizations of the
famous Jacobi relations \eqref{sign} and \eqref{324}, mean that the
various polynomial $\theta$-identities, \eg, \eqref{sign}, are the
additional constraints to the basic equations \eqref{X}--\eqref{Dtau}.
We do not dwell on degenerations of system \eqref{intA} into elementary
functions and consider only a generic situation describing a
non-canonical version of $\theta$-functions. The quantities $\bA$,
$\bB$ define initial conditions to Eqs.~\eqref{X}--\eqref{Dtau} and are
parameters for  system \eqref{intA}.

\subsection{Canonical $\theta$-series and elliptic
functions\label{S7.2}}

Let us explain how reduction to the canonical case of
Jacobi--Weierstrass is performed. This procedure discloses some
interesting facts.

Put $\bA=\bB=1$. Then, as it follows from \eqref{intA}, function $\eta$
satisfies the Chazy equation \eqref{chazy} and functions
$\vartheta_{3,4}$ do the Jacobi equation \eqref{Jacobi}. Integral
\eqref{jacobiA} still remains  to be free. Putting further
${\boldsymbol{\mathfrak{A}}}=1$, we can rewrite equations/identities
\eqref{intA} into the symmetrical form \eqref{var}. This procedure,
however, changes not only the structure of equations but their
algebraic integral as well.

\begin{proposition}\label{P1}
Algebraic integral ${\boldsymbol{\mathfrak{A}}}(\vartheta)$ of
symmetrical equations \eqref{var} has the form
\begin{equation}\label{A3}
({\boldsymbol{\mathfrak{A}}}^4-1)\cdot
\vartheta_2^4\,\vartheta_3^4\,\vartheta_4^4=
(\vartheta_3^4-\vartheta_2^4-\vartheta_4^4)^3\,.
\end{equation}
\end{proposition}

\begin{remark}\label{R7}
This identity should be treated as a correct form of the `complete
Jacobi identity' if determining \odes\  for quantities $\vartheta$,
$\eta$ have been used in  the symmetrized form \eqref{var}. It is
particularly remarkable that if we consider algebraic integrals
\eqref{jacobiA} and \eqref{A3} as algebraic curves in projective
coordinates $\vartheta_2:\vartheta_3:\vartheta_4$, then  we  find that
curve \eqref{jacobiA} has genus three while \eqref{A3} is a curve of
genus nineteen! In addition to this complication under
${\boldsymbol{\mathfrak{A}}}\ne1$, none of the functions $\eta$,
$\vartheta_{2,3,4}$, or logarithmic derivatives $\ln_{\tau}\!\vartheta$
satisfies any equation of 3rd order, as it was for equations
\eqref{Jacobi}--\eqref{halphen}. These assertions can be proved with
use of polynomial Gr\"obner bases techniques over variables
$\eta,\vartheta,\dot\eta,\dot\vartheta,\ldots$ \cite{Cox} but we omit
complete proofs for reasons of space. Broadly speaking, we lose a
differential closedness of the three functions $\vartheta_3$,
$\vartheta_4$, and $\eta$.
\end{remark}

From the aforementioned it appears that we must do away with the rules
of differentiations computations in the symmetrical form \eqref{var},
\eqref{last} and redefine them according to Eqs.~\eqref{intA} under
$\bA=\bB=1$. Some heuristic arguments lead to the following formulae.

\begin{proposition}
Integrable rules of differentiating the $\vartheta,\eta$-constants are
as follows. The $\eta$-derivative reads
\begin{align*}
\frac{d\eta}{d\tau}&=\frac{\ri}{\pi}\,\bigg\{2\,\eta^2-
\frac{\pi^4}{72}\,\mbig(\vartheta_3^8-
\vartheta_3^4\,\vartheta_4^4+
\vartheta_4^8 \mbig)\bigg\}\\
\intertext{and $\vartheta$-constants are differentiated as}
\frac{d\vartheta\AB{\alpha}{\beta}}{d\tau}&= \frac{\ri}{\pi}\,
\bigg\{
\eta+\frac{\pi^2}{24}\,
\mbig{[}
\big(2-3\spin{\alpha}-3\spin{\beta}\big)\,\vartheta_4^4-
\big(1-3\spin{\beta}\big)\,\vartheta_3^4\mbig] \bigg\}\,
\vartheta\AB{\alpha}{\beta}\\
\intertext{or, equivalently,}
\frac{d\vartheta_k}{d\tau}&= \frac{\ri}{\pi}\,
\bigg\{
\eta+\frac{\pi^2}{24}\,
\mbig{[}2\,(3\,k^2-18\,k+25)\,\vartheta_4^4-
(3\,k^2-15\,k+16)\,\vartheta_3^4\mbig] \bigg\}\,
\vartheta_k
\end{align*}
under $k=1,2,3,4$ $($recall $\vartheta_1\equiv0$\/$)$ and, as before,
arbitrary integral $(\alpha,\beta)$.
\end{proposition}

Let us denote  $\boldsymbol{\mathrm{P}}\DEF\theta_2^2/\theta_1^2$. Then
we derive that the following identity holds:
\begin{equation}\label{wp}
\pow{\boldsymbol{\mathrm{P}}}{\!\!z}{2}=4\,\pi^2
\big(\vartheta_4^2\cdot
\boldsymbol{\mathrm{P}}+\bA^4\vartheta_3^2\big)
\big(\vartheta_3^2\cdot
\boldsymbol{\mathrm{P}}+\bB^4\vartheta_4^2\big)\,
\boldsymbol{\mathrm{P}}\,.
\end{equation}
Therefore $\boldsymbol{\mathrm{P}}$ is expressible in terms of
Weierstrass' $\wp$-function which is proportional to the ratio of
Jacobi's $\theta$-series by  formula \eqref{wp2}:
$$
\wp(2z|\btau)=\frac{\pi^2}{12}
\bigg\{\vartheta_3^4(\btau)+\vartheta_4^4(\btau)+
3\,\vartheta_3^2(\btau)\,
\vartheta_4^2(\btau)\,
\frac{\theta_2^2(z|\btau)}{\theta_1^2(z|\btau)}\bigg\}.
$$
From this point on we shall use notation $\vartheta_{2,3,4}(\btau)$,
$\theta_k(z|\btau)$ for explicit pointing out the canonical $\vartheta,
\theta$-series, their modulus $\btau$, and argument $z$. The same
symbols without arguments will denote dynamical variables entering into
our \odes. Bringing equation \eqref{wp} into the form of Weierstrass'
cubic and applying the standard technique \cite{hancock,WW}, we obtain
that elliptic modulus $\btau$ for equation \eqref{wp} is determined as
a root of the transcendental equation (we recall \eqref{Klein})
\begin{equation}\label{J}
J(\btau)=\frac{1}{54}
\frac{\big({\boldsymbol{\mathfrak{A}}}^8\vartheta_2^8
+\bA^8\vartheta_3^8+\bB^8\vartheta_4^8\big)^3}
{{\boldsymbol{\mathfrak{A}}}^8\bA^8\bB^8\cdot
\vartheta_2^8\,\vartheta_3^8\,\vartheta_4^8}\,.
\end{equation}
Here, for symmetry, we have used integral \eqref{jacobiA}. Assuming for
the moment that $\btau$ has been determined, one derives that the
following formula for the ratio $\boldsymbol{\mathrm{P}}$ must exist:
\begin{equation*}
\begin{split}
\frac{\theta_2^2}{\theta_1^2}&=
\frac{\vartheta_3^4(\btau)+\vartheta_4^4(\btau)}
{3\,\vartheta_3^2\,\vartheta_4^2}
-\frac{\bA^4\vartheta_3^4+\bB^4\vartheta_4^4}
{3\,\vartheta_3^2\,\vartheta_4^2} +
\frac{\vartheta_3^2(\btau)\,\vartheta_4^2(\btau)}
{\vartheta_3^2\,\vartheta_4^2}\cdot
\frac{\theta_2^2(z+z_{\s0}|\btau)}{\theta_1^2(z+z_{\s0}|\btau)}
\\
&=\frac{4}{\pi^2}\,\frac{\wp\mbig[1](2\,(z+z_{\s0})|\btau\mbig[1])}
{\vartheta_3^2\vartheta_4^2}-
\frac{\bA^4}{3}\,\frac{\vartheta_3^2}{\vartheta_4^2}-
\frac{\bB^4}{3}\,\frac{\vartheta_4^2}{\vartheta_3^2}\,.
\end{split}
\end{equation*}
Analogous formulae can be obtained for the quotients
$\theta_k/\theta_j$ without squares. It is not difficult to see that
such a variation would lead to the Jacobi elliptic functions
$\mathrm{sn}\sim \theta_1/\theta_4$, etc:
\begin{equation}\label{14}
\bigg(\frac{\theta_1}{\theta_4}\bigg)_{\!\!\!z}^{\!\!\!2} =\pi^2
\bigg\{\bA^4\vartheta_3^2\cdot
\mbig[7](\frac{\theta_1}{\theta_4}\mbig[7])^{\!\!\!2}
-\vartheta_2^2\bigg\}
\bigg\{{\boldsymbol{\mathfrak{A}}}^4\vartheta_2^2\cdot
\mbig[7](\frac{\theta_1}{\theta_4}\mbig[7])^{\!\!\!2} -
\vartheta_3^2 \bigg\}.
\end{equation}
We thus infer  that ratio of any two $\theta$-solutions to
Eqs.~\eqref{X}--\eqref{Dtau} is proportional to a ratio of canonical
$\theta$-series with new modulus $\btau$ determined from Eq.~\eqref{J}.
Before proceeding to further integration, we need a closed formula
solution to the problem of finding that modulus.

\section{Modular inversion problem and related topics\label{S8}}

\noindent As mentioned in Introduction,  no explicit formula
realization of the scheme \eqref{Klein} is hitherto available if
elliptic curve has been given in Weierstrassian form
\begin{equation}\label{weab}
y^2=4\,x^3-a\,x-b
\end{equation}
or in a more general form
\begin{equation}\label{ell4}
y^2=a_0^{}\,x^4+4\,a_1^{}\,x^3+6\,a_2^{}\,x^2+4\,a_3^{}\,x+a_4^{}\,.
\end{equation}
Analytic solution to the problem is known only for the canonical
Legendre form
\begin{equation}\label{legendre}
y^2=(1-x^2)(1-k^2 x^2)\,.
\end{equation}
In this case it is given by the famous formula of Jacobi
\begin{equation}\label{KK'}
\tau=\ri\,\frac{\ds\ellK'(k)}{\ellK(k)}\,,
\end{equation}
where $\ellK$ and $\ds\ellK'$ are complete elliptic integrals
\cite{jacobi,bateman}. By virtue of classical formula
$$
k^2=\frac{\vartheta_2^4(\tau)}{\vartheta_3^4(\tau)}
$$
this solution implies the identity
\begin{equation}\label{t2}
\tau\equiv\ri\,\frac{{\ds\ellK'}\!\mbig[5](
\frac{\vartheta_2^2(\tau)}{\vartheta_3^2(\tau)}
\mbig[5])}{\ellK\mbig[5](\frac{\vartheta_2^2(\tau)}
{\vartheta_3^2(\tau)}\mbig[5])} \mod\; \boldsymbol{\Gamma}(2)
\qquad\forall\tau\in\Hp\,.
\end{equation}

Elliptic curves are, however, parametrized by group
$\boldsymbol{\Gamma}(1)$, not by $\boldsymbol{\Gamma}(2)$. Moreover,
transitions between \eqref{weab}, \eqref{ell4}, and \eqref{legendre}
requires knowledge of roots of the $x$-polynomials \eqref{ell4} or
\eqref{weab} (see \cite[13.5]{bateman}) and modulus $\tau$ is computed
via ratios of certain hypergeometric series. Owing to the fact that
$_2F_1(J)$-series converges only inside the unity circle, the
$_2F_1$-solutions will differ in structure depending on whether $|J|>1$
or $|J|<1$.  For instance, in Weierstrassian representation
\eqref{weab} the resulting formulae constitute rather cumbersome
expressions and, in addition to that, they involve the series in
logarithmic derivative of Euler's $\Gamma$-function. See, \eg,
\cite[pp.~27--28]{kleinMA}, one-half-page-long collection of formulae
(22)--(27) in Sect.~14.6.2 of book \cite{bateman}, or enumeration of
all the particular cases in \cite[{\bf I}:~pp.~341--348]{halphen};
though more compact hypergeometric form to solution was obtained by
Bruns \cite{bruns}. Such forms of solutions are not convenient in
applications since they can not be manipulated analytically. Meanwhile
the problem has an elegant solution.

\subsection{Analytic formula solution}

Modulus $\tau$ depends only on the value of absolute
$\boldsymbol{\mathrm{J}}$-invariant which in turn is computed via
coefficients $a$'s through the two  invariants \cite{hancock,a,weber}
\begin{equation}\label{g23a}
\g2=a_0^{}a_4^{}-4\,a_1^{}a_3^{}+3\,a_2^2\,,\qquad \g3=
\mbig[11]|
\begin{matrix}
a_0^{}&\!\!a_1^{}&\!\!a_2^{}\\
a_1^{}&\!\!a_2^{}&\!\!a_3^{}\\
a_2^{}&\!\!a_3^{}&\!\!a_4^{}
\end{matrix}\mbig[11]|
\end{equation}
according to  Klein's definition
\begin{equation}\label{Jg23}
\boldsymbol{\mathrm{J}}=\frac{\g2{}^{\hspace{-0.4em}3}}
{\g2{}^{\hspace{-0.4em}3}-27\,\g3{}^{\hspace{-0.4em}2}}\,.
\end{equation}
It is well known that  function $J(\tau)$ is related to a
hypergeometric equation of the form \cite{indus}
\begin{equation}\label{Jode}
J(J-1)\,\psi''+\frac16\,(7\,J-4)\,\psi'+\frac{1}{144}\,\psi=0
\end{equation}
and generic solutions of such equations are usually designated as
$_2F_1(\alpha,\beta;\gamma|z)$. However, under some restrictions on
parameters $(\alpha,\beta,\gamma)$ the solutions are representable in
terms of known special functions; this point occurs when the
$_2F_1$-series admits a quadratic transformation \cite{bateman2}. In
this case $_2F_1$-equation reduces to an equation with two parameters,
\eg, to Legendre's equation \cite{abramowitz,WW,bateman2}. This is just
the case of the $\psi$-equation~\eqref{Jode}. Its solution is a linear
combination
$$
\psi=\sqrt[6]{J\,}\,\big\{A\,P_\nu^\mu(\sqrt{1-J\,})+
B\,Q_\nu^\mu(\sqrt{1-J\,})\big\}
$$
of Legendrian  functions with parameters
$(\nu,\mu)=\big({-}\frac12,\frac13\big)$. Recall that functions
$P_\nu^\mu(z)$, $Q_\nu^\mu(z)$ are independent solutions of the linear
equation
\begin{equation}\label{Leg}
(1-z^2)\,\psi''-2\,z\,\psi'+\Big\{\nu(\nu+1)-
\Mfrac{\mu^2}{1-z^2}\Big\}\,\psi=0\,.
\end{equation}
See \cite[Ch.~3]{bateman2} for definitions and exhaustive properties of
these functions. From the aforesaid it appears that  formula
\begin{equation*}\label{ratio}
\tau=\frac{\boldsymbol{a}\,P_\nu^\mu(\sqrt{1-J\,})+\boldsymbol{b}\,
Q_\nu^\mu(\sqrt{1-J\,})} {\boldsymbol{c}\,P_\nu^\mu(\sqrt{1-J\,})+
\boldsymbol{d}\,Q_\nu^\mu(\sqrt{1-J\,})}
\end{equation*}
must hold, where parameters $(\boldsymbol{a}, \boldsymbol{b},
\boldsymbol{c}, \boldsymbol{d})$ have definite numeric values. In order
to find them, we need only any three values of $\tau$ under which the
quantity $J(\tau)$ has known exact values. There are lot of such points
and all of them correspond to tori with complex multiplication
\cite{weber}. For instance
$$
J(\ri)=1\,,\qquad J\Big(\Mfrac12+\ri\,\Mfrac12\,\sqrt{3}
\Big)=0\,,\qquad
J\big(\sqrt{2}\,\ri\big)=\frac{5^3}{3^3}\,,\qquad\text{etc.}
$$
The asymptotic property $P(z)/Q(z)\to \infty$  as $z\to\infty$ implies
that $\boldsymbol{c}=0$ since $J(\ri\infty)=\infty$. It therefore
suffice  to consider only two simplest points $J=\{0,1\}$ and
corresponding values of functions $P,Q(\sqrt{1-J\,})$ are easily
computed. We obtain these values with use of formulae (9)--(10) and
table cases (22), (40) in Sect.~3.2 of  \cite{bateman2}.

\begin{theorem}[Weierstrassian analog of \eqref{KK'}]\label{T11}
For elliptic curve in Weierstrassian form \eqref{weab} its modulus
$\tau=\om'\!/\omega$, \ie, solution of the transcendental equation
$$
J(\tau)=\frac{a^3}{a^3-27\,b^2}\,,
$$
is given by the expression
\begin{equation}\label{PQ}
\tau=\ri\,\frac{{\pow{P}{\!\!\sm\s1\!/\!6}{\,\,\s0}}
\big({-}\sqrt{\boldsymbol{\mathfrak{g}}\,}\big)}
{\pow{P}{\!\!\sm\s1\!/\!6}{\,\,\s0}
\big(\sqrt{\boldsymbol{\mathfrak{g}}\,}\big)}\,,\qquad
\boldsymbol{\mathfrak{g}}\DEF27\,\frac{b^2}{a^3}\,.
\end{equation}
If  curve has the generic form \eqref{ell4} then
$\boldsymbol{\mathfrak{g}}=1-\boldsymbol{\mathrm{J}}^{\sm1}$ is
computed according to \eqref{g23a}--\eqref{Jg23}.
\end{theorem}

\begin{proof}
Let us use relations between Legendrian equations \eqref{Leg} with
different indices. More precisely, we carry out the above mentioned
quadratic transformation
$$
z\mapsto J=\frac{a\,z^2+b}{c\,z^2+d}
$$
and demand that the normal form of Eq.~\eqref{Jode} is preserved, that
is,
$$
\tilde\psi''=-\frac{1}{12^2}\,\frac{36\,J^2-41\,J+32} {(J-1)^2J^2}\,
\tilde\psi\,.
$$
The only possibilities we then have are
$$
z^2=1-J\,,\qquad(\nu,\,\mu)=\Big({-}\frac12,\:\pm\frac13\Big)
$$
and
$$
z^2=1-\frac{1}{J}\,,\quad(\nu,\,\mu)=\Big({-}\frac16,\:0\Big)
\quad\text{or}
\quad(\nu,\,\mu)=\Big({-}\frac56,\:0\Big)\,.
$$
Coefficients $(\boldsymbol{a}, \boldsymbol{b}, \boldsymbol{c},
\boldsymbol{d})$ are derived as said above. For the first case we
obtain
\begin{equation}\label{ver1}
\tau =\bigg\{\pi\,\ri\,\frac{\pow{P}{\nu}{\mu}}{\pow{Q}{\nu}{\mu}}
\big(\sqrt{1-J}\big)-1
\bigg\}\,\re^{\!\frac{\pi}{3}\ri}_{\mathstrut}\,,
\end{equation}
where $(\nu,\mu)=\big({-}\frac12,\frac13\big)$. A simpler version comes
from the second cases, \eg, from case
$(\nu,\mu)=\big({-}\frac16,0\big)$. Functions $(P,Q)$ therewith are
interchanged. As before, their proportionality coefficient is derived
with use of tables in Sect.~3.2 of \cite{bateman2}. The last step is to
use formula 3.2(10) of \cite{bateman2} to express a certain linear
combination of $P(z)$ and $Q(z)$ through a single $P(-z)$; see
\cite[(8.2.3)]{abramowitz}. Under our parameters this relation reads
$$
\pi\,\pow{P}{\!\!\sm\s1\!/\!6}{\,\,\s0}(-z)=
\pi\re^{\frac\pi6\boldsymbol{s}\ri}_{\mathstrut}
\pow{P}{\!\!\sm\s1\!/\!6}{\,\,\s0}(z)+
\pow{Q}{\!\sm\s1\!/\!6}{\,\,\s0}(z)\,,\qquad\boldsymbol{s}\DEF
\mathfrak{sign}(\boldsymbol\Im z)\,.
$$
Multivalued functions $P$ and $Q$ are defined such that this identity
holds for $\boldsymbol\Im z\lessgtr 0$ (see also
\cite[3.3.1(10)]{bateman2}) but this restriction disappears when
passing to the ratio $P/Q$ and the answer simplifies into the ultimate
formula \eqref{PQ} under arbitrary $a$, $b$.
\end{proof}

The result \eqref{ver1} was announced recently in \cite{br1} and used
there in connection with a nontrivial application to the soliton theory
when considering a linear spectral problem of the form
$\Psi'''+u(x)\,\Psi'-\frac12\,u'(x)\,\Psi=\lambda\,\Psi$.

The function $\sqrt{1-J(\tau)}$ is a single-valued one and therefore
Eq.~\eqref{ver1} entails an interesting analog of  identity \eqref{t2}.

\begin{corollary}\label{C5}
For all $\tau\in\Hp$ the following $\boldsymbol{\Gamma}(1)$-analog of
identity \eqref{t2} holds\/$:$
$$
\tau \equiv\bigg\{\pi\,\ri\,\frac{\pow{P}{\nu}{\mu}}{\pow{Q}{\nu}{\mu}}
\Big(\ri\,\mfrac{\sqrt{27}}{\pi^6}\mfrac{\g3(\tau)}
{\ded^{12}(\tau)}\Big)-1
\bigg\}\,\re^{\!\frac{\pi}{3}\ri}_{\mathstrut} \mod\;
\boldsymbol{\Gamma}(1)
$$
under  $(\nu,\mu)=\big({-}\frac12,\frac13\big)$.
\end{corollary}

\subsection{Consequences}

An interrelation between \eqref{PQ} and Jacobi's formula \eqref{KK'}
needs to be understood  if the elliptic curve \eqref{ell4} has already
been given in the canonical form \eqref{legendre} as
$$
y^2=(1-x^2)(1-\varkappa^2 x^2)\,.
$$
Its absolute $J$-invariant is determined not by the classical
expression
\begin{equation}\label{J1}
J=\frac{4}{27}\,\frac{(k^4-k^2+1)^3}{k^4\,(k^2-1)^2}\,,
\end{equation}
wherein $k^2=\varkappa^2$, but by expression of the form
\begin{equation}\label{J2}
J=\frac{1}{108}\,\frac{(\varkappa^4+14\,\varkappa^2+1)^3}
{\varkappa^2\,(\varkappa^2-1)^4}\,.
\end{equation}
In the standard Legendre--Jacobi theory of equation \eqref{legendre}
\cite{WW,tannery,weber,halphen,we2} the function $x$ is proportional to
Jacobi's  $\mathrm{sn}$-function, whereas Weierstrass' $\wp$ is
proportional to the square of $\mathrm{sn}$; hence $\wp\rightleftarrows
\mathrm{sn}$ is not a \emph{birational} transformation. Function $x$
also solves equation \eqref{14} under
$\bA={\boldsymbol{\mathfrak{A}}}=1$ and hence has periods $2\,\ri
\,\ds\ellK'$ and $4\,\ellK$, so their ratio is equal to $\frac12\tau$
rather than $\tau$. To be more precise, one can carry out some standard
calculations and derive birational transformations between Weierstrass'
$\{\wp,\wp'\}$ and the $\theta$-ratios, \ie, Jacobi's basis
\begin{alignat*}{1}
\mathrm{sn}\big(\pi\,\vartheta_3^2(\tau)\,z;k\big)&=
\frac{\vartheta_3(\tau)}{\vartheta_2(\tau)}\cdot
\frac{\theta_1(z|\tau)}{\theta_4(z|\tau)}\,,\\
\mathrm{cn}\big(\pi\,\vartheta_3^2(\tau)\,z;k\big)&=
\frac{\vartheta_4(\tau)}{\vartheta_2(\tau)}\cdot
\frac{\theta_2(z|\tau)}{\theta_4(z|\tau)}\,,\\
\mathrm{dn}\big(\pi\,\vartheta_3^2(\tau)\,z;k\big)&=
\frac{\vartheta_4(\tau)}{\vartheta_3(\tau)}\cdot
\frac{\theta_3(z|\tau)}{\theta_4(z|\tau)}\,.
\end{alignat*}

\begin{proposition}[Inversion of \eqref{wp2}]
Every homogeneous $\theta(z|\tau)$-ratio and consequently Jacobi's
functions $\{\mathrm{sn}, \mathrm{cn}, \mathrm{dn}\}$ are rationally
represented via Weierstrass' $\{\wp,\wp'\}$-functions. The three basic
$\theta$-ratios read
\begin{alignat*}{1}
\frac{\theta_2(z|\tau)}{\theta_1(z|\tau)}&=
-\frac{\vartheta_2(\tau)}{\vartheta_1'(\tau)}\cdot
\frac{\wp'(2z|2\tau)}{\wp(2z|2\tau)-\ep(2\tau)}\\
&= 2\, \frac{\vartheta_2(\tau)}{\Dvartheta(\tau)} \cdot\Big\{
\zeta(2z|2\tau)- \zeta(2z\!-\!2\tau|2\tau)-\etap(2\tau)
\Big\}\,,\\
\frac{\theta_3(z|\tau)}{\theta_1(z|\tau)}&=
-\frac{\vartheta_3(\tau)}{2\,\Dvartheta(\tau)}\cdot
\frac{\wp'\big(z\big|{\ts\frac{\tau+1}{2}}\big) }{\wp
\big(z\big|{\ts\frac{\tau+1}{2}}\big)
-e\big({\ts\frac{\tau+1}{2}}\big)}\\
&= \frac{\vartheta_3(\tau)}{\Dvartheta(\tau)} \cdot\Big\{
\zeta\mbig[3](z\big|{\ts\frac{\tau+1}{2}}\mbig[3])-
\zeta\mbig[3](z\!-\!1\big|{\ts\frac{\tau+1}{2}}\mbig[3])-
\eta\mbig[3]({\ts\frac{\tau+1}{2}}\mbig[3]) \Big\}\,,\\
\frac{\theta_4(z|\tau)}{\theta_1(z|\tau)}&=
-\frac{\vartheta_4(\tau)}{2\,\Dvartheta(\tau)}\cdot
\frac{\wp'\big(z\big|{\ts\frac{\tau}{2}}\big) }{\wp
\big(z\big|{\ts\frac{\tau}{2}}\big)
-e\big({\ts\frac{\tau}{2}}\big)}\\
&= \frac{\vartheta_4(\tau)}{\Dvartheta(\tau)} \cdot\Big\{
\zeta\mbig[3](z\big|{\ts\frac{\tau}{2}}\mbig[3])-
\zeta\mbig[3](z\!-\!1\big|{\ts\frac{\tau}{2}}\mbig[3])-
\eta\mbig[3]({\ts\frac{\tau}{2}}\mbig[3]) \Big\}
\end{alignat*}
and nine other ones are obtained by the three half-period shifts
$z\mapsto z+\frac12\,\big\{1,\tau,\tau+1\big\}$.
\end{proposition}

Half-moduli on right hand sides of these equations explain the
`distinction' between Weierstrass' modulus for $(\zeta,\wp,\wp')$ and
Jacobi's one for $\mathrm{sn}$. By this we mean that the transition
$\varkappa^2\rightleftarrows k^2$, \ie, transformation
\eqref{J1}$\rightleftarrows$\eqref{J2}, is realized through a
duplication of modulus:
$$
\varkappa^2=k^2(2\,\tau)\,.
$$
To put it differently, the map $\tau\mapsto2\,\tau$ is a one-to-one
transformation---it is just a normalization of $\tau\in\Hp$---and every
elliptic curve is uniquely determined by the value $\tau$ (or
$2\,\tau$). We might of course work with modulus $2\,\tau$ instead of
$\tau$, however, in this case, classical integral representation of
fundamental group $\boldsymbol{\Gamma}(1)$ must be changed to the
matrices  $\big(\begin{smallmatrix}2&0\\0&1
\end{smallmatrix}\big)
\boldsymbol{\Gamma}(1)\big(\begin{smallmatrix}2&0\\0&1
\end{smallmatrix}\big)^{\!\sm1}$ (generically non-integral).
As far as we know, this precise correlation between Weierstrass' and
Jacobi's modular inversions has not been mentioned in the literature
\cite{WW,enneper,we,we2,tannery,weber,hancock,a,lawden,armitage,
bateman,halphen}.

One further comment is in order. Transformation of the curve
\eqref{ell4} from general form into the Weierstrass one (and vice
versa) is performed through the linear fractional change of variable
$x$ \cite{indus,weber,hancock,bateman}. This requires knowledge of
roots of $x$-polynomial \eqref{ell4}, \ie, solution of a quartic
equation, however. This is not convenient in investigation if
coefficients of the polynomial do not have definite numerical values
but are parameters. For this reason it would be useful to have a
transformation  over the field of coefficients
$\mathbb{C}(a_0^{},\ldots, a_4^{})$, \ie, \emph{without resorting to
solution of any equations}. Such a birational change does indeed exist
and a version of it is shown below. To simplify formulae we make a
trivial transformation bringing \eqref{ell4} to a shorten form with
$(a_{\s0},a_{\s1})=(1,0)$.

\begin{proposition}\label{P2}
The elliptic curve
\begin{equation}\label{wz}
y^2=x^4-6\,\alpha\,x^2+4\,\beta\,x+\gamma
\end{equation}
is equivalent to the canonical Weierstrass form through a birational
change over $\mathbb{C}(\alpha,\beta)$ $($no $\gamma$ here\/$)$.
Corresponding Weierstrass' cubic has the form
$$
\boldsymbol{w}^2=4\,\boldsymbol{z}^3-(3\,\alpha^2+\gamma)\,
\boldsymbol{z}-(\alpha^3-\gamma\,\alpha-\beta^2)
$$
and the transformation between  these curves reads as follows\/$:$
\begin{alignat}{3}
\boldsymbol{z}&=\frac12\,(x^2-y-\alpha)\,, & \qquad\qquad
x&=\frac12\,\frac{\boldsymbol{w}-\beta}{\boldsymbol{z}-\alpha}\,,
\notag\\
\boldsymbol{w}&=x^3-y\,x-3\,\alpha\,x+\beta\,, &
y&= \frac14\,\frac{(\boldsymbol{w}-\beta)^2}
{(\boldsymbol{z}-\alpha)^2}-2\,\boldsymbol{z}-\alpha \notag\\
&&&=\frac{\beta}{2}\,
\frac{\beta-\boldsymbol{w}}{(\boldsymbol{z}-\alpha)^2}
+\frac14\,\frac{9\,\alpha^2-\gamma}{\boldsymbol{z}-\alpha}-
\boldsymbol{z}+\alpha \notag\,.
\end{alignat}
Legendrian form \eqref{legendre} corresponds in these formulae to the
substitution
$$
y \dashrightarrow \frac{y}{k}\,,\qquad
(\alpha,\beta,\gamma)=
\bigg(\frac{k^2+1}{6\,k^2},0,\frac{1}{k^2}\bigg).
$$
\end{proposition}

\begin{proof}
The variable $x$, as a function on the curve \eqref{wz}, has two simple
poles. If $\mathfrak{u}$ denotes a uniformizer for \eqref{wz} then we
may place these poles at points $\mathfrak{u}=\{0,v\}$ and hence write
\begin{equation}\label{zp}
\begin{aligned}
x&=\zeta(\mathfrak{u};\g2,\g3)- \zeta(\mathfrak{u}-v;\g2,\g3)- C\,,
\\
y&=\wp(\mathfrak{u}-v;\g2,\g3)-\wp(\mathfrak{u};\g2,\g3)\,,
\end{aligned}
\end{equation}
that is $y=\frac{dx}{d\mathfrak{u}}$. Let us manage parameters $\g2$,
$\g3$, $v$, and $C$ to turn \eqref{wz}, \eqref{zp} into an identity in
$\mathfrak{u}$. A calculation with use of $\zeta,\wp$-addition theorems
yields the constant  $C$ and mutual computability of parameters
$(\alpha,\beta,\gamma)\leftrightarrows(v,\g2,\g3)$. We obtain that
$C=\zeta(v;\g2,\g3)$ and
\begin{gather}
\alpha=\wp(v;\g2,\g3)\,,\quad \beta=\wp'(v;\g2,\g3)\,,\quad
\gamma=\g2-3\,\wp^2(v;\g2,\g3)\,,\notag\\\label{abc}
v=\wp^{\sm1}(\alpha;\g2,\g3)\,,\quad \g2=3\,\alpha^2+\gamma\,,\qquad
\g3=\alpha^3-\gamma\,\alpha-\beta^2\,.\notag
\end{gather}
This expressions follow from the fact that  pair $(\alpha,\beta)$ lies
on the curve $\beta^2=4\,\alpha^3-\g2\,\alpha-\g3$. This also implies
the transcendental version
$$
y^2=x^4-6\,\wp(v)\,x^2+4\,\wp'(v)\,x+ \big\{\g2-3\,\wp^2(v)\big\}
$$
of algebraic equation \eqref{wz} itself. The birational isomorphism
stated in the proposition is then just a $(\wp,\wp')$-version of
Eqs.~\eqref{zp} under notation $\boldsymbol{z}\DEF\wp(\mathfrak{u})$
and $\boldsymbol{w}\DEF\wp'(\mathfrak{u})$.
\end{proof}

In this connection it should be mentioned Weierstrass'
$\wp,\wp'$-formulae in \cite[{\bf I}: pp.~118--120]{halphen},
\cite[{\bf V}]{we2}, \cite{a}, \cite[{\bf IV}: pp.~66--67]{tannery} and
in the inaugural dissertation by Biermann
\cite[\S\,1]{biermann}\footnote{A student to whom Weierstrass leaved
the problem as an exercise; it became an introductory clause of
\cite{biermann}.} wherein the problem of transition between
Weierstrass' cubic and quartic equation \eqref{wz} was posed for the
first time. Complete form to Biermann's birational transformations has
been given in \emph{Exs}.~2--3 of \cite[20$\boldsymbol{\cdot}$6]{WW}.
Formulae presented in \cite[p.~6]{biermann} and
\cite[20$\boldsymbol{\cdot}$6]{WW} are, however, rather complicated, so
their simplest form (supplemented by a parameters computation) is given
by Proposition~\ref{P2}, formulae \eqref{zp}, and Theorem~\ref{T11}. It
is know that transformations between two forms of one elliptic curve
may contain a free parameter and we can readily introduce it into
Proposition~\ref{P2}. To do this it is sufficient to change
$\mathfrak{u}\mapsto\mathfrak{u}-\mathfrak{u}_{\s0}$ and make use of
addition theorems in (re)definitions
$\boldsymbol{z}\DEF\wp(\mathfrak{u}-\mathfrak{u}_{\s0})$ and
$\boldsymbol{w}\DEF\wp'(\mathfrak{u}-\mathfrak{u}_{\s0})$. Clearly, the
$\mathfrak{u}_{\s0}$ is more convenient quantity  than the `algebraic
parameters' entering into the Biermann--Whittaker--Watson formulae.

The immediate (and known \cite{halphen,a}) consequence of the technique
above is an application to  the cubic and quartic equations. With use
of Theorem~\ref{T11} we can present their roots in a completely closed
and analytic  form. Solutions to cubic $4\,x^3-a\,x-b=0$ are obvious;
these are the Weierstrassian points
$x_\kappa=\big\{e(\tau),\ep(\tau),\epp(\tau)\big\}$ under
$\tau=\tau(a,b)$ as above.

\begin{corollary}\label{C6}
Closed and radical-free formula for roots $x_\kappa$ of the quartic
equation
$$
x^4-6\,\alpha\,x^2+4\,\beta\,x+\gamma=0
$$
reads as follows
$$
x_\kappa=2\,\zeta\Big( \Mfrac12\,v+\omega_\kappa\Big)-
\zeta(v+2\,\omega_\kappa)\,,\qquad
\omega_\kappa=\om\cdot\{0,1,\tau,\tau+1 \}\,.
$$
Here,
$$
v=\int\limits_\infty^{\;\alpha}\!\!\!\frac{dz}
{\sqrt{4z^3-(3\,\alpha^2+\gamma)\,z-\alpha^3+\gamma\,\alpha+\beta^2}}
\,,\qquad
\om^2=
\frac{3\,\alpha^2+\gamma}{\alpha^3-\gamma\,\alpha-\beta^2}\cdot
\frac{\g3(\tau)}{\g2(\tau)}\,,
$$
$$
\tau=\ri\,\frac{{\pow{P}{\!\!\sm\s[5]{1\!/\!6}}{\,\,\s[5]0}}
{\big(}{-}\sqrt{\boldsymbol{\mathfrak{g}}\,}\big)}
{\pow{P}{\!\!\sm\s[5]{1\!/\!6}}{\,\,\s[5]0}
\big(\sqrt{\boldsymbol{\mathfrak{g}}\,}\big)}\,,\qquad
\boldsymbol{\mathfrak{g}}=
27\,\frac{(\alpha^3-\gamma\,\alpha-\beta^2)^2}
{(3\,\alpha^2+\gamma)^3}
$$
and the arbitrary value of the Weierstrass elliptic integral is taken.
\end{corollary}

These formulas provide analytic single-valued expressions for roots
$x_\kappa$ as functions of coefficients and, thereby, the problem of
multi-valuedness and jumps between roots does not appear here in
contrast to the standard (cumbersome) radical-type formulae.

To summarize briefly, we may conclude that both the transformation
between elliptic curves and modular inversion do not require any
auxiliary constructions; in each case the ultimate answer is given by
an explicit analytic formula independently of Weierstrassian,
Legendrian, or general representation \eqref{ell4}. Hence, solution to
the problem \eqref{J} may be thought of as completely solved and we
return to integration of the basic equations \eqref{X}--\eqref{Dtau}
and \eqref{intA}.

\section{Noncanonical $\theta$-functions\label{S9}}

\subsection{Exponential quadratic extension of
$\theta$-functions\label{S9.1}}

The system of equations \eqref{X} ($\hhence$\eqref{x}) has fifth order,
whereas Weierstrassian base of functions $(\sigma,\zeta,\wp)$ has the
order three. In the canonical case $\bA=\bB=1$ one can easily derive
the differential equation satisfied by any of canonical Jacobi's
$\theta$-functions:
\begin{equation}\label{w}
\pow{F}{z}{2}= -4\,\bigg\{F+
{\frac{\pi^2}{3}}(\vartheta_3^4+\vartheta_4^4)\bigg\} \bigg\{F+
{\frac{\pi^2}{3}}(\vartheta_2^4-\vartheta_4^4)\bigg\} \bigg\{F-
{\frac{\pi^2}{3}}(\vartheta_2^4+\vartheta_3^4)\bigg\},
\end{equation}
where $F\DEF\ln_{\mathit{zz}}\! \theta_k(z|\tau)+4\,\eta$ and
$\vartheta=\vartheta(\tau)$, $\eta=\eta(\tau)$.

In the non-canonical case $\bA\ne 1\ne \bB$ this equation must have an
analog in form of some differential equation of 5th order. It is not
difficult to obtain from \eqref{x} that each $\theta$-solution of
Eqs.~\eqref{x} satisfies one equation
\begin{equation}\label{F}
\begin{array}{c}
\ds F^2F_{\mathit{zzz}}-2\,F F_z F_{\mathit{zz}}+\pow Fz3+
(F^4)_z=0\,,\\[2ex]
\ds F\DEF(\ln\theta)_{\mathit{zz}}^{}+4\,\bigg\{\eta+\frac{\pi^2}{12}
(\vartheta_3^4+\vartheta_4^4)\bigg\}\,.
\end{array}
\end{equation}
This important equation, as  equation of 3rd order for function $F$, is
the generalization of  a 2nd order differential consequence of the
canonical Weierstrassian equation $\pow Fz2=4\,F^3-\g2F-\g3$, that is
equation $F_{\mathit{zzz}}=12\,FF_z$. The latter \emph{is not a
reduction} of Eq.~\eqref{F}, although equation \eqref{F} is also solved
by $\wp$-function:
$$
F=\wp(\om|\om,\om')-\wp(z+c|\om,\om')
$$
with free constants $(\om,\om',c)$. The distinction between them lies
in the fact that as long as we do not require the differential closure
of $\theta$'s, it is sufficient to use one Weierstrassian equation
\eqref{w}. In both of these cases periods $(2\,\om,2\,\om')$ and
modulus $\tau$ appear as integration constants. Equation \eqref{F} is
easily integrated if we rewrite it in the form
\begin{equation}\label{alg}
\bigg(\frac{1}{F_z} \mbig[7](\frac{\pow
Fz2}{F}\mbig[7])_{\!\!z}\bigg)_{\!\!z} \!+8\,F_z=0\,.
\end{equation}
We thus derive the complete integral of equations \eqref{x}, whatever
parameters $\vartheta$, $\eta$. Let
$$
M\DEF\varkappa^2\bigg\{\eta(\btau)+ \frac{\pi^2}{12}
\big[\vartheta_3^4(\btau)+\vartheta_4^4(\btau)\big]\!\bigg\}
-\bigg\{\eta+\frac{\pi^2}{12} \big[\vartheta_3^4+\vartheta_4^4\big]
\bigg\}\,.
$$
Then some routine computations yield the following result.

\begin{theorem}\label{T12}
Differential equations \eqref{X} and \eqref{x} have the following
general solution\/$:$
\begin{equation}\label{Xgeneral}
\begin{alignedat}{4}
\pm\,\theta_1={}&&\frac{\vartheta_2\vartheta_3\vartheta_4}
{2\,\ded^3(\btau)}\cdot{}&
C\,\theta_1(\varkappa\,z+B|\btau)\:\re^{2M(z+A)^2}\,,\\
\pm\,\theta_2={}&&\frac{\varkappa\,\vartheta_2}{\vartheta_2(\btau)}
\cdot{}&
C\,\theta_2(\varkappa\,z+B|\btau)
\:\re^{2M(z+A)^2} \,,\\
\pm\,\theta_3={}&&
\frac{\varkappa\,\vartheta_3}{\vartheta_3(\btau)}\cdot{}
&C\,\theta_3(\varkappa\,z+B|\btau) \:\re^{2M(z+A)^2}\,,
\\
\pm\,\theta_4={}&&\frac{\varkappa\,\vartheta_4}
{\vartheta_4(\btau)}\cdot{}
&C\,\theta_4(\varkappa\,z+B|\btau) \:\re^{2M(z+A)^2}\,,
\\
{\pm}\,\Dtheta={}&&\frac{\vartheta_2\vartheta_3\vartheta_4}
{2\,\ded^3(\btau)}\cdot{}&
C\,\big\{ \varkappa\,\Dtheta(\varkappa\,z+B|\btau)+
4M(z+A)\,\theta_1(\varkappa\,z+B|\btau)\big\}\:\re^{2M(z+A)^2}\,,
\end{alignedat}
\end{equation}
where $\{A,B,C,\varkappa,\btau\}$ is a complete set of integration
constants and signs $\pm$ may be freely changed for arbitrary pair
$(\theta_j,\theta_k)$.
\end{theorem}

Solution \eqref{Xgeneral} shows that its dependence on $\varkappa$ and
$\btau$ is rather nontrivial in comparison with its dependence on
constants $B$, $C$ and the linear exponent $\re^{Az}$; these are easily
`guessable' in solution \eqref{sol}. An additional point to emphasize
is that dependence of both equation \eqref{F} and its solution
\eqref{Xgeneral} on parameters $(\vartheta, \eta)$ is represented,
omitting the trivial multiplicative constant $C$, through the one
essential parameter
\begin{equation}\label{param}
\frac14\,\Lambda=\eta+
\frac{\pi^2}{12}\big(\vartheta_3^4+\vartheta_4^4\big)\,.
\end{equation}
Using  \eqref{Xgeneral}, we can, after some algebra, rewrite  integrals
\eqref{A12} in the `$(\varkappa,\btau)$-representation'.

\begin{corollary}[Generalization of Jacobi's identities]
Noncanonical $\vartheta,\theta$-functions satisfy the  identities
\begin{equation}\label{omega}
\vartheta_2^2\,\theta_4^2-\vartheta_4^2\,\theta_2^2=\varkappa^2
\frac{\vartheta_3^4(\btau)}{\vartheta_3^4}\cdot
\vartheta_3^2\,\theta_1^2\,, \qquad
\vartheta_2^2\,\theta_3^2-\vartheta_3^2\,\theta_2^2=\varkappa^2
\frac{\vartheta_4^4(\btau)}{\vartheta_4^4}\cdot
\vartheta_4^2\,\theta_1^2\,.
\end{equation}
\end{corollary}

The canonical Weierstrass--Jacobi case is defined by the restriction
$\bA=\bB=1$ and, therefore, is equivalent to conditions on constants
$(\btau,\varkappa)$: $\vartheta(\btau)=\vartheta$, $\eta(\btau)=\eta$,
and $\varkappa=\pm1$.

It should be particularly emphasized the following. Notwithstanding the
fact that non-canonical case is realized through the elementary
function---the quadratic exponent $\re^{2Mz^2}$---it depends highly
non-trivially on constants $(\varkappa,\btau)$ and generates a
\emph{transcendental extension} since canonical $\sigma$- and
$\theta$-functions are defined up to a linear exponent by formulae
\eqref{sol}. Dependence of the extension on parameters $\vartheta$ and
$\eta$ is also nontrivial. Quasi-periodicity properties of the
$\Dtheta,\theta$-extensions, \ie, analogs of formulae \eqref{shift},
are readily  established from Eqs.~\eqref{Xgeneral} and
Theorem~\ref{T6}; we do not display them here.

\begin{remark}\label{R8}
A brief mention of the quadratic exponential multiplier in front of
Jacobian function $\Theta(u)$ can be found in \cite[pp.~156,
189]{cayley1}. Such a function was also considered by Jacobi himself;
the pages 307--318 of his \cite[{\bf I}]{jacobi} are devoted to  study
of the object $\chi(u)=\re^{r\,uu}\,\Omega(u)$, where
$\int_{\s0}^u\!E(u)\,du=\log\Omega(u)$ under the standard Jacobi's
notation for $E(u)$ and $r$ has a special (not generic) value. The
first appearance of the quadratic exponent can be found even in
\emph{Fundamenta Nova} \cite[{\bf I}:~p.~226]{jacobi}. In connection
with certain differential identities and heat equation for
$\theta$-functions this exponent appears also in \cite{ohyama}.
\end{remark}

Before going further we pause to comment on the nature of the
integrability of Eqs.~\eqref{X}.

\subsection{Algebraic integrability of \odes\ for
$\theta$-functions\label{S9.2}}

Equation \eqref{alg} discloses an interesting feature of canonical and
non-canonical $\theta$-series. Let us define the term algebraic
integrability as a property of differential equations that have
solutions in terms of finitely many integrals of algebraic functions.

\begin{theorem}\label{T13}
Differential equations \eqref{X} are algebraically integrable upon
adjoining an inversion of integrals.
\end{theorem}

\begin{proof}
By virtue of \eqref{alg} we may write
\begin{equation}\label{holo}
\int\limits^{\;F}\!\!
\frac{dx}{\sqrt{x\,(x-\boldsymbol{a})(x-\boldsymbol{b})}}=
2\,\ri\,z+\boldsymbol{c}\qquad \hence\qquad
F=\Xi(z;\boldsymbol{a},\boldsymbol{b},\boldsymbol{c})\,,
\end{equation}
where $\boldsymbol{a},\boldsymbol{b},\boldsymbol{c}$ are some
integration constants. Integration is thus completed if we introduce
the inversion operation $\Xi$:
\begin{equation}\label{necanon}
\theta=\exp\! \int\limits^{\;z}\!\!\mbig[7]\{
\mbox{\footnotesize$\ds\int\limits^{\;x}$}
\Xi(y;\boldsymbol{a},\boldsymbol{b},\boldsymbol{c})\,dy\mbig[7]\}\, dx
\cdot
\re_{\mathstrut}^{\sm\s{\frac12}\Lambda\,z^2+\boldsymbol{d}\,z+
\boldsymbol{e}}\,,
\end{equation}
where $ \boldsymbol{d},\boldsymbol{e}$ are new constants. The inversion
function $\Xi$ here is of course not understood to be a ratio of
$\theta$-series. Integration to the $\Dtheta$  is obvious.
\end{proof}

The two-fold integration of the inversion operation in \eqref{necanon}
can now be reduced to a 1-fold integration of the algebraic function
and this leads to a meromorphic integral rather than the holomorphic
one as in \eqref{holo}. By this means we obtain the following
nonstandard way of introducing the theta-function.

\begin{corollary}\label{C7}
The canonical $\theta$-series, along with its non-canonical extension,
can be defined through a meromorphic elliptic integral.
\end{corollary}

To prove this it will suffice to make the following change in formula
\eqref{necanon}:
$$
\int\limits^{\;z}\!
\Xi(y;\boldsymbol{a},\boldsymbol{b},\boldsymbol{c})\,dy=
\!\!\int\limits^{\,\Xi(z;\boldsymbol{a},\boldsymbol{b},\boldsymbol{c})}
\hspace{-1em}
\frac{z\,dz}{\sqrt{z\,(z-\boldsymbol{a})(z-\boldsymbol{b})}}\,.
$$

To all appearances, Tikhomandritski\u\i\  \cite{tikh2} was the first to
point out a way  of definition of $\theta$ (different from described
above) through a meromorphic integral but his old note \cite{tikh2}
went unnoticed in the literature.

\begin{theorem}\label{T14}
Dynamical systems defining non-canonical extensions of
$\theta$-functions \textup{\eqref{X}--\eqref{Dtau}} and
$\vartheta$-constants \eqref{intA} are Hamiltonian. They admit the
gradient flow forms $\dot X=\Omega\,\nabla\!\!\mathscr{H}(X)$ with
Poisson brackets $\Omega=\Omega(X)$ which may not be the constant ones.
\end{theorem}

This theorem enhances results on algebraic integrability
(Theorem~\ref{T13}) but its proof and consequences will be given
elsewhere because the systems under question require the
even-dimensional extensions; these are non-obvious in advance. Here, we
give just an example, that, on the one hand, generalizes Hamiltonicity
of the canonical version of Eqs.~\eqref{intA} described in
\cite[Theorem~13]{br3} and, on the other hand, is a rational subcase of
a pencil of the brackets found in the same work. Let $\mathscr{U}$,
$\mathscr{V}$, and $\mathscr{W}$ be defined as three vector field
components for system \eqref{intA} as follows:
$$
\begin{aligned}
\mathscr{U}(\vartheta,\eta) &
\DEF \frac{\ri}{\pi}\bigg\{\eta+\frac{\pi^2}{12}\,
\big(\vartheta_3^4+\vartheta_4^4 -3\bB^4\vartheta_4^4\big)\bigg\}\,
\vartheta_3\,,\\
\mathscr{V}(\vartheta,\eta) &
\DEF \frac{\ri}{\pi}\bigg\{\eta+\frac{\pi^2}{12}\,
\big(\vartheta_3^4+\vartheta_4^4
-3\bA^4\vartheta_3^4\big)\bigg\}\,\vartheta_4\,,
\\
\mathscr{W}(\vartheta,\eta) &\DEF \frac{\ri}{\pi}\,2\,\eta^2-
\frac{\pi^3}{72}\,\ri\,\mbig\{\vartheta_3^8+
\big(9\bA^4\bB^4-6\bA^4-6\bB^4+2\big)\,\vartheta_3^4\,\vartheta_4^4+
\vartheta_4^8 \mbig\}
\end{aligned}
$$
and let the $\boldsymbol{\mathfrak{A}}$-integral \eqref{jacobiA} be
taken as a Hamilton function. Then one can show and verify by a
straightforward computation the following result.

\begin{theorem}
The system~\eqref{intA} admits the gradient flow  form
$$
\frac{d}{d\tau}\!\!
\Smaller[2]{\left\lgroup\Larger[2]{\vbox to 6ex{}}\right.}
\hspace{-0.6em}
\begin{array}{c}\vartheta_2\\ \vartheta_3\\\vartheta_4\\\eta
\end{array}\hspace{-0.7em}
\Smaller[2]{\left.\Larger[2]{\vbox to 6ex{}}\right\rgroup}
=\thinunderbrace{\frac{\vartheta_2}{4\mathscr{H}}
\Smaller[2]{\left\lgroup\Larger[2]{\vbox to 6ex{}}\right.}
\hspace{-0.6em}
\begin{array}{lccc}
\phantom{-}\,0&\mathscr{U}&\mathscr{V}&\mathscr{W}\\
-\mathscr{U}&0&0&0\\-\mathscr{V}&0&0&0\\-\mathscr{W}&0&0&0\\
\end{array}
\hspace{-0.4em}
\Smaller[2]{\left.\Larger[2]{\vbox to 6ex{}}\right\rgroup\!\!\!\!}%
}{\Omega(\vartheta,\eta)}
\Smaller[2]{\left\lgroup\Larger[2]{\vbox to 6ex{}}\right.}
\hspace{-0.6em}
{\begin{array}{l}
\mathscr{H}_{\vartheta_2}\\\mathscr{H}_{\vartheta_3}\\
\mathscr{H}_{\vartheta_4}\\\mathscr{H}_{\eta}
\end{array}}
\hspace{-0.7em}
\Smaller[2]{\left.\Larger[2]{\vbox to 6ex{}}\right\rgroup}
$$
under the Hamilton function
$$
\mathscr{H}(\vartheta_2,\vartheta_3,\vartheta_4,\eta)=
\bA^4\frac{\vartheta_3^4}{\vartheta_2^4}-
\bB^4\frac{\vartheta_4^4}{\vartheta_2^4}\,.
$$
Corresponding Poisson bracket is degenerated
$(\det\Omega(\vartheta,\eta)\equiv0)$ but single-valued.
\end{theorem}

In  \cite{br4} we also showed that the algebraic integrability of
$\theta$-functions may be treated as a Liouvillian extension of certain
differential fields. It therefore has an intimate connection with the
Picard--Vessiot solvability of spectral problems defined by linear
\odes.

\subsection{Renormalization of $\theta$-functions\label{S9.3}}

Solution \eqref{Xgeneral} suggests us to make the following
renormalization $\theta\mapsto\boldsymbol\theta$:
$$
\boldsymbol\theta_1=\theta_1\,, \qquad
\boldsymbol\theta_2=\pi\,\vartheta_3\,\vartheta_4
\cdot\theta_2\,, \qquad
\boldsymbol\theta_3=\pi\,\vartheta_2\,\vartheta_4
\cdot\theta_3\,, \qquad
\boldsymbol\theta_4=\pi\,\vartheta_2\,\vartheta_3
\cdot\theta_4\,.
$$
Then equations \eqref{X} and \eqref{F} will contain the single
parameter \eqref{param}. We get
\begin{equation}\label{Xbold}
\hspace{-0.6em}
\begin{array}{c}
\ds\frac{\partial\boldsymbol\theta_1}{\partial
z}=\boldsymbol\Dtheta\,, \qquad
\frac{\partial\boldsymbol\Dtheta}{\partial z}=
\frac{\boldsymbol\Dtheta{}^2}{\boldsymbol\theta_1}-
\frac{\boldsymbol\theta_2^2}{\boldsymbol\theta_1}- \Lambda
\cdot\boldsymbol\theta_1\,.
\\[3ex]
\ds\frac{\partial\boldsymbol\theta_2}{\partial z}=
\frac{\boldsymbol\Dtheta}{\boldsymbol\theta_1}\,\boldsymbol\theta_2-
\frac{\boldsymbol\theta_3\boldsymbol\theta_4}{\boldsymbol\theta_1}\,,
\qquad
\frac{\partial\boldsymbol\theta_3}{\partial z}=
\frac{\boldsymbol\Dtheta}{\boldsymbol\theta_1}\,\boldsymbol\theta_3-
\frac{\boldsymbol\theta_2\boldsymbol\theta_4}{\boldsymbol\theta_1}\,,
\qquad \frac{\partial\boldsymbol\theta_4}{\partial z}=
\frac{\boldsymbol\Dtheta}{\boldsymbol\theta_1}\,\boldsymbol\theta_4-
\frac{\boldsymbol\theta_2\boldsymbol\theta_3}{\boldsymbol\theta_1}\,.
\end{array}\hspace{-1em}
\end{equation}
It immediately follows  that the $\tau$-dependence, that is
$\tau$-differentiating the $\theta$-functions \eqref{Dtau}, is also
simplified. Putting for simplicity $\tau=4\pi\ri\,t$, we obtain the
system of equations
\begin{equation}\label{Dtaubold}
\begin{aligned}
\frac{\partial\boldsymbol\theta_1}{\partial\, t}&=
\frac{\boldsymbol\Dtheta{}^2}
{\boldsymbol\theta_1}-\frac{\boldsymbol\theta_2^2}
{\boldsymbol\theta_1}- \Lambda\cdot\boldsymbol\theta_1\,, \qquad
\frac{\partial\boldsymbol\Dtheta}{\partial\, t}=
\frac{\boldsymbol\Dtheta{}^3} {\boldsymbol\theta_1^2}-
3\big(\boldsymbol\theta_2^2+
\Lambda\cdot\boldsymbol\theta_1^2\big)
\,\frac{\boldsymbol\Dtheta}{\boldsymbol\theta_1^2}+
2\,\frac{\boldsymbol\theta_2\boldsymbol\theta_3 \boldsymbol\theta_4}
{\boldsymbol\theta_1^2}\,,\\
\frac{\partial\boldsymbol\theta_2}{\partial\, t}&=
\frac{\boldsymbol\Dtheta{}^2}{\boldsymbol\theta_1^2}\,
\boldsymbol\theta_2-2\,\boldsymbol\Dtheta\,
\frac{\boldsymbol\theta_3\boldsymbol\theta_4}
{\boldsymbol\theta_1^2}-
\big(\boldsymbol\theta_2^2-\boldsymbol\theta_3^2-
\boldsymbol\theta_4^2\big)\,
\frac{\boldsymbol\theta_2}{\boldsymbol\theta_1^2} -
\mbig\{\Lambda-\ln_t^{}(\vartheta_3\vartheta_4)
\mbig\}\cdot\boldsymbol\theta_2\,,\\
\frac{\partial\boldsymbol\theta_3}{\partial\, t}&=
\frac{\boldsymbol\Dtheta{}^2}
{\boldsymbol\theta_1^2}\,\boldsymbol\theta_3-
2\,\boldsymbol\Dtheta\,
\frac{\boldsymbol\theta_2\boldsymbol\theta_4}
{\boldsymbol\theta_1^2}+\boldsymbol\theta_4^2
\,\frac{\boldsymbol\theta_3}{\boldsymbol\theta_1^2} -
\mbig\{\Lambda-\ln_t^{}(\vartheta_2\vartheta_4)\mbig\}\cdot
\boldsymbol\theta_3\,,\\
\frac{\partial\boldsymbol\theta_4}{\partial\, t}&=
\frac{\boldsymbol\Dtheta{}^2}
{\boldsymbol\theta_1^2}\,\boldsymbol\theta_4-
2\,\boldsymbol\Dtheta\,
\frac{\boldsymbol\theta_2\boldsymbol\theta_3}
{\boldsymbol\theta_1^2}+\boldsymbol\theta_3^2\,
\frac{\boldsymbol\theta_4}{\boldsymbol\theta_1^2}-
\mbig\{\Lambda-\ln_t^{}(\vartheta_2\vartheta_3)\mbig\}\cdot
\boldsymbol\theta_4\,.
\end{aligned}
\end{equation}

With use of these equations the mechanism of integrability for
$\theta$-functions (and analysis at all) becomes very simple. Another
point that should be emphasized here is an asymmetry of equations. It
manifests the fact that we may not use Jacobi's polynomial
$\theta$-identities \eqref{sign} in advance since that identities are
just particular integrals of Eqs.~\eqref{Xbold} and \eqref{Dtaubold};
the latter are constructed from the heat equation
${\boldsymbol\theta}_t=\boldsymbol\theta_{\!\mathit{zz}}$ by
definition. Furthermore, equations \eqref{Dtaubold} contain not
variables $\vartheta$ but their logarithmic derivatives. Because of
this, compatibility conditions
$\boldsymbol\theta_{tz}=\boldsymbol\theta_{\!zt}$ for these equations
will be 1) \emph{algebraic} relations between functions
$\boldsymbol\theta$ over field of coefficients $\Lambda$,
$\ln_t^{}\!\vartheta$ and, on the other hand, 2) the only differential
relation containing $\Lambda_t^{}\FED\dot{\mathrm\Lambda}$. A
computation yields
\begin{equation}\label{Lambda}
\begin{array}{c}\ds
\frac{\dot\vartheta_2}{\vartheta_2}+\Lambda=0\,, \qquad
\frac{\dot\vartheta_3}{\vartheta_3}+\Lambda=
\frac{\boldsymbol\theta_3^2-\boldsymbol\theta_2^2}
{\boldsymbol\theta_1^2}\,, \qquad
\frac{\dot\vartheta_4}{\vartheta_4}+\Lambda=
\frac{\boldsymbol\theta_4^2-\boldsymbol\theta_2^2}
{\boldsymbol\theta_1^2}\,,
\\[3ex] \ds
\dot{\mathrm\Lambda}-2\,\bigg(\frac{\dot\vartheta_3}{\vartheta_3}+
\frac{\dot\vartheta_4}{\vartheta_4}\bigg)\Lambda-2\,
\frac{\dot\vartheta_3}{\vartheta_3}\,
\frac{\dot\vartheta_4}{\vartheta_4}=0
\end{array}
\end{equation}
and we see that parameter $\ln_t^{}\vartheta_2$ is not an independent
one but enters into the theory fictitiously, through $\Lambda$. The
first of Eqs.~\eqref{Lambda} is in effect the first equation in
\eqref{intA}. Right hand sides of 2nd and 3rd of Eqs.~\eqref{Lambda}
are functions of $z$ but their left hand sides are of $t$. This leads
again to the algebraic integrals \eqref{A12}:
$$
\boldsymbol\theta_3^2-\boldsymbol\theta_2^2=
\bB^4\vartheta_4^4\cdot \boldsymbol\theta_1^2\,,\qquad
\boldsymbol\theta_4^2-\boldsymbol\theta_2^2=\bA^4
\vartheta_3^4\cdot \boldsymbol\theta_1^2
$$
and  2nd and 3rd  equations  in \eqref{intA}. The fourth equation is
obvious. An important point here is the fact that the primary object of
the theory---compatibility condition---manifests itself not as the
symmetrical equations \eqref{var} but as the non-symmetrical ones
\eqref{intA}. Furthermore, because systems
\eqref{Xbold}--\eqref{Dtaubold} are nonlinear ones, we should use an
additional differentiation in \eqref{Lambda} to eliminate
$\boldsymbol\theta$ completely. For symmetry we take the three
quantities
$$
2\,\bigg(\frac{\dot\vartheta_2}{\vartheta_2},\:
\frac{\dot\vartheta_3}{\vartheta_3},\:
\frac{\dot\vartheta_4}{\vartheta_4}\bigg)\FED(X,Y,Z)
$$
and obtain at once the following equations:
\begin{equation}\label{darboux}
\dot X=(Y+Z)\,X-YZ\,,\qquad \dot
Y=(X+Z)\,Y-XZ\,,\qquad \dot Z=(X+Y)\,Z-XY\,.
\end{equation}
This is the famous Darboux--Halphen system \cite[{\bf
I}:~p.~331]{halphen} and one of its consequence is  equivalent to
Eq.~\eqref{halphen}. The scale change $X\mapsto\frac12\,X$ turns it
into equation
\begin{equation}\label{X3}
\dddot{X\,}\!\big(\dot X-X^2\big)
= \ddot X\,(\ddot X-4\,X^3) -2\,\dot X^2\big(\dot X-3\,X^2\big)\,.
\end{equation}
Applications of  system \eqref{darboux} are very well known. See, \eg,
works by Ablowitz et all \cite{ablowitz}, \cite[p.~577]{conte},
Takhtajan \cite{tah}, and Conte \cite[pp.~143, 147]{conte}. As a vacuum
cosmological model these equations come from a particular case of the
Bianchi-IX model \cite{chakr}.

Thus, renormalization of $\theta$-functions trivializes the integration
scheme of the defining \odes\  and clarifies interrelations between
differential properties of $\vartheta,\theta$-functions, the heat
equation, the Darboux--Halphen system, and their consequences. All the
equations are integrated in terms of canonical and non-canonical
theta-series.

\begin{remark}\label{R9}
We can continue renormalization
$\boldsymbol\theta\mapsto\tilde{\!\boldsymbol\theta}$ by putting\,
${\tilde{\!\boldsymbol{\theta}}}=\boldsymbol\theta\,
\exp\!\big({\frac12\Lambda z^2}\big)$. Then parameter $\Lambda$
disappears in Eqs.~\eqref{Xbold} as if we put $\Lambda=0$ there; the
integrability conditions then become the simple algebraic relations
$$
Y=X+2\,\pi^2\varkappa^2\,\vartheta_4^4(\btau)\,,\qquad
Z=X+2\,\pi^2\varkappa^2\,\vartheta_3^4(\btau)\,,
$$
and the only differential equation of  Riccati type  for function
$X(t)$ with variable coefficients $\varkappa=\varkappa(t)$,
$\btau=\btau(t)$:
$$
\dot X=X^2+2\,\pi^2\varkappa^2
\big\{\vartheta_3^4(\btau)+\vartheta_4^4(\btau)\big\}\cdot
X-4\,\pi^4\varkappa^4\,\vartheta_3^4(\btau)\,\vartheta_4^4(\btau)\,.
$$
This equation is  also integrable since it is a consequence of
equations \eqref{Lambda}--\eqref{X3}.
\end{remark}

\subsection{General integrals\label{S9.4}}

Insomuch as equations \eqref{darboux}, \eqref{halphen}, \eqref{X3}
serve both the canonical and non-canonical case, the general integral
of equations \eqref{intA} is a variation of solutions \eqref{JacChaz}.
Let us denote the four integration constants for the system
\eqref{intA} as $(a,b,c,d)$. Let us next put
$\T\DEF\frac{a\,\tau+b}{c\,\tau+\delta}$, where, as usual,
$a\,\delta-b\,c=1$, so $\delta$ is not a free parameter.

\begin{theorem}\label{T15}
General solution to non-canonical dynamical system \eqref{intA} is as
follows\,$:$
\begin{equation}\label{T}
\hspace{-1em}
\begin{array}{c}
\ds\vartheta_2= d\,\frac{\vartheta_2(\T)}
{\sqrt{c\,\tau+\delta\,}}\,,\qquad \vartheta_3=
\frac{1}{\bA}\,\frac{\vartheta_3
(\T)}{\sqrt{c\,\tau+\delta\,}}\,,\qquad \vartheta_4= \frac{1}{\bB}\,
\frac{\vartheta_4(\T)}{\sqrt{c\,\tau+\delta\,}}\,,\\[3ex]
\ds\eta=\frac{1}{(c\,\tau+\delta)^2}\bigg\{ \eta(\T)+
\frac{\pi^2}{12}\, \mbig[4][(1-\bA^{\sm4})
\, \vartheta_3^4(\T)+
(1-\bB^{\sm4})\, \vartheta_4^4(\T)\mbig[4]]\bigg\} +
\frac{1}{2}\,\frac{\pi\,\ri\,c}{c\,\tau+\delta}\,,
\end{array}  \hspace{-1.5em}
\end{equation}
where $\vartheta_k(\T)$ and $\eta(\T)$ are understood to be the
canonical $\vartheta,\eta$-series \eqref{consts} and \eqref{eta}.
\end{theorem}

\begin{proof} The straightforward verification. \end{proof}

As a corollary we found that the principal parameter of the theory
acquires the form
$$
\Lambda=\frac{4}{(c\,\tau+\delta)^2}\bigg\{ \eta({\T})+
\frac{\pi^2}{12} \mbig[4][\vartheta_3^4(\T)+
\vartheta_4^4(\T)\mbig[4]]\bigg\}+
\frac{2\,\pi\,\ri\,c}{c\,\tau+\delta}\,;
$$
it contains integration constants but does not contain the integrals
$\bA$, $\bB$. Therefore it does not depend on whether canonical or
non-canonical case is taken. It also satisfies the 3rd order equation
\eqref{halphen} wherein we should put $X=\frac{\ri}{4\pi}\Lambda$.

In addition to Remark~\ref{R7} we note that in spite of seeming
simplicity, the symmetrical system \eqref{var} is not amenable to
integration. Among other things, it is not a compatibility condition
for equations \eqref{X} and \eqref{Dtau} and thus may not be used as an
alternative to the correct and integrable system \eqref{intA}. The
algebraic integral for the system \eqref{var}, \ie, \eqref{A3}, is
\emph{nonlinear} in variables $\vartheta^4$.

In regard to Remark~\ref{R4}, we note that Jacobi's system \eqref{ABab}
is integrated in its full generality along with the system
\eqref{intA}. Computations show that Jacobi's $a(h)$ is
$$
a=2\,I\,\frac{\vartheta_2^4}{\vartheta_3^4}
\mbig[7](\Mfrac{\alpha\,h+\beta}{\gamma\,h+\delta}\mbig[7])+I
\qquad(\alpha\,\delta-\beta\,\gamma=1)
$$
and remaining functions $b(h)$, $A(h)$, and $B(h)$ are easily computed
from equations \eqref{ABab} by differentiating followed by trivial
simplification. We thus obtain the complete set of integration
constants for  Eqs.~\eqref{ABab}. All this material is discussed at
greater length in \cite{br3}.

Let us assume now that quantities $\eta$, $\vartheta$ in \eqref{X} are
functions of $\tau$ according to Eqs.~\eqref{intA} and integration
constants $\{A,B,C,\varkappa,\btau\}$ are unknown functions of $\tau$.
Substituting formulae \eqref{Xgeneral} into \eqref{Dtau}, we get a
system of \odes\  for these functions. The calculations can be reduced
in advance since we have already had two integrals \eqref{A12}
($\hhence$ \eqref{omega}):
\begin{equation}\label{intO}
\varkappa\,\frac{\vartheta_3^2(\btau)}{\vartheta_3^2}=\bA^2\,,
\qquad \varkappa\,\frac{\vartheta_4^2(\btau)}
{\vartheta_4^2}=\bB^2\,.
\end{equation}
These relations define a point transformation between algebraic form of
integrals  $(\bA,\bB)$ and their transcendental counterpart, \ie, the
pair $(\varkappa,\btau)$; the quantities $\vartheta_3$ and
$\vartheta_4$ are parameters here. Hence we can obtain the sought-for
equations/solutions  by a simpler way. Let the dot above a symbol
denote the derivative with respect to $\tau$. Then we derive:
$$
\dot\btau=\varkappa^2\,, \qquad
\pi\,\ri\,\frac{\dot{\!\varkappa\,}}{\varkappa}=2\,M\,,\qquad \dot
A=0\,,\qquad \dot B=A\,\dot{\!\varkappa\,}\,,\qquad \frac{\dot
C}{C}=-\frac{\dot{\!\varkappa\,}}{\varkappa}\,.
$$
The first two equations have the solution (implicit in $\btau$)
\begin{equation}\label{four}
\frac{\vartheta_3(\btau)}{\vartheta_4(\btau)}=
\boldsymbol{p}\,\bigg\{ \frac{\vartheta_3(\T)}
{\vartheta_4(\T)}\bigg\}^{\!\boldsymbol{q}},\qquad
\varkappa=\frac{\sqrt{\boldsymbol{q}}}{c\,\tau+\delta}\,
\frac{\vartheta_2^2(\T)}{\vartheta_2^2(\btau)}\,,
\end{equation}
where $\boldsymbol{p}$, $\boldsymbol{q}$ are new integration constants.
The remaining equations are easily integrated to give
$$
A=\bE\,,\qquad
B=\frac{\bE\,\sqrt{\boldsymbol{q}}}{c\,\tau+\delta}\,
\frac{\vartheta_2^2(\T)}{\vartheta_2^2(\btau)}
+\bD\,,\qquad
C=\frac{\bC}{\sqrt{\boldsymbol{q}}}\,
\frac{\vartheta_2^2(\btau)}{\vartheta_2^2(\T)}\,(c\,\tau+\delta)\,,
$$
where  $\bC$, $\bD$, $\bE$ are further integration constants. By virtue
of Eqs.~\eqref{intO} and \eqref{T} we must put
$\boldsymbol{p}=\boldsymbol{q}=1$. Such a reduction of the number of
integration constants is dictated by the fact that integrals
\eqref{intO} are integrals of both $z$- and $\tau$-equations and the
equations themselves are nonlinear. The first equation in \eqref{four}
turns into a relation between $\btau$ and $\T$. One can show that this
relation is controlled by the standard modular group
$\boldsymbol{\Gamma}(4)$ \cite{mumford}:
$$
\frac{\vartheta_3(\btau)}{\vartheta_4(\btau)}=
\frac{\vartheta_3(\T)} {\vartheta_4(\T)}\quad\hence\quad
\btau=\widehat{\boldsymbol{\Gamma}(4)}(\T)\,.
$$
We choose the simplest case $\btau(\tau)=\T$ and hence the functions
$A(\tau)$, $B(\tau)$, $C(\tau)$, and $\varkappa(\tau)$ are immediately
determined. We thus arrive at the ultimate answer.

\begin{theorem}\label{T16}
Let integrability conditions of Eqs.~\textup{\eqref{X}--\eqref{Dtau}}
or, equivalently,  Eqs.\ \textup{\eqref{x}--\eqref{tau}}, be given by
Eqs.~\eqref{intA} and their solution \eqref{T}. Then the general and
simultaneous integral of equations \eqref{X} and \eqref{Dtau} reads as
follows\/$:$
\begin{alignat*}{2}
\pm\,\theta_1={}&&\frac{1}{\bA\bB}\,
\frac{d\,\bC}{\sqrt{c\,\tau+\delta\,}}\:
\theta_1\Big(\Mfrac{z+\bE}{c\,\tau+\delta}+ \bD
\Big|\Mfrac{a\,\tau+b}{c\,\tau+\delta}
\Big)\,\re_{\mathstrut}^{\frac{\sm\pi\,\ri\, c}
{c\tau+\delta}(z+\bE)^2}\,,\\
\pm\,\theta_2={}&& \frac{d\,\bC}{\sqrt{c\,\tau+\delta\,}}\:
\theta_2\Big(\Mfrac{z+\bE}{c\,\tau+\delta}+ \bD
\Big|\Mfrac{a\,\tau+b}{c\,\tau+\delta}
\Big)\,\re_{\mathstrut}^{\frac{\sm\pi\,\ri\, c}
{c\tau+\delta}(z+\bE)^2}\,,\\
\pm\,\theta_3={}&&\frac{1}{\bA}\,
\frac{\phantom{d}\,\bC}{\sqrt{c\,\tau+\delta\,}}\:
\theta_3\Big(\Mfrac{z+\bE}{c\,\tau+\delta}+ \bD
\Big|\Mfrac{a\,\tau+b}{c\,\tau+\delta}
\Big)\,\re_{\mathstrut}^{\frac{\sm\pi\,\ri\, c}
{c\tau+\delta}(z+\bE)^2}\,,\\
\pm\,\theta_4={}&& \frac{1}{\bB}\,
\frac{\phantom{d\,}\bC}{\sqrt{c\,\tau+\delta\,}}\:
\theta_4\Big(\Mfrac{z+\bE}{c\,\tau+\delta}+ \bD
\Big|\Mfrac{a\,\tau+b}{c\,\tau+\delta}
\Big)\,\re_{\mathstrut}^{\frac{\sm\pi\,\ri\, c}
{c\tau+\delta}(z+\bE)^2}\,.
\end{alignat*}
Here, $a\,\delta-b\,c=1$ and formula for $\Dtheta$ is a $z$-derivative
of first of these formulae.
\end{theorem}

\begin{proof}
Straightforward calculation shows that these expressions do indeed
solve the systems \eqref{X} and \eqref{Dtau} under arbitrary constants
$\{\bA$, $\bB$, $\bC$, $\bD$, $\bE\}$. Coefficients $\eta$ and
$\vartheta$'s contain parameters $\{a,b,c,d\}$ which are free.
\end{proof}

Reduction to the canonical case \eqref{sol} is brought about by putting
$\bA=\bB=d=1$ and by choosing the transformation
$\left(\begin{smallmatrix}a&b\\c&\delta
\end{smallmatrix} \right)=\left(\begin{smallmatrix}1&0\\2&1
\end{smallmatrix} \right)\in  \boldsymbol{\Gamma}(2)$ since,
the group ${\boldsymbol{\Gamma}(2)}$ does not permute  functions
$\theta_k$ or $\vartheta_k$. Indeed, functions
$$
\frac{1}{\sqrt{c\,\tau+\delta\,}}\:
\theta_{k}\Big(\Mfrac{z+\bE}{c\,\tau+\delta}+ \bD
\Big|\Mfrac{a\,\tau+b}{c\,\tau+\delta}
\Big)\,\re_{\mathstrut}^{\frac{-\pi\,\ri\, c}
{c\tau+\delta}(z+\bE)^2}\,,
$$
under the parameters above, become
$$
\theta_k\sim\frac{1}{\sqrt{2\,\tau+1\,}}\:
\theta_{k}\Big(\Mfrac{z+\bE}{2\,\tau+1}+ \bD
\Big|\Mfrac{1\,\tau+0}{2\,\tau+1}
\Big)\,\re_{\mathstrut}^{\frac{-\pi\,\ri\, 2}
{2\tau+1}(z+\bE)^2}=\cdots\,,
$$
and, according to Corollary~\ref{C4},
$$
\begin{aligned}
\cdots&={}
\frac{\const}{\sqrt{2\,\tau{+}1\,}}\cdot\sqrt{2\,\tau{+}1}
\re_{\mathstrut}^{\frac{2\,\pi\,\ri}
{2\tau+1}\mbig[0]\{z+\bE+\bD(2\tau+1)\mbig[0]\}^{\!2}}
\theta_{k}\big(z+\bE+ \bD\,(2\tau{+}1)|\tau\big)\,
\re_{\mathstrut}^{\frac{-2\pi\,\ri}
{2\tau+1}(z+\bE)^2}\\
&=\const\cdot\theta_{k}\big(z+ 2\,\bD\,\tau+\bE+\bD|\tau\big)\,
\re_{\mathstrut}^{2\pi\ri\bD\,(2z+2\bD\,\tau)}\,,
\end{aligned}
$$
that is  the `linearly exponential' form \eqref{sol} under
$(A,B)=(2\,\bD, \bE+\bD)$.

\section{An application. The sixth Painlev\'e transcendent\label{S10}}

\noindent In addition to applications of the previous machinery
mentioned in \cite{br1,br3,br4}, in this section, we present briefly
one more and very nontrivial application  to the famous 6th Painlev\'e
equation \cite{conte}
\begin{equation}\label{P6}
\begin{aligned}
y_{\mathit{xx}}^{}&=\frac12\! \left(\frac1y+\frac{1}{y-1}+
\frac{1}{y-x}\right) y_x^2- \left(\frac1x+\frac{1}{x-1}+
\frac{1}{y-x}\right)y_x^{}\\
&\== +\frac{y(y-1)(y-x)}{x^2(x-1)^2}
\left\{\alpha-\beta\,\frac{x}{y^2}+\gamma\,\frac{x-1}{(y-1)^2}-
\Big(\delta-\Mfrac12\Big)\,\frac{x(x-1)}{(y-x)^2} \right\}.
\end{aligned}
\end{equation}
A deep connection of this equation with elliptic functions was
established by Painlev\'e himself in work \cite{painleve} wherein he
gave a remarkable form to Eq.~\eqref{P6}:
\begin{equation}\label{P6wp}
-\frac{\pi^2}{4}\,\frac{d^2
\boldsymbol{z}}{d\tau^2}=\alpha\,\wp'(\boldsymbol{z}|\tau)+
\beta\,\wp'(\boldsymbol{z}-1|\tau)+
\gamma\,\wp'(\boldsymbol{z}-\tau|\tau)+
\delta\,\wp'(\boldsymbol{z}-1-\tau|\tau)
\end{equation}
by performing the transcendental change of variables
$(y,x)\rightleftarrows(\boldsymbol{z},\tau)$:
\begin{equation}\label{subs}
x=\frac{\vartheta_4^4(\tau)}{\vartheta_3^4(\tau)}\,,\qquad
y=\frac13+\frac13\frac{\vartheta_4^4(\tau)}{\vartheta_3^4(\tau)}
-\frac{4}{\pi^2}
\frac{\wp(\boldsymbol{z}|\tau)}{\vartheta_3^4(\tau)}\,.
\end{equation}

In 1995 Hitchin \cite{hitchin} found a beautiful solution to
Eq.~\eqref{P6} which hitherto remains the most nontrivial of all those
currently known. It corresponds to parameters
$\alpha=\beta=\gamma=\delta=\frac18$ and reads parametrically as
follows \cite[pp.~74, 78]{hitchin}:
\begin{equation}\label{wpz}
\wp(\boldsymbol{z}|\tau)=\wp(A\tau+B|\tau)+\frac12\,
\frac{\wp'(A\tau+B|\tau)}
{\zeta(A\tau+B|\tau)-(A\tau+B)\,\eta(\tau)+\ri\frac\pi2 A}\,,
\end{equation}
where $A$, $B$ are free constants. It is not difficult to show that in
the case of Hitchin's parameters $\{\alpha$, $\beta$, $\gamma$,
$\delta\}$ equation \eqref{P6wp}  can be written in the
$\theta$-function form
$$
-\pi^2\,\frac{d^2
\boldsymbol{z}}{d\tau^2}=4\,\wp'(2\boldsymbol{z}|\tau) \qquad
\hhence\qquad\frac{d^2
\boldsymbol{z}}{d\tau^2}=4\pi\,\ded^9(\tau)\,
\frac{\theta_1^{}(2\boldsymbol{z}|\tau)}
{\theta_1^4(\boldsymbol{z}|\tau)}\,.
$$

With complete rules of differential $\theta$-computations in hand, we
can obtain their analog for Weierstrassian functions and thereby
automate and simplify manipulation with all solutions to Eq.~\eqref{P6}
expressible in terms of elliptic or $\theta$-functions.

\subsection{Weierstrass' functions and Hitchin's solution}

In order to derive rules for derivatives of Weierstrassian functions
with respect to half-periods $(\om,\om')$ one can use the preceding
$\theta$-apparatus supplemented with rules \eqref{st},
\eqref{wp1}--\eqref{wp2}, or, alternatively, transformations between
derivatives in $(\g2,\g3)$ and $(\om,\om')$
\cite[pp.~263--265]{forsyth}. The $(\g2,\g3)$-derivatives were
considered by Weierstrass \cite[{\bf V}]{we}, \cite[{\bf I}]{halphen},
\cite[\bf IV]{tannery} and, at about the same time, by Frobenius and
Stickelberger \cite{frobenius}. Let us denote a sign of the period
ratio as $\boldsymbol{\mathfrak{s}}\DEF
\mathfrak{sign}\mbig\{\boldsymbol\Im
\big(\frac{\omega\s'}{\om}\big)\mbig\}$. Applying now any of the
techniques above and carrying out  some simplification, we obtain  that
the sought-for formulae acquire very compact and symmetrical form (not
presented in the literature).

\begin{theorem}\label{T17}
Rules for differentiating Weierstrass'
$(\sigma,\zeta,\wp,\wp')$-functions are
\begin{align*}
& \left\{
\begin{aligned}
\boldsymbol{\mathfrak{s}} \,\frac{\partial\sigma}{\partial\om}&=
-\frac{\ri}{\pi}\, \Big\{
\om'\mbig(\wp-\zeta^2-\mfrac{1}{12}\,\g2\,z^2 \mbig)
+2\,\etap\,(z\,\zeta-1)\Big\}\,\sigma\\
\boldsymbol{\mathfrak{s}}\, \frac{\partial\sigma}{\partial\om'}&=
\phantom{-}\frac{\ri}{\pi}\, \Big\{
\om\mbig(\wp-\zeta^2-\mfrac{1}{12}\,\g2\,z^2 \mbig)
+2\,\eta\,(z\,\zeta-1)\Big\}\,\sigma\,,
\end{aligned}
\right.\\
& \left\{
\begin{aligned}
\boldsymbol{\mathfrak{s}}\,\frac{\partial\zeta}{\partial\om}&=
-\frac{\ri}{\pi}\, \Big\{
2\,(\om'\zeta-z\,\etap)\,\wp+\om'\,\mbig(\wp'-\mfrac
16\,\g2\,z\mbig) +2\,\etap\zeta\Big\} \\
\boldsymbol{\mathfrak{s}}\,
\frac{\partial\zeta}{\partial\om'}&=\phantom{-} \frac{\ri}{\pi}\,
\Big\{ 2\,(\om\,\zeta-z\,\eta)\,\wp+\om\,\mbig(\wp'-\mfrac
16\,\g2\,z\mbig) +2\,\eta\,\zeta\Big\}\,,
\end{aligned}
\right.\\
& \left\{
\begin{aligned}
\boldsymbol{\mathfrak{s}}\,
\frac{\partial\wp}{\partial\om}&=\phantom{-} \frac{\ri}{\pi}\,
\Big\{ 2\,(\om'\zeta-z\,\etap)\,\wp'
+4\,(\om'\wp-\etap)\,\wp-\Mfrac23\,\om' \g2\Big\} \\
\boldsymbol{\mathfrak{s}}\, \frac{\partial\wp}{\partial\om'}&=
-\frac{\ri}{\pi}\, \Big\{ 2\,(\om\,\zeta-z\,\eta)\,\wp'
+4\,(\om\,\wp-\eta)\,\wp-\Mfrac23\,\om\, \g2\Big\}\,,
\end{aligned}
\right.\\
& \left\{
\begin{aligned}
\boldsymbol{\mathfrak{s}}\,
\frac{\partial\wp'}{\partial\om}&=\phantom{-} \frac{\ri}{\pi}\,
\mbig\{ 6\,(\om'\wp-\etap)\,\wp'
+(\om'\zeta-z\,\etap)(12\,\wp^2-\g2) \mbig\} \\
\boldsymbol{\mathfrak{s}}\,
\frac{\partial\wp'}{\partial\om'}&=-\frac{\ri}{\pi}\, \mbig\{
6\,(\om\,\wp-\eta)\,\wp' +(\om\,\zeta-z\,\eta)(12\,\wp^2-\g2)
\mbig\}\,.
\end{aligned}
\right.
\end{align*}
\end{theorem}

Setting in these equations $\om=1$, $\om'=\tau$, and
$\boldsymbol{\mathfrak{s}}=1$, we arrive at a dynamical system
containing  parameter $z$:
\begin{equation}\label{szwp}
\left\{\mspace{-15mu}
\begin{array}{l}
\ds\mspace{8.5mu}\frac{\partial\sigma}{\partial\tau}=
\phantom{-}\frac{\ri}{\pi} \Big\{ \wp-\zeta^2
+2\,\eta\,(z\,\zeta-1)-\Mfrac{1}{12}\,\g2\,z^2\Big\}\,\sigma
\\[2ex]
\left.
\begin{aligned}
\frac{\partial\zeta}{\partial\tau}&= \phantom{-}
\frac{\ri}{\pi} \Big\{ \wp'+2\,(\zeta-z\,\eta)\,\wp+ 2\,\eta\,\zeta
-\Mfrac16\,\g2\,z\Big\}\\
\frac{\partial\wp}{\partial\tau}&=-\frac{\ri}{\pi}
\Big\{ 2\,(\zeta-z\,\eta)\,\wp'+4\,(\wp-\eta)\,\wp
-\Mfrac23\,\g2\Big\}\\
\frac{\partial\wp'}{\partial\tau}&= -\frac{\ri}{\pi}\, \mbig\{
6\,(\wp-\eta)\,\wp'+(\zeta-z\,\eta)(12\,\wp^2-\g2) \mbig\}
\hspace{1em}\,,\hspace{-1.5em}
\end{aligned}\right\}
\end{array}
\right.
\end{equation}
where we use the right brace additionally to denote  the differential
closedness of the functions $(\zeta,\wp,\wp')$. It follows that the
triple of functions $\zeta(z|\tau)$, $\wp(z|\tau)$, and $\wp'(z|\tau)$
is differentially closed with respect to both variables $z$ and $\tau$.
We have also to close the derivatives of coefficients $\g2$ and $\eta$.
This is realized by the Halphen system \eqref{g2g3} but third variable
$\g3$ is not present in system \eqref{szwp} or in the system of
$z$-equations
\begin{equation}\label{weier}
\frac{d\zeta}{dz}=-\wp\,,\qquad \frac{d\wp}{dz}=\wp'\,,\qquad
\frac{d\wp'}{dz}=6\,\wp^2-\frac12\,\g2\,.
\end{equation}
Therefore we may treat the classical Weierstrassian relation between
$\wp$ and $\wp'$
$$
\g3(\wp,\wp')=4\,\wp^3-\g2\,\wp-(\wp')\ts{}^2
$$
as the \emph{algebraic $($polynomial\/$)$ integral} of equations
\eqref{weier} or as the \emph{surface of a constant level} for the
dynamical system with variable coefficients, \ie,  system \eqref{szwp}.

Another corollary of this system  is the fact that each of functions
$\zeta$, $\wp$, $\wp'$ satisfies the ordinary $\tau$-differential
equation of 2nd order with variable coefficients $\g2$, $\g3$, $\eta$,
and function $\sigma$ does an equation of 3rd order. These equations
are too large to display here.  The function
$\boldsymbol{Z}=\zeta(z|\tau)-z\,\eta(\tau)$, as an example, solves an
equation obtainable by elimination  of variable $\wp$  from the two
polynomials that follow from \eqref{szwp}:
$$
\big\{\pi\,\ri\, \boldsymbol{Z}_\tau+2\,(\wp+\eta)\,\boldsymbol{Z}
\big\}^2-4\,\wp^3+\g2\,\wp+\g3\,,
$$
$$
\frac{\pi^2}{8}\,\frac{\boldsymbol{Z}_{\tau\tau}}{\boldsymbol{Z}}
+\ri\,\frac{\pi}{2}\,\big(\boldsymbol{Z}^2+\wp-2\,\eta\big)\,
\frac{\boldsymbol{Z}_\tau}{\boldsymbol{Z}}+
(\wp+\eta)\,\boldsymbol{Z}^2-\wp^2+\eta\,\wp-\eta^2+\frac14\,\g2\,;
$$
these polynomials are understood to be equal to zero. They do not
explicitly contain the variable $z$. We do not consider here the
$\theta,\theta'$-analogs of these equations.

We observe that Hitchin's solution is a function of the quantities
$\zeta(A\tau+B|\tau)$, $\wp(A\tau+B|\tau)$, and $\wp'(A\tau+B|\tau)$.
Thus, these three functions define a dynamical system in its own right
that follows from equations \eqref{szwp}--\eqref{weier}.

\newcommand{\Z}{\Larger{\boldsymbol{\zeta}}}
\newcommand{\WP}{\Larger{\boldsymbol{\wp}}}
\newcommand{\WPP}{\Larger{\boldsymbol{\wp'}}}

\begin{proposition}\label{P3}
The Hitchin case of Painlev\'e equation \eqref{P6} is equivalent to the
dynamical system
\begin{equation}\label{canon}
\frac{d\Z}{d\tau}=\frac{\ri}{\pi} \big\{ \WPP+
2\,(\WP+\eta)\,\Z\big\}\,,\qquad \frac{d\WP}{d\tau}=-\frac{\ri}{\pi}
\Big\{ 2\,\Z\,\WPP+4\,(\WP-\eta)\,\WP -\Mfrac23\,\g2\Big\}
\end{equation}
with variable coefficients $\eta=\eta(\tau)$, $\g2=\g2(\tau)$,
$\g3=\g3(\tau)$. Here,
$$
\WPP\DEF\sqrt{4\WP^3-\g2\WP-\g3}
$$
and equations \eqref{canon} may be supplemented by their corollary
$$
\frac{d\WPP}{d\tau}\equiv -\frac{\ri}{\pi}\, \big\{
6\,(\WP-\eta)\,\WPP+(12\,\WP^2-\g2)\,\Z \big\}\,.
$$
The general integral of the equations reads
\begin{equation}\label{ZP}
\Z=\zeta(A\tau+B|\tau)-A\etap(\tau)-B\eta(\tau)\,,\qquad
\WP=\wp(A\tau+B|\tau)\,.
\end{equation}

\end{proposition}

In other words, we may view the sixth Painlev\'e transcendent under
Hitchin's parameters as a pair of `$\tau$-equations' \eqref{canon},
\ie, as an integrable case of the $\tau$-representation  \eqref{P6wp}.

One further consequence of the results above is that we can construct
the Painlev\'e--Hitchin integral $\tau$-calculus. Indeed, the first of
equations \eqref{szwp} suggests that the quantity
$\wp-(\zeta-z\,\eta)^2$ is  integrable  with respect to $\tau$ and
expressible  through a logarithm of the $\sigma$-function. Namely,
\begin{equation}\label{tmp}
-\pi\,\ri\,\frac{d}{d\tau}\!\ln\sigma=\wp-(\zeta-z\,\eta)^2+
\Big(\eta^2-\Mfrac{1}{12}\,\g2\Big)\, z^2-2\,\eta\,.
\end{equation}
Owing to equations \eqref{canon} and the fact that Hitchin's solution
have a $z$-argument in form $z=A\,\tau+B$, we can extend \eqref{tmp}
and consider the objects \eqref{ZP}, \ie, $\WP-\Z^2$. Therefore
$$
-\pi\,\ri\,\frac{d}{d\tau}\!\ln\sigma(A\,\tau+B|\tau)=
\WP-\Z^2+f(\tau)\,,
$$
where unknown function $f(\tau)$ is independent of $\Z$ and $\WP$. It
is readily determined by a differentiation  followed by use of
differential connection \eqref{ded} between $\ded$ and $\eta$. Applying
an antiderivative to the last equation, we obtain the nice indefinite
integral
$$
\frac{\ri}{\pi}\int\limits^{\;\,\tau}\!\!\big(\WP-\Z^2
\big)d\tau=\ln \theta_1\!\big({\tfrac12}A\tau+\tfrac
12B\big|\tau\big)- \ln\ded(\tau)+\frac\ri4\,\pi A^2\tau\,;
$$
it can be checked by a straightforward differentiation and conversion
everything to $\vartheta,\theta$-functions. This integral is nothing
but the Painlev\'e $\tau$-analog of Weierstrassian relation
$\tsint\zeta\,dz=\ln\sigma$ between meromorphic $\zeta$-function and
entire function $\sigma$.

\subsection{$\theta$-function forms of Hitchin's solution}

In addition to solution \eqref{wpz}, Hitchin suggested also  its
$\theta$-function form. We reproduce it here in original notation
\cite[p.~33]{hitchin}:
\begin{equation*}\label{morazm}
\begin{aligned}
y(x)&=\frac{\vartheta_1'''(0)}
{3\,\pi^2\vartheta_4^4(0)\,\vartheta_1'(0)}
+\frac13\,\bigg(1+\frac{\vartheta_3^4(0)}{\vartheta_4^4(0)}
\bigg)\\
&\== +\frac{\vartheta_1'''(\nu)\,\vartheta_{1\!}^{}(\nu)-
2\,\vartheta_1''(\nu)\,\vartheta_1'(\nu)
+4\,\pi\,\ri\,c_1^{}\,(\vartheta_1''(\nu)\,\vartheta(\nu)-
\vartheta_1'^2(\nu)
)}{2\,\pi^2\,\vartheta_4^4(0)\,\vartheta_{1\!}^{}(\nu)\,
(\vartheta_{1\!}'(\nu)+
2\,\pi\,\ri\,c_1^{}\,\vartheta_{1\!}^{}(\nu))}\,,
\end{aligned}
\end{equation*}
where $\nu=c_1^{}\tau+c_2^{}$. Differential properties of
$\theta$-functions or conversion formulae like \eqref{wp1}--\eqref{wp2}
suggest that the availability of the higher $\theta$-derivatives  is
excessive here and we can simplify this solution. Doing this, we obtain
the very simple formula
\begin{equation*}\label{simp}
y=\frac{\sqrt{x}}{\theta_1^2}
\bigg\{\frac{\pi\,\vartheta_2^2\cdot\theta_2\,\theta_3\,\theta_4}
{\Dtheta+2\,\pi A\,\theta_1}-\theta_2^2\bigg\},
\end{equation*}
where symbols $\vartheta$, $\Dtheta$, $\theta$ are understood to be
equal to
$$
\Dtheta=\Dtheta\Big(\!A\mfrac{\ellK(\sqrt{x})}{\ellK'(\sqrt{x})}+
B\Big|\mfrac{\ri\,
\ellK(\sqrt{x})}{\ellK'(\sqrt{x})}\Big),\qquad
\theta_k=\theta_k\Big(\!A\mfrac{\ellK(\sqrt{x})}{\ellK'(\sqrt{x})}+
B\Big|\mfrac{\ri\,
\ellK(\sqrt{x})}{\ellK'(\sqrt{x})}\Big)\,,\qquad
\vartheta_2=\vartheta_2\Big(\mfrac{\ri\,
\ellK(\sqrt{x})}{\ellK'(\sqrt{x})}\Big).
$$
Elliptic integrals $\ellK$ and $\ellK'$ were introduced above. They
give an inversion of the first formula in \eqref{subs}, that is an
equivalent of Jacobi's formula \eqref{KK'}:
$$
\tau=\ri\,\frac{\ellK(\sqrt{x})}{{\ds\ellK'}(\sqrt{x})}\,.
$$

Further reading of Hitchin' solution is related to the fact that all
solutions to the Painlev\'e equations are meromorphic functions with
fixed branch points. For  equation \eqref{P6} these are the three
points $x=\{0,1,\infty \}$ and the general theory of Painlev\'e
equations guaranties availability of what is called the
$\Larger{\btau}$-representation \cite[p.~165]{conte}
\begin{equation}\label{TAU}
y\sim x\,(1-x)\,\frac{d}{dx}\!
\ln\frac{\Larger{\btau}_{\!\!\!1}^{}}{\Larger{\btau}_{\!\!\!2}^{}}
\end{equation}
Here, the `bold tau' is a traditional tau-function notation having
nothing in common with modulus $\tau$ or modulus $\btau$ in
sects.~\ref{S7.2} and \ref{S9}. Equations \eqref{szwp} have the
consequence
$$
\frac{\pi}{2\,\ri}\,\frac{d}{d\tau}
\!\ln\big\{\zeta(z|\tau)-z\,\eta(\tau)\big\}= \wp(z|\tau)+\frac12\,
\frac{\wp'(z|\tau)}{\zeta(z|\tau)-z\,\eta(\tau)}+\eta(\tau)\,.
$$
Comparing this property with \eqref{wpz}, we observe the total
logarithmic derivative
\begin{equation*}\label{lnwp}
\wp(z|\tau)=\frac{\pi}{2\,\ri}\,\frac{d}{d\tau}\!\ln
\frac{\zeta(A\tau+B|\tau)-A\,\etap(\tau)-B\,\eta(\tau)}
{\ded^2(\tau)}\,.
\end{equation*}
Replacing $(A,B)$ with $(2A,2B)$ and transforming the right-hand side
of this equation into the $\theta$-functions, we can rewrite the
previous parametric form of the solution as follows:
$$
y=\frac{2\,\ri}{\pi}\frac{1}{\vartheta_3^4(\tau)}\, \frac{d}{d\tau}
\!\ln\frac{\Dtheta(A\tau+B|\tau)+2\,\pi\,\ri\,A\,
\theta_1(A\tau+B|\tau)}
{\vartheta_2^2(\tau)\,\theta_1(A\tau+B|\tau)}\,, \qquad
x=\frac{\vartheta_4^4(\tau)}{\vartheta_3^4(\tau)}\,.
$$
Conversion of this result into the original $x$-representation
\eqref{TAU} becomes now an exercise because all the differential
calculus has been described completely.

\begin{proposition}\label{P4}
The general solution to the Hitchin's case of Painlev\'e equation
\eqref{P6} in the tau-function form \eqref{TAU} is as follows\/$:$
$$
\begin{aligned}
y&=x\,(1-x)\,\frac{d}{dx}\!\ln \frac{\left\{\Dtheta\!\!\left(\!
A\frac{\ellK(\sqrt{x})}{\ellK'(\sqrt{x})}+B\Big|\frac{\ri\,
\ellK(\sqrt{x})}{\ellK'(\sqrt{x})}\right)+ 2\,\pi
A\cdot\theta_1\!\!\left(\!
A\frac{\ellK(\sqrt{x})}{\ellK'(\sqrt{x})}+B\Big|\frac{\ri\,
\ellK(\sqrt{x})}{\ellK'(\sqrt{x})}\right) \right\}^{\!2}} {(1-x)\,
\theta_1^2
\!\left(\! A\frac{\ellK(\sqrt{x})}{\ellK'(\sqrt{x})}+B\Big|\frac{\ri\,
\ellK(\sqrt{x})}{\ellK'(\sqrt{x})}\right) \ellK'^2(\sqrt{x})}
\\
&= \frac{\ellE'(\sqrt{x})}{\ellK'(\sqrt{x})}+2\,x\,(1-x)\,
\frac{d}{dx}\!\ln\!\bigg\{\frac{\Dtheta}{\theta_1}\!\!
\left(\! \textstyle
A\frac{\ellK(\sqrt{x})}{\ellK'(\sqrt{x})}+B\Big|\frac{\ri\,
\ellK(\sqrt{x})}{\ellK'(\sqrt{x})}\right)+2\,\pi A\bigg\},
\end{aligned}
$$
where Legendre's elliptic integrals $(\ellK,\ellK',\ellE,\ellE')$
\textup{\cite{bateman,WW}}, as functions of $\sqrt{x}$, are
differentially closed according to the rules
$$
\begin{aligned}
2\,\frac{d\ellK}{dx} &=\frac{\ellE}{x\,(1-x)}- \frac{\ellK}{x}\,,
&\qquad
2\,\frac{d\ds\ellK'}{dx}&=\frac{\ds\ellE'}{x\,(x-1)}-
\frac{\ds\ellK'}{x-1}\,,\\
2\,\frac{d\ellE}{dx} &=\frac{\ellE}{x}- \frac{\ellK}{x}\,, &
2\,\frac{\ds d\ellE'}{dx}&=\frac{\ds\ellE'}{x-1}-
\frac{\ds\ellK'}{x-1}\,.
\end{aligned}
$$
\end{proposition}

Verifying this form of solution by a direct substitution in \eqref{P6}
is a good and rather nontrivial exercise. We do not go into further
details because comprehensive analysis of this {\large$\btau$}-function
form, including additional motivation, explanations, and corollaries,
have been detailed in \cite{br2}, where the complete reference list can
also be found.

\end{document}